\documentclass[11pt, a4paper, final]{amsart}
\usepackage[T1]{fontenc}
\usepackage{lmodern}
\usepackage{mathtools}
\usepackage[utf8]{inputenc}
\usepackage{graphicx}

\usepackage{amsfonts}
\usepackage{amsmath}
\usepackage{amssymb}
\usepackage{amsthm}
\usepackage{mathtools}
\usepackage{paralist, tabularx}
\usepackage{enumitem}
\usepackage[utf8]{inputenc}
\usepackage{mathabx,epsfig}
\usepackage{tikz-cd}
\usepackage{verbatim}

\usepackage{etoolbox}

\usepackage{xcolor}
\definecolor{green}{RGB}{0,127,0}
\definecolor{redd}{RGB}{191,0,0}
\definecolor{red}{RGB}{105,89,205}
\usepackage[colorlinks=true]{hyperref}

\usepackage[notref, notcite]{showkeys}
\usepackage[cmtip,arrow]{xy}

\DeclareMathOperator{\id}{id}

\newcommand{\RR}{{\bar{R}}}

\newcommand{\C}{\mathfrak{C}}

\newcommand{\nref}[2]{\hyperref[#1]{\ref*{#1}$_{#2}$}}

\DeclareMathOperator{\im}{{Im}}

\DeclareMathOperator{\Th}{{Th}}

\DeclareMathOperator{\tp}{{tp}}
\DeclareMathOperator{\cl}{{cl}}

\DeclareMathOperator{\Span}{{Span}}

\newtheorem{theorem}{Theorem}
\numberwithin{theorem}{section}
\newtheorem{lemma}[theorem]{Lemma}

\newtheorem{fact}[theorem]{Fact}
\newtheorem{proposition}[theorem]{Proposition}
\newtheorem{problem}[theorem]{Problem}
\newtheorem{conjecture}[theorem]{Conjecture}

\newtheorem{question}[theorem]{Question}
\newtheorem{corollary}[theorem]{Corollary}

\newtheorem{clm}{Claim}
\newtheorem*{clm*}{Claim}

\theoremstyle{definition}
\newtheorem{definition}[theorem]{Definition}

\newtheorem{example}[theorem]{Example}

\theoremstyle{remark}
\newtheorem{remark}[theorem]{Remark}

\newtheorem{innercustomgeneric}{\customgenericname}
\providecommand{\customgenericname}{}
\newcommand{\newcustomtheorem}[2]{%
  \newenvironment{#1}[1]
  {%
   \renewcommand\customgenericname{#2}%
   \renewcommand\theinnercustomgeneric{##1}%
   \innercustomgeneric
  }
  {\endinnercustomgeneric}
}

\newcustomtheorem{customthm}{Theorem}

\AtEndEnvironment{proof}{\setcounter{clm}{0}}
\newenvironment{clmproof}[1][\proofname]{\proof[#1]}{\endproof}

\usepackage[hyphenbreaks]{breakurl}

\newcommand{\orcidlogo}{\includegraphics[height=\fontcharht\font`\B]{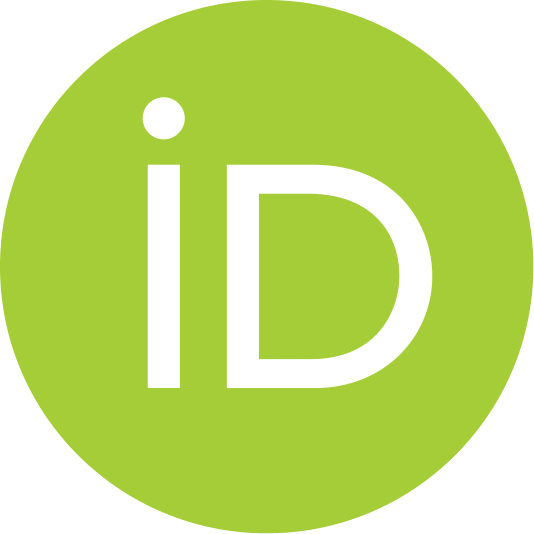}}
\newcommand{\orcid}[1]{\href{#1}{\orcidlogo #1}}

\usepackage[
backend=bibtex,
url=false,
isbn=false,
backref=true,
citestyle=alphabetic,
bibstyle=alphabetic,
autocite=inline,
maxnames=99,
minalphanames=4,
maxalphanames=4,
sorting=nyt,
]{biblatex}

\title{On the structure of approximate rings}

\author{Krzysztof Krupi\'{n}ski and Simon Machado}
\address{Krzysztof Krupi\'{n}ski \orcid{https://orcid.org/0000-0002-2243-4411}\\
Instytut Matematyczny, Uniwersytet Wroc{\l}awski\\
pl. Grunwaldzki 2, 50-384 Wroc{\l}aw, Poland}
\email{Krzysztof.Krupinski@math.uni.wroc.pl}

\address{Simon Machado \orcid{https://orcid.org/0000-0002-1787-6864}\\
ETH Zurich, Ramistrasse 101, Zurich, Switzerland}
\email{smachado@ethz.ch}


\begin{document}
	
	\begin{abstract}
       By a [$K$-]approximate subring of a ring we mean an additively symmetric subset $X$ such that $X \cdot X \cup (X + X)$ is covered by finitely many [resp.\ $K$] additive translates of $X$.

We prove a structure theorem for finite approximate subrings. Our aim is to develop a general framework for the sum–product phenomenon that applies uniformly across arbitrary rings. The main result identifies nilpotent  quotients as the fundamental obstruction to growth under both addition and multiplication. Another application of the main structure theorem is a ring-theoretic counterpart of Gromov's theorem on groups of polynomial growth.

The principal tool in the proof is the existence of definable locally compact models for arbitrary approximate subrings  from \cite{Kru}.

This existence theorem extends beyond the finite (and pseudofinite) setting. To illustrate the scope of the method, we also establish a structure theorem for uniformly discrete approximate subrings of semi-simple real algebras, generalizing a classical sum-product result of Meyer.
    \end{abstract}

	\maketitle
	\tableofcontents

\section{Introduction}
The sum-product phenomenon, first studied by Erd\H{o}s and Szemer\'edi \cite{ErSz}, is a fundamental result in additive combinatorics. It expresses the principle that the additive and multiplicative structures of the ring of integers, $\mathbb{Z}$, are incompatible: there is a constant $\epsilon > 0$ such that for any finite subset $A \subseteq \mathbb{Z}$, either the sumset $A+A = \{a+b : a,b \in A\}$ or the product set $A\cdot A = \{ab : a,b \in A\}$ must be large, specifically of size at least $|A|^{1+\epsilon}$. Since its introduction, this phenomenon has found numerous striking applications. For instance, it played a crucial role in results by Helfgott \cite{Helfgott} and Bourgain--Gamburd \cite{BourgainGamburd}, where establishing the sum-product phenomenon in various rings, including finite ones, was pivotal in order to prove spectral gap for Markov operators. It also found striking applications in number theory via exponential sums \cite{Bourgain2005expsum}.

Foundational works by Bourgain--Katz--Tao \cite{MR2053599}, and Tao \cite{Tao2} have further explored how the phenomenon manifests in rings satisfying a ``few zero divisors'' assumption (see also the earlier work of \cite{Chang2007MatrixSpace} regarding matrix rings). These works distinguish between rings with few zero divisors, which are amenable to elementary methods, and rings with many scales (such as $\mathbb{Z}/p^n\mathbb{Z}$), which require advanced techniques like Bourgain's multiscale analysis.

In this paper, we argue that the key distinction lies not in the presence or absence of zero divisors, but rather in the structural properties of nilpotency. We show that the sum-product phenomenon is fundamentally tied to the nilpotency of certain substructures within the ring, offering a unified approach to understanding this phenomenon across \emph{all} rings.

We present here a simplified version (Corollary \ref{corollary: sum-product phenomenon with epsilon}) of our 
sum-product phenomenon; for the full statement, see Theorem \ref{theorem: sum-product phenomenon with any g}. For a subset $Z$ of a ring $R$ and $n \in \mathbb{N}_{>0}$, by $nZ$ we denote the set of sums of $n$ elements of $Z$.

\begin{theorem}[Sum-product phenomenon in all rings]\label{theorem: concrete sum-product}
Let $\epsilon > 0$. There exists a non-decreasing unbounded function $f\colon \mathbb{N} \rightarrow \mathbb{N}$ such that the following holds.

Let $R$ be a ring and $X \subseteq R$ be a finite subset. Define the set $X' := 4(X-X) + (X-X) \cdot 4(X-X)$. Then:
\begin{itemize}
    \item either $|X+X+X\cdot X| \geq f(|X|) |X|$,
    \item or there exists a subring $R' \subseteq R$ and an ideal $I \subseteq R' \cap X'$ such that $R'/I$ is nilpotent and $|X' \cap R'| \geq |X|^{1-\epsilon}$.
\end{itemize}
\end{theorem}

While this result provides a unified framework, it comes at the cost of quantitative precision. Specifically, the function $f$ is unspecified. Unlike the classical sum-product phenomenon in $\mathbb{Z}$, where the growth is known to be a power function, the precise growth rate here remains unknown. This limitation reflects the nature of our methods, which rely on ideas developed in the proof of the structure theorem for finite approximate subgroups by Breuillard, Green, and Tao (henceforth the BGT theorem) \cite{BGT}.

A more precise statement, relating structural data to the ratio $\frac{|X+X+XX|}{|X|}$, requires the notion of \emph{approximate rings}. Recently introduced by the first author in \cite{Kru}, this concept captures the lack of growth under addition and multiplication quantitatively. A subset $X$ of a ring is called a \emph{$K$-approximate subring} if it is additively symmetric (i.e., $0 \in X$ and $X=-X$) and the set $X\cdot X \cup (X +X)$ can be covered by $K$ additive translates of $X$. Variants of this definition have appeared previously in the literature, often in finitary contexts (e.g., \cite{Kow, Bre}).

It is straightforward to verify that a \emph{finite} $K$-approximate subring satisfies $|X+X+XX| \leq K^3 |X|$. Conversely, a covering argument shows that if a subset $X$ satisfies $\frac{|X+X+XX|}{|X|} \leq K$, then the difference set $X-X$ is a $K^{O(1)}$-approximate subring (see Remark \ref{remark: small n-pling for rings}). Similarly, if the ambient ring is unital and $1 \in X$ with $|XX-XX| \leq K|X|$, then $X-X$ forms a $K^{O(1)}$-approximate subring (see \cite[Lemma 3.3.1]{Kow}).

Let $X$ be an approximate subring. We define recursively 
$$X_0=X \textrm{ and } X_{n+1}:=X_n +X_n + X_n \cdot X_n.\footnote{This definition and notation will be used throughout the paper.}$$\label{page: definition of  X_n}
Then the union $\bigcup_{n \in \mathbb{N}} X_n$ coincides with the subring $\langle X \rangle$ generated by $X$. With this notion at hand, we can state our main structural result for \emph{finite} approximate subrings (the full statement is preceded by a direct corollary which may be easier to grasp at first glance).

\begin{theorem}\label{theorem: Structure finite app rings}
    Let $K \geq 0$ and $X \subseteq R$ be a finite $K$-approximate subring. There exists a subring $R' \subseteq R$ and an ideal $I \subseteq R'$ such that:
    \begin{enumerate}
        \item $X$ is covered by $O_K(1)$ additive cosets of $R'$;
        \item $I \subseteq X_3$;
        \item $R'/I$ is nilpotent of class $O_K(1)$.
    \end{enumerate}
\end{theorem}

\begin{theorem}[Structure of finite approximate subrings]\label{theorem: main theorem intro}
For any $K \in \mathbb{N}$, there exist constants $N_1(K), \dots, N_4(K) \in \mathbb{N}$ such that for every finite $K$-approximate subring $X$, there exists an $N_1(K)$-approximate subring $Y \subseteq 4X + X \cdot 4X$ which is $N_2(K)$-commensurable to $X$, and an ideal $I \lhd \langle Y \rangle$ contained in $Y_{N_3(K)}$ such that $\langle Y \rangle/I$ is nilpotent of class at most $N_4(K)$.

Furthermore, we establish two regimes for these constants:
\begin{itemize}
    \item One may choose $N_3(K)=1$ (providing strong control on the ideal depth) and $N_1(K)=K^{510}+K^{22}$;
    \item Alternatively, one may choose $N_4(K)=\lfloor 4 \log_2(K) \rfloor$ (providing logarithmic control on the nilpotency class) and $N_1(K)=K^{510}+K^{22}$.
\end{itemize}
\end{theorem}

A direct application of Theorem \ref{theorem: main theorem intro} is the aforementioned sum-product phenomenon. Another application is a ring-theoretic analogue of Gromov's theorem on groups of polynomial growth (see Definition \ref{definition: ring of polynomial growth}). We state a simplified version here; the full statement (Theorem \ref{theorem: Gromov for torsion-free rings}) yields stronger structural information.

\begin{theorem}
A finitely generated torsion-free ring has polynomial growth if and only if it is virtually nilpotent.
\end{theorem}

Future applications will likely be driven by obtaining effective quantitative bounds:

\begin{problem}
In Theorem \ref{theorem: Structure finite app rings}, determine explicit bounds (simultaneous or separate) for $N_1(K), \dots, N_4(K)$. Specifically, can $N_1, N_2, N_3$ be bounded polynomially and $N_4$ logarithmically?
\end{problem}

If all constants were shown to be polynomial, this would provide a unified statement generalizing both sum-product theorems in the ``few zero divisors'' setting and Bourgain's multiscale setting.  Namely, if $N_2(K)$ could be polynomially bounded, then the function $f$ in Theorem \ref{theorem: concrete sum-product} could be taken to be a certain explicit power function (see Remark \ref{remark: polynomial bound on f}). Moreover, in that case Theorem \ref{theorem: main theorem intro} would provide a generalization of Tao's sum-product phenomenon in the few zero divisors setting \cite{Tao2}. If moreover $N_3(K)$ and $N_4(K)$ could be bounded by a universal constant and a logarithm in $K$ respectively, then we would recover sum-product phenomena in $\mathbb{Z}/p\mathbb{Z}$ in the spirit of results due to Bourgain \cite{Bourgain08sumprod} and its extensions by Salehi Golsefidy \cite{SalehiGolsefidy2016sumprod}.

The model-theoretic methods employed here rely on compactness arguments and therefore do not yield effective bounds. However, they offer a clear strategy for obtaining effective (albeit likely weak) bounds. Unlike the general group-theoretic setting of BGT, the additive structure of a ring is abelian. This permits the use of the effective Freiman--Ruzsa theorem to initiate the structural analysis, avoiding some of the difficulties inherent in non-commutative inverse theorems. Note also that Theorem \ref{theorem: Structure finite app rings} yields a good bound on $N_1(K)$ simultaneously with separate (good) bounds on $N_3(K)$ and $N_4(K)$. Our methods may lead to a simultaneous explicit bound on all these three parameters. However, obtaining explicit (preferably polynomial) bound on $N_2(K)$ remains a significant challenge.

Importantly, the definition of approximate rings does not rely on finiteness. This enables us to generalize a sum-product phenomenon due to Meyer beyond the scope of classical additive combinatorics. Meyer showed in \cite{meyer1972algebraic} that infinite discrete approximate subrings of local fields are \emph{arithmetic}, in the sense that they arise from the number-theoretic construction of \emph{Pisot--Salem numbers}.

Pisot--Salem numbers are real algebraic integers $\alpha > 1$ such that all other Galois conjugates have modulus strictly less than $1$. Given a number field $K \subseteq \mathbb{R}$, the set $P(K)$ is defined as the collection of all Pisot--Salem numbers that generate $K$. Meyer's theorem asserts that infinite, uniformly discrete approximate subrings of $\mathbb{R}$ are commensurable with $P(K)$ for some number field $K \subseteq \mathbb{R}$. This result extends to other local fields, and Dani--Gowri Navada \cite{DaniGowriNavada1996Harmonious} expanded its scope to certain subsets of matrices.

A consequence of our main result concerning approximate subrings of semi-simple real algebras  (namely, Theorem \ref{Theorem: Gen. Pisot Meyer}) is the following generalization  of Meyer's theorem, which for simplicity we state here for simple real algebras (see Corollary \ref{corollary: gen. of Meyer to simple real algebras} and Remark \ref{corollary: gen. of Meyer to semi-simple real algebras}).  Recall that simple, finite-dimensional real algebras are precisely the matrix algebras $M_n(D)$, where $D$ is $\mathbb{R}$, $\mathbb{C}$, or the quaternions $\mathbb{H}$, whereas semi-simple, finite-dimensional real algebras are the products of finitely many such matrix algebras. The standard notation $\mathcal{O}_{K,v}$ used below is recalled in the introduction to Section \ref{Section: Sum-product for real algebras}.

\begin{theorem}\label{theorem: generalization of Meyer}
Let $X$ be an infinite, uniformly discrete approximate subring in a finite-dimensional simple real algebra $A$, and assume $X$ spans $A$. There exists a number field $K$, an archimedean place $v$, a $K_v$-algebra structure on $A$ extending the $\mathbb{R}$-algebra structure, and a $K_v$-basis $(e_1, \ldots, e_d)$ of $A$ such that $X$ is commensurable with $\sum_{i=1}^d \mathcal{O}_{K,v} \cdot e_i.$

In the specific case where $A=M_n(\mathbb{R})$, we have $K_v = \mathbb{R}$, $\mathcal{O}_{K,v}$ corresponds to the set of Pisot--Salem numbers $P(K)$, and
\[ X \text{ is commensurable with } \sum_{i=1}^d P(K)\cdot e_i.\]
\end{theorem}

The main engine of our proof (Theorem \ref{Theorem: Ring cap}) reveals a fundamental connection to the concept of \emph{cut-and-project schemes}. Originating from Meyer's pioneering work on approximate lattices (see \cite{meyer1972algebraic}), these constructions yield large families of discrete sets that, while aperiodic, exhibit strong algebraic and arithmetic structure. 

Although we prove our main structural result in the context of semi-simple real agebras, it is very likely that such structural results should generalize to algebras over local fields and their finite products (in the spirit of \cite{hru2} for approximate lattices and discrete approximate subgroups). Indeed, Theorem \ref{Theorem: Ring cap}, which easily generalizes, reduces the study of discrete approximate subrings to that of co-compact subrings of such algebras. In turn, those often have a strong number-theoretic origin. 

It is insightful to compare Theorem \ref{theorem: main theorem intro} with Theorem \ref{theorem: generalization of Meyer}, as both establish structural results of a similar nature. In essence, \emph{arithmeticity} in Theorem \ref{theorem: generalization of Meyer} plays a role analogous to that of \emph{nilpotency} in Theorem \ref{theorem: main theorem intro}. This analogy hints at a deeper connection between both setups, which we explore through the concept of a \emph{locally compact model} discussed in Section \ref{subsection: strategies}.

\subsection{Locally compact compact model and overview of the methods}\label{subsection: strategies}
The key tool to prove the main results of this paper are {\em locally compact models of approximate subrings} introduced in \cite{Kru}. Below we discuss locally compact models and briefly describe the strategies of the main proofs. 

Recall that a subset $X$ of a group is called an {\em approximate subgroup} if it is symmetric (i.e. $e \in X$ and $X^{-1}=X$) and $X X \subseteq FX$ for some finite $F \subseteq \langle X \rangle$. Approximate subgroups were introduced by Tao in \cite{Tao} and have become one of the central objects in additive combinatorics. This notion originates in the fundamental objects of study in additive combinatorics, namely subsets of a group with small doubling, tripling, etc.

One of the most important tools to prove structural results on approximate subgroups are {\em locally compact models}. A breakthrough in the study of the structure of approximate subgroups was obtained by Hrushovski in \cite{Hru}, where a locally compact model for an arbitrary pseudofinite approximate subgroup (more generally, {\em near-subgroup}) $X$ was obtained by using model-theoretic tools, and in consequence also a Lie model  was found for some approximate subgroup commensurable with $X$. This paved the way for Breuillard, Green, and Tao to give a full classification of all finite approximate subgroups in \cite{BGT}. 
Moreover, the existence of locally compact models for some approximate subgroups was an engine for various structural results on approximate lattices \cite{hru2, machado2019goodmodels, Machado}. 

Let us recall the definition of a locally compact model.
Let $X$ be an approximate subgroup and $G := \langle X \rangle$ the subgroup generated by $X$. By a {\em locally compact [resp. Lie] model} of $X$ we mean a group homomorphism $f \colon \langle X \rangle \to H$ to some locally compact [resp. Lie] group $H$ such that $f[X]$ is relatively compact in $H$ and there is a neighborhood $U$ of the neutral element in $H$ with $f^{-1}[U] \subseteq X^m$ for some $m<\omega$.
(In this paper, locally compact spaces are Hausdorff by definition.) 

Unfortunately, there are approximate subgroups which do not have any locally compact models, e.g. see \cite[Section 4.3]{HKP}. Hrushovski found a remedy introducing generalized locally compact models (using appropriate quasi-homomorphisms instead of homomorphisms), proving that they always exist, and applying them to structural or classifications results on certain approximate lattices \cite{hru2}. They were also used by the second author in \cite{Machado} to describe the structure of approximate lattices in linear algebraic groups over local fields. In \cite{KrPi22}, a simpler construction of universal generalized locally compact models was obtained by means of topological dynamics and basic model theory.

In contrast to approximate subgroups, the main result of \cite{Kru} tells us that every approximate subring has a locally compact model (see Fact \ref{fact: locally compact model exists} below). 
A {\em locally compact model} of the approximate subring $X$ is a ring homomorphism $f \colon \langle X \rangle \to S$ to some locally compact ring $S$ such that $f[X]$ is relatively compact in $S$ and $f^{-1}[U] \subseteq X_m$ for some $m< \omega$ and $U \subseteq S$ a neighborhood of $0$ in $S$ (see page \pageref{page: definition of  X_n} for the definition of the sets $X_n$). Using locally compact models, some structural results for approximate subrings were already obtained in \cite[Section 5]{Kru}, e.g. a classification of all approximate subrings of rings of positive characteristic or a version of the sum-product phenomenon when there are no zero divisors.

It is also worth recalling (see \cite[Remark 3.6]{Kru}) that if an additively symmetric subset $X$ of a ring has a locally compact model, then $X_n$ is an approximate subring for some $n \in \mathbb{N}$. This shows that if we hope for locally compact models, we should work with our definition of approximate subrings.

Similarly to the case of approximate subgroups, also for approximate subrings it is useful (and natural from the model-theoretic point of view) to consider definable objects.

By a {\em definable} (in some structure $M$) {\em approximate subring} we mean an approximate subring $X$ such that $X_0,X_1,\dots$ are all definable in $M$ and $+$ and $\cdot$ restricted to any $X_n$ are also definable in $M$.  A locally compact model $f \colon \langle X \rangle \to S$ of $X$ is called {\em definable} if  for any $C \subseteq U \subseteq S$ where $C$ is compact and $U$ is open there exists a definable $Y$ such that $f^{-1}[C] \subseteq Y \subseteq f^{-1}[U]$. Equivalently, we can require that $Y$ is additionally contained in some $X_n$. Note that the definable context generalizes the classical one, since given an arbitrary approximate subring $X$, one can take $M : =\langle X \rangle$ equipped with the full structure (i.e., with predicates for all subsets of all finite Cartesian powers of $M$), and then $X$ is trivially definable in $M$ and every locally compact model of $X$ is automatically definable.

Here is the aforementioned main result of \cite{Kru} (see \cite[Corollary 4.11]{Kru}) which is the main tool in this paper.

\begin{fact}\label{fact: locally compact model exists}
Every definable approximate subring has a definable locally compact model $f \colon \langle X \rangle \to S$ with $f^{-1}[U] \subseteq 4X + X \cdot 4X$\footnote{In a forthcoming joined paper of the first author with Mateusz Rzepecki the expression $4X + X \cdot 4X$ in the main theorem of \cite{Kru} will be improved to $4X + X \cdot 2X$ (which is optimal in the sense that $4X + X \cdot X$ is not enough). Having this in mind, everywhere in this paper, the expression $4X + X \cdot 4X$ can be replaced by $4X + X \cdot 2X$. This allows to decrease some constants throughout the paper, e.g. in item (2) of Theorem \ref{theorem: Structure finite app rings}  we actually have $I \subseteq X_2$ and the constant $K^{510}+K^{22}$ in Theorem \ref{theorem: main theorem intro} could be replaced by a smaller one.} for some neighborhood $U$ of $0$ in $S$. 
\end{fact}

Let us now describe our strategy to prove Theorem \ref{theorem: main theorem intro}, say the version with $N_4(K)=\lfloor 4 \log_2(K) \rfloor$, $N_1(K)=K^{510}+K^{22}$, and unspecified  $N_2(K)$, $N_3(K)$. First, in Section \ref{section: target algebra}, for any definable (in an arbitrary structure $M$) $K$-approximate subring $X$, using a definable locally compact model of $X$ provided by Fact \ref{fact: locally compact model exists}, we find a definable $(K^{510}+K^{22})$-approximate subring $Y\subseteq X + X \cdot 4X$ commensurable to $X$ with a definable locally compact model whose target space is a finite-dimensional real algebra of dimension $\leq 4 \log_2(K)$ and as such can be treated as a subalgebra of $\textrm{M}_n(\mathbb{R})$. Using this reduction and the Frobenius norm on $\textrm{M}_n(\mathbb{R})$, as a byproduct we obtain the following surprising structural result on arbitrary approximate subrings.

\begin{corollary}\label{Cor: approximate subring closed under mult}
Any definable approximate subring $X$ is commensurable with a definable approximate subring $Y\subseteq \langle X \rangle$ which is closed under multiplication (i.e., $Y \cdot Y \subseteq Y$). In fact, elaborating on the arguments as briefly mentioned three paragraphs below, we obtain such a $Y$ which is additionally contained in $4X + X \cdot 4X$.
\end{corollary}

Now, assume $X$ is a pseudofinite $K$-approximate subring (see Appendix \ref{appendix: pseudofiniteness} for an explanation of pseudofiniteness). Take $Y$ obtained in the previous paragraph. In Section \ref{section: escape norm}, using the Frobenius norm on $\textrm{M}_n(\mathbb{R})$ and our definable locally compact model of $Y$, we prove that the escape norm with respect to some definable approximate subring $Z \subseteq \langle Y \rangle $ with $\langle Z \rangle = \langle Y \rangle$ and commensurable to $X$ is {\em strong} in the sense that it satisfies properties (1)-(3) from Corollary \ref{proposition: good escape norm  2} (roughly speaking, $||x+y|| \leq 4(||x|| + ||y||)$ and $||xy|| \leq 2||x||\cdot||y||$). 
It is worth emphasizing that a quite technical part of \cite{BGT} which shows that the so-called trapping conditions imply that the escape norm is strong is completely eliminated in our situation by using the Frobenius norm on $\textrm{M}_n(\mathbb{R})$.

Theorem \ref{theorem: main theorem intro} is then proved in Section \ref{section: structure of finite approximate rings}. A general idea is to adapt to the context of rings  the strategy of the proof of the main theorem of \cite{BGT} as presented in \cite{Dries}.  More precisely, assuming that Theorem \ref{theorem: main theorem intro} with $N_4(K)=\lfloor 4 \log_2(K) \rfloor$ and $N_1(K)=K^{510}+K^{22}$ fails, by model-theoretic compactness, we get a pseudofinite $K$-approximate subring $X$ such that there is NO definable $(K^{510}+K^{22})$-approximate subring $Y \subseteq X + X \cdot 4X$ commensurable with $X$ for which there exists a definable ideal $I$ of $\langle Y \rangle^*$ contained in $Y_m$ for some $m \in \mathbb{N}$ and with $\langle Y \rangle^*/I$ nilpotent of class at most $4\log_2(K)$ (where $\langle Y \rangle^*$ denotes the smallest definable subring containing $Y$; see Corollary \ref{corollary: <X>*}). Then we replace $X$ by an approximate subring $Z$ with the properties from the previous paragraph, and we adapt the idea of the proof of \cite[Theorem 7.2]{Dries} to the context of rings (using the fact that the escape norm with respect to $Z$ is strong) to get a final contradiction.

Regarding Theorem \ref{theorem: main theorem intro} in the version with $N_3(K)=1$, $N_1(K)=K^{510}+K^{22}$, and unspecified  $N_2(K)$, $N_4(K)$, the general strategy is the same, but it is a bit more involved. First, from Gelason-Yamabe theorem we deduce its ring-theoretic counterpart (see Proposition \ref{proposition: Gleason-Yamabe for compact rings} and Corollary \ref{corollary: Gleason-Yamabe for loc. compact rings}). Using this together with Fact \ref{fact: locally compact model exists}, for any definable $K$-approximate subring $X$, we find a definable $(K^{510}+K^{22})$-approximate subring $Y \subseteq X + X \cdot 4X$ commensurable to $X$ and a definable locally compact model $f \colon \langle Y \rangle \to \mathcal{A}$ with dense image, where $\mathcal{A}$ is a connected locally compact ring whose additive group is a Lie group of the form $\mathbb{R}^n \times C$ for a connected compact Lie group $C$, and with $f^{-1}[U]\subseteq Y$ for some neighborhood $U$ of $0$ in $\mathcal{A}$ (the last property is essential to get $N_3(K)=1$). Then $\{0\} \times C$ is an ideal, and $\mathcal{A}/\{0\} \times C$ becomes an $n$-dimensional real algebra, on which we have the Frobenius norm. Then we follow the lines of the strategy described above with passing to a commensurable approximate subring with strong escape norm, etc., but now the resulting locally compact model still has the target space as described in this paragraph.

Finally, we discuss some tools involved in the proof of Theorem \ref{theorem: generalization of Meyer} presented in Section \ref{Section: Sum-product for real algebras}. Locally compact models are again found at the core of the proof. Indeed, they enable us to show that discrete approximate rings arise from a so-called \emph{cut-and-project scheme} originating in Meyer's work \cite{meyer1972algebraic}.

In general, a \emph{cut-and-project scheme} is the data of two locally compact rings $A,B$ and a discrete subring $\Delta \subseteq A \times B$  projecting injectively to $A$. Given such a triple $(A,B,\Delta)$, we can produce discrete approximate subrings of $A$ by selecting a relatively compact  additively symmetric neighbourhood $W_0 \subseteq B$ of $0$ and defining 
$$ M := \pi_A\left[\Delta \cap \left(A \times W_0\right)\right]$$
where $\pi_A: A \times B \rightarrow A$ is the natural projection. Then $M$ is a  uniformly discrete approximate subring of $A$. We call $W_0$ the \emph{window}, $A$ the \emph{direct space}, $B$ the \emph{internal space}, and $M$ is referred to as a \emph{(weak) model set}. 

Our main tool is: 

\begin{theorem}\label{Theorem: ring cap, intro}
    Let $\Lambda$ be a  uniformly discrete approximate subring of a locally compact ring $A$. Then $\Lambda$ is commensurable with a weak model set associated with  the direct space $A$ and  the internal space  being a finite-dimensional real algebra $B$. 
\end{theorem}

Here again, locally compact models underpin this result; the internal space $B$ is the target of a locally compact model of an approximate ring commensurable with $\Lambda$ and the implicit discrete subring $\Delta$ is the graph of the same locally compact model. 

When $A$ is furthermore a real algebra, Theorem \ref{Theorem: ring cap, intro} mostly reduces the study of  uniformly discrete approximate subrings of real algebras to the study of discrete subrings of real algebras. In turn, the latter often exhibit strong number theoretic properties -- much like discrete subrings of $\mathbb{R}$ are always of the form $m \mathbb{Z}$ for some integer $m$. We can then exploit this fact to prove Theorem \ref{theorem: generalization of Meyer}.

The readers interested only in Theorem \ref{theorem: main theorem intro} should go through Sections \ref{section: preliminaries}, \ref{section: target algebra}, \ref{section: escape norm}, and \ref{section: structure of finite approximate rings}. The readers interested in Theorem \ref{theorem: generalization of Meyer} should read Sections \ref{section: target algebra} and \ref{Section: Sum-product for real algebras}.

\subsection{Auxiliary structural applications of locally compact models}

In Section \ref{section: additional structural results}, we obtain some other consequences of the existence of locally compact models, this time using also model-theoretic connected components which stand behind locally compact models, which is discussed in the first part of Section  \ref{section: additional structural results}.

First of all, elaborating on the proof of \cite[Theorem 5.4]{Kru}, we get the following variants of Tao's sum-product phenomena from \cite{Tao2}.  The second statement goes much beyond the context of \cite{Tao2}, as it concerns infinite approximate subrings. 
(Thickness is combinatorial notion of largeness important in model theory; see Definition \ref{definition: thickness abstractly}.)

\begin{theorem}\label{theorem: thickness and sum-product}
For every $K \in \mathbb{N}$ there exists $N(K) \in \mathbb{N}$ such that for every  finite $K$-approximate subring $X$ of a ring either there is an $N(K)$-{\em thick} (in particular, of cardinality at least $\frac{|Y|}{N(K)-1}$) subset of $Y:=4X + X \cdot 4X$ consisting of zero divisors or $Y$ is a subring (additively $K^{11}$-commensurable with $X$).
\end{theorem}

\begin{theorem}\label{theorem: NSOP}
Let $X$ be a definable approximate subring .
If $\Th(M)$ (i.e. the theory of $M$) has NSOP (i.e. {\em non strict order property}), then either there is a definable thick subset $D$ of $Y:=4X + X \cdot 4X$ consisting of zero divisors or $Y$ is a subring (additively $K^{11}$-commensurable with $X$).
\end{theorem}

The class of NSOP (see Definition \ref{definition: NSOP}) theories is reach. It contains all stable, and, more generally, all simple theories. Thus, among many interesting examples, this class includes the theories of algebraically closed fields, separably closed fields, differentially closed fields, or bounded PAC fields. In particular, Theorem \ref{theorem: NSOP} applies to all definable approximate subrings of $M_n(K)$, where $M:=(K,+,\cdot)$ [$M:=(K,+,\cdot, D)$ when $K$ is a diferrentailly closed field] is any of the above fields. 


The final part of Section \ref{section: additional structural results} is related to approximate subfields.
Recall the definition of a finite approximate subfield (e.g. see \cite[Definition 5.1]{Bre}). Let $F$ be a field. A
finite subset $X$ of $F$ is said to be a {\em K-approximate subfield of $F$} if it is additively and multiplicatively symmetric (multiplicatively in the sense that $1 \in X$ and $X\setminus \{0\} = (X\setminus \{0\})^{-1}$), and there is a subset $E$ of $F$ with $|E| \leq K$, such that $XX +X$ is contained in $EX \cap (E + X)$. By \cite[Lemma 5.2]{Bre}, we know that if $X$ is a finite approximate subfield, then for every $n \in \mathbb{N}\setminus \{0\}$, the set $\textrm{Alg}_n(X)$ (of the quotients of sums of at most $n$ terms each of which is a product of at most $n$ elements from $X$) is contained in $E_nX \cap (E_n+X)$  for some subset $E_n$ of cardinality $K^{O_n(1)}$.

We prove the following

\begin{proposition}\label{proposition: infinite approximate subfields}
If $X$ is an infinite approximate subring of a field with the property that for every positive $n \in \mathbb{N}$, $\textrm{Alg}_n(X)$ is covered by finitely many additive translates of $X$, then $4X + X \cdot 4X$ is a subfield.
\end{proposition}

\begin{corollary}\label{corollary: generics in fields}
Let $F$ be an infinite field and $X \subseteq F$ additively symmetric and generic (in the sense that finitely many additive translates of $X$ cover $F$). Then $4X + X \cdot 4X =F$.
\end{corollary}

The next corollary is a variant of the classification of finite approximate subfields from \cite[Theorem 5.3]{Bre}. 

\begin{corollary}\label{corollary: classification of finite approximate subfields}
For every $K \in \mathbb{N}$ there is $N(K) \in \mathbb{N}$  such that for every  finite $K$-approximate subfield $X$ either $|X| <N(K)$ or $4X + X \cdot 4X$ is a subfield. 
More generally, instead of $X$ being a finite $K$-approximate subfield, one can consider the context when $X$ is a finite $K$-approximate subring of a field with $\textrm{Alg}_n(X)$ covered by $N(K,n)$ additive translates of $X$ (where $N(K,n) \in \mathbb{N}$ depends only on $K$ and $n$).
\end{corollary}


\section{Preliminaries}\label{section: preliminaries}

\subsection{Model theory}

In this paper, we use only basic model theory. Here we recall some fundamental things. The reader interested in details may consult e.g. \cite{Tent_Ziegler}.

A {\em language} (usually denoted by  $\mathcal{L}$) consists of function, relation, and constant symbols. Using those symbols together with quantifiers, variables and logical symbols, one constructs recursively the set of all {\em $\mathcal{L}$-formulas}; {\em $\mathcal{L}$-sentences} are $\mathcal{L}$-formulas without free variables. An {\em $\mathcal{L}$-structure} is a nonempty set $M$ together with interpretations of all the symbols of $\mathcal{L}$. For example, if $\mathcal{L}$ consists of two binary function symbols, then any ring is an $\mathcal{L}$-structure. 
A {\em theory} in the language $\mathcal{L}$ is a set of $\mathcal{L}$-sentences.
A theory $T$ in $\mathcal{L}$ is {\em consistent} if it does not prove a contradiction (using classical logic calculus) or, equivalently, it has a {\em model}, i.e. an $\mathcal{L}$-structure $M$ in which all sentences from $T$ are true (symbolically, $M \models T$). {\em Compactness theorem} (or {\em model-theoretic compactness}) tells us that a theory $T$ has a model if and only if every finite subset of $T$ has a model. A consistent theory $T$ in $\mathcal{L}$ is {\em complete} if for every $\mathcal{L}$-sentence $\varphi$, $T$ proves $\varphi$ or $T$ proves $\neg \varphi$. Whenever $M$ is an $\mathcal{L}$-structure, then $\Th(M):=\{ \varphi \textrm{ an }\mathcal{L}\textrm{-sentence} : M \models \varphi\}$ (i.e. {\em the theory of $M$}) is a complete theory, and, conversely, any complete theory is the theory of each model of $T$.

Let us fix an arbitrary $\mathcal{L}$-structure $M$.
An $\mathcal{L}$-structure $N$ is an {\em elementary extension} of $M$ (symbolically, $M \prec N$) if $M \subseteq N$ and for every $\mathcal{L}$-formula $\varphi(x_1,\dots,x_n)$ and tuple $(a_1,\dots,a_n) \in M^n$ we have $M \models \varphi(a_1,\dots,a_n) \iff N \models  \varphi(a_1,\dots,a_n)$. 

For any subset $A$ of $M$ we can expand the language $\mathcal{L}$ to $\mathcal{L}_A$ be adding constant symbols for the members of $A$, which are then interpreted in $M$ as the corresponding elements of $A$.
By a {\em type} over $A \subseteq M$ in variables $x$ we mean a consistent collection $\pi(x)$ of $\mathcal{L}_A$-formulas, where $\pi(x)$ being {\em consistent} means that for any finitely many formulas $\varphi_1(x),\dots,\varphi_n(x) \in \pi(x)$ we have $M \models (\exists x)( \varphi_1(x) \wedge \dots \wedge \varphi_n(x))$. By compactness theorem, this is equivalent to the property that $\pi(x)$ has a realization $a$ in some $N \succ M$, i.e. $N \models \varphi(a)$ for all $\varphi(x) \in \pi(x)$. 
A {\em complete type} over $A$ in variables $x$ is a type $p(x)$ over $A$ such that for every $\mathcal{L}_A$-formula $\varphi(x)$ we have $\varphi(x) \in p$ or $\neg \varphi(x) \in p$. This is equivalent to saying that $p=\tp(a/A):=\{\varphi(x) \textrm{ an $\mathcal{L}_A$-formula}: N\models \varphi(a)\}$ for some tuple $a$ in some $N \succ M$.

For a given cardinal $\kappa$, we say that $N \succ M$ is {\em $\kappa$-saturated} if for every $B \subseteq N$ of cardinality $<\kappa$, every complete type $p$ over $B$ has a realization in $N$. 
Using compactness theorem, for every $\kappa$ there exists $N \succ M$ which is $\kappa$-saturated. 

The downward L\"{o}wenheim-Skolem theorem tells us that if $M$ is infinite and $\lambda$ is an infinite cardinal which is at least $|\mathcal{L}| + \aleph_0$, then there is $N \prec M$ of cardinality $\lambda$. In particular, if $\mathcal{L}$ is countable, every infinite $\mathcal{L}$-structure has an infinite  countable elementary substructure.

For an $\mathcal{L}_A$-formula $\varphi(x)$, $\varphi(M)$ denotes the {\em set of realizations} of $\varphi(x)$ in $M$, i.e. $\varphi(M):= \{ a \in M^{|x|}: M \models \varphi(a)\}$. By an {\em $A$-definable subset} of $M$ [more generally, of a Cartesian power $M^n$] we mean the set of realizations in $M$ of an $\mathcal{L}_A$-formula $\varphi(x)$ with one [resp. $n$] free variables $x$. By a {\em definable subset} we mean an $M$-definable subset. If $M$ is $\kappa$-saturated and $A \subseteq M$ is of cardinality less than $\kappa$, then by an {\em $A$-type-definable} set in $M$ we mean a nonempty intersection of possibly infinite family of $A$-definable sets. If $X$ is a definable set in $M$ over a set of parameters $A\subseteq M\cap N$, where $N \prec M$ or $N \succ M$, then $X(N)$ denotes the interpretation of $X$ in $N$ (i.e., the set of realizations in $N$ of a formula in the language $\mathcal{L}_A$ defining $X$ in $M$).

The whole discussion above works also for many-sorted languages and structures. The difference is that now the language $\mathcal{L}$ specifies a collection of sorts and all symbols of the language are associated with some products of sorts. Then a (many-sorted) $\mathcal{L}$-structure consists of sorts (or interpretations of sorts), each relation symbol is interpreted as a subset of the associated product of sorts, and similarly for function and constant symbols. For example, an $R$-module $V$ can be naturally treated as a 2-sorted structure with the sorts $(R,+_R,\cdot_R)$ and $(V,+_V)$, and with the function $\cdot \colon R \times V \to V$.

An {\em imaginary sort} is a set of the form $M^n/E$, where $E$ is a $\emptyset$-definable equivalence relation on $M^n$. By $M^{eq}$ one denotes $M$ expanded by all imaginary sorts and the quotient maps corresponding to them (so $M^{eq}$ is a many-sorted structure). It turns out that any definable subset of a product of sorts $M^{n_1}/E_1 \times \dots \times M^{n_k}/E_k$ is the quotient of a definable set in $M^{n_1} \times \dots \times  M^{n_k}$ and there is no harm in passing to $M^{eq}$. For example, we do not loose $\kappa$-saturation, being an elementary extension, etc.. Also, $(M^{eq})^{eq}$ is naturally identified with $M^{eq}$ (i.e. $M^{eq}$ trivially has elimination of imaginaries). Even if $E$ is an $A$-definable equivalence relation on $M^n$ (where $A \subseteq M$ is finite), the quotient $M^n/E$ is an imaginary sort working in the language $\mathcal{L}_A$. In fact, $M^n/E$ can be naturally identified with $(M^n \times \{a\})/E'$, where $E'$ is a $\emptyset$-definable (in $\mathcal{L}$) equivalence relation on $M^{n+|a|}$ (where $a$ is an enumeration of $A$). In any case, there is no harm to work with equivalence relations definable over parameters.

In this paper, the main model theory context will be that of pseudofinite approximate subrings. There are several ways to formalize this concept. We choose a rather explicit one, and present it in Appendix \ref{appendix: pseudofiniteness}. In this appendix, we will work with a finite language, and use imaginary sorts with respect to equivalence relations definable over parameters. Having in mind the above discussion, we can freely talk about definable subsets or definable functions between such imaginary sorts.

\subsection{Approximate subrings}

The definition of an approximate subring is given in the abstract. Definable approximate subrings and definable locally compact models of definable approximate subrings are defined in Section \ref{subsection: strategies}. 

To avoid any confusion, let us emphasize that rings in this paper are associative, but need not be commutative or unital. For a subset $X$ of any ring, $2X=X+X:=\{x+y: x,y \in X\}$ and $X^2=X \cdot X:=\{xy: x,y \in X\}$; $nX$ and $X^n$ are defined analogously. 

We say that a ring $R$ is {\em nilpotent of class at most $n$} if $R^{n+1}=\{0\}$. In particular, $R$ being nilpotent of class at most $0$ means that $R=\{0\}$, and being nilpotent of class at most $1$ means that $R^2=\{0\}$.

We say that a ring $R$ is {\em $n$-nilpotent} if there are $u_1,\dots,u_n \in R$ generating $R$ and such that $u_i \cdot u_j, u_j \cdot u_i \in \langle u_1,\dots,u_{i-1} \rangle$ for any $1\leq i \leq j \leq n$. Any tuple $u_1,\dots,u_n$ as above will be called a {\em nilpotent base} of $R$. It is clear that $n$-nilpotency implies being nilpotent of class at most $n$, but it is essentially stronger, as it bounds the number of generators. Note also that in the definition of $n$-nilpotency we can equivalently require that $u_i \cdot u_j, u_j \cdot u_i$ belong to the additive group (rather than ring) generated by $u_1,\dots,u_{i-1}$.

For the proof of following fact see e.g. Proposition 2.3 in \cite{Bre} and references in there. 

\begin{fact}[Plunnecke-Ruzsa sumset estimates]\label{fact: small n-pling implies approximate}
Let $X$ be a finite nonempty subset of an abelian group. If $|X+X| \leq K|X|$, then $|nX - mX|\leq K^{n+m}|X|$ for all $n,m \in \mathbb{N}$ and $X-X$ is a $K^5$-approximate subgroup.
\end{fact}

Let us first prove the following remark already mentioned in the introduction:

\begin{remark}\label{remark: small n-pling for rings}
Let $X$ be a finite nonempty subset of a ring. If $|X+X+XX| \leq K|X|$ (or just $|X+X|, |X+XX| \leq K|X|$), then $X-X$ is a $(K^{5} +K^{19})$-approximate subring.
\end{remark}

\begin{proof}
Since $|X+X| \leq K|X|$, by Fact \ref{fact: small n-pling implies approximate}, $X-X$ is a $K^5$-approximate subgroup, so $2(X-X) \subseteq E+(X-X)$ for some $E \subseteq \langle X \rangle$ of size at most $K^5$.  As $|X+X^2| \leq K|X|$, by Ruzsa's covering lemma (e.g., see \cite[Lemma 3.2]{Bre}), $X^2 \subseteq F +X-X$ for some $F\subseteq X^2$ of size at most $K$. Hence, $(X-X)\cdot(X-X)  \subseteq 2X^2 -2X^2\subseteq 2(F+X-X) - 2(F+X-X)=2F -2F +4(X-X) \subseteq 2F -2F + 3E + (X-X)$. The conclusion follows as $|2F-2F+3E| \leq K^2 \cdot K^2 \cdot (K^5)^3=K^{19}$.
\end{proof}

The next lemma leads to the explicit polynomial bound on $N_1(K)$ in Theorem \ref{theorem: main theorem intro}.

\begin{lemma}\label{lemma: K510}
If $X$ is a $K$-approximate subring and $H$ is a subring of an ambient ring $R$, then $Y:=(4X + X \cdot 4X) \cap H$ is a $(K^{510}+K^{22})$-approximate subring.
\end{lemma}

\begin{proof}
Let $Z := 4X + X \cdot 4X$. First, observe that it is enough to show that $2Z \cup Z^2$ is covered by $C:=K^{510}+K^{22}$ additive translates of $2X$. For that we use the argument from \cite[Lemma 2.9]{Dries}. Namely, consider any $a \in R$ such that there is $h \in (a +2X) \cap (2Z \cup Z^2) \cap H$. Then $a-h \in 2X$, so $(a-h +2X) \cap H \subseteq Y$, and hence $(a+2X) \cap (2Z \cup Z^2) \cap H \subseteq (a+2X) \cap H \subseteq h+Y$. Thus, if $2Z \cup Z^2$ is covered by $C$ additive translates of $2X$, then $(2Z \cup Z^2) \cap H$ is covered by $C$ translates of $Y$. This is enough, because $2Y \cup Y^2 \subseteq (2Z \cup Z^2) \cap H$.

In the rest of the proof, we will show that $2Z \cup Z^2$ is covered by $C$ additive translates of $2X$. 
We have that $2X \subseteq F_1+X$ and $X^2 \subseteq F_2 +X$ for some $F_1 \subseteq 3X$ and $F_2 \subseteq X^2 +X$ with $F:= F_1 \cup F_2$ of size at most $K$. Write each element $f \in F_2$ as $f_1\cdot f_2 + f_3$ for some $f_1,f_2,f_3 \in X$. We have
\begin{equation}\label{first equation}
2Z = 8X + 2(X \cdot 4X) \subseteq 8X + 8X^2 \subseteq 8X + 8(F+X) = 8F + 16X \subseteq 22F + 2X.
\end{equation}
$Z^2$ requires a longer computation. Namely, 
\begin{equation}\label{second equation}
Z^2 \subseteq (4X)^2 + (4X) \cdot X \cdot (4X) + X \cdot (4X) \cdot (4X) + (X \cdot 4X)^2,
\end{equation}
and we have the following inclusions:

\begin{itemize}
\item $(4X)^2 \subseteq 16X^2 \subseteq 31F + X$.
\item 
Both $(4X) \cdot X \cdot (4X)$ and  $X \cdot (4X) \cdot (4X)$ are contained in $16X^3$. On the other hand, $X^3 =X \cdot X^2 \subseteq X(F_2 +X) \subseteq XF_2 + X^2 \subseteq F_2 +X +XF_2$. We also have $XF_2 =  X\cdot \{f_1f_2+f_3: f \in F_2\} \subseteq X \cdot \bigcup_{f \in F_2} (Xf_2 +X) \subseteq \bigcup_{f \in F_2} (X^2f_2 +X^2) \subseteq X^2 + \bigcup_{f \in F_2} (F_2+X)f_2 \subseteq 2X^2 + \bigcup_{f \in F_2} F_2f_2 \subseteq 3F +  \bigcup_{f \in F_2} F_2f_2 + X$.  We conclude that $16X^3$ is contained in $95F + 16\bigcup_{f \in F_2} F_2f_2 +X$. 

\item 
$(X \cdot 4X)^2 \subseteq 16X^4$. On the other hand, $X^4 \subseteq (F_2+X)(F_2+X) \subseteq F_2^2 +X^2+ F_2 X + X F_2\subseteq F_2^2 + F_2 +X +F_2X + XF_2$. In the second bullet, we already showed that  $XF_2 \subseteq 3F +  \bigcup_{f \in F_2} F_2f_2 + X$. Analogously, $F_2X \subseteq 3F +  \bigcup_{f \in F_2} f_1F_2 + X$. Therefore, $16X^4 \subseteq 16F^2 + 159F +  16\bigcup_{f \in F_2} f_1F_2 +   16\bigcup_{f \in F_2} F_2f_2 +X$.
\end{itemize}

By (\ref{second equation}) and the three bullets, we get that 
\begin{equation}
Z^2 \subseteq 16F^2 + 382F + 16\bigcup_{f \in F_2} f_1F_2 + 32\bigcup_{f \in F_2} F_2f_2 +2X,
\end{equation}
so $Z^2$ is covered by $K^{510}$ translates of $2X$, which together with (\ref{first equation}) yields the desired conclusion.
\end{proof}

From now on, in this section, let $M$ be a structure in a language $\mathcal{L}$, $X$ an approximate subring definable in $M$, $\C \succ M$ an $|M|^+$-saturated elementary extension, $R:=\langle X \rangle$, $\bar X:=X(\C)$, $\bar R :=\langle \bar X \rangle$. The sequence $(X_n)_{n \in \mathbb{N}}$ is defined as in the introduction on page \pageref{page: definition of  X_n}.

Locally compact models and the main result of \cite{Kru} on the existence of locally compact models are already discussed in Section \ref{subsection: strategies}; in particular, see Fact \ref{fact: locally compact model exists}.


The proof of Lemma 3.2(2) in \cite{Kru}, yields 

\begin{remark}\label{remark: premimages of compact sets}
If $f \colon \langle X \rangle \to S$ is a locally compact model and $C \subseteq S$ is compact, then $f^{-1}[C] \subseteq X_m$ for some $m$.
\end{remark}

\begin{fact}\label{fact: D is approximate}
Let $f \colon \langle X \rangle \to S$ be a locally compact model, $U \subseteq V$ with $U$ a  neighborhood of $0$ and $V$ relatively compact, and let $D$ be between $f^{-1}[U]$ and $f^{-1}[V]$. Then $D$ is commensurable with $X$. If $D$  is also symmetric, then it is an approximate subring commensurable with $X$. Moreover, if $D$ is definable, then $f|_{\langle D \rangle} \colon \langle D \rangle \to S$ is a definable locally compact model of $D$.
\end{fact}

\begin{proof}
The first part follows from \cite[Fact 3.5]{Kru} and the second from \cite[Corollary 2.2]{Kru}. In the moreover part, the fact that  $f|_{\langle D\rangle}$ is a locally compact model of $D$ is clear from the choice of $D$. For the definability of $f|_{\langle D \rangle}$, consider any $C \subseteq O \subseteq S$, where $C$ is compact and $O$ is open. By the definability of $f$, there is a definable set $Y$ between $f^{-1}[C]$ and $f^{-1}[O]$. On the other hand, 
by Remark \ref{remark: premimages of compact sets},
$(f|_{\langle D\rangle})^{-1}[C] \subseteq D_n$ for some $n$. Hence, $Y':=Y \cap D_n$ is a definable subset of $\langle D \rangle$ which lies between $(f|_{\langle D\rangle})^{-1}[C]$ and $(f|_{\langle D\rangle})^{-1}[O]$.
\end{proof}

The next remark follows from Remark \ref{remark: premimages of compact sets} and Fact \ref{fact: D is approximate}.

\begin{remark}\label{remark: preimage of a compact set}
If $f \colon \langle X \rangle \to S$ is a locally compact model of $X$, then for every relatively compact $C \subseteq S$ the preimage $f^{-1}[C]$ is covered by finitely many additive translates of $X$. If $W$ is a relatively compact neighborhood of $0$, then  $f^{-1}[U]$ is commensurable with $X$. 
\end{remark}

\begin{fact}\label{fact: extension to the monster}
Every definable locally compact model $f \colon \langle X \rangle \to S$ of $X$ has a unique extension to an $M$-definable locally compact model $\bar f \colon \langle \bar X \rangle \to S$ of $\bar X$, where {\em $M$-definable} means that $\bar f^{-1}[F] \cap \bar X_n$ is $M$-type-definable for every closed $F \subseteq S$ and every $n \in \mathbb{N}$.
\end{fact}

\begin{proof}
The existence of a unique extension of $f$ to an $M$-definable map  $\bar f \colon \langle \bar X \rangle \to S$  and the fact that $\bar f$ is a ring homomorphism is explained in Lemma 3.2(1) of \cite{Kru}. As explained in the proof of that lemma, $\bar f [\bar X] \subseteq \cl(f[X])$, so $\bar f [\bar X]$ is relatively compact. 

Now, take an open neighborhood $U$ of $0$ in $S$ such that $f^{-1}[U] \subseteq X_m$ for some $m \in \mathbb{N}$. We will show that $\bar f^{-1}[U] \subseteq \bar X_m$, which will complete the proof that $\bar f$ is a locally compact model. By $M$-definability of $\bar f$, we have that $\bar f^{-1}[U]$ is a union of $M$-definable sets. If it was not contained in $\bar X_m$, then one of the sets of this union would intersect the complement of $\bar X_m$, so the same would hold for the corresponding sets computed in $M$, and hence $f^{-1}[U]$ would not be contained in $X_m$.
\end{proof}

The following remark easily follows from Remark \ref{remark: premimages of compact sets} and model-theoretic compactness.

\begin{remark}\label{remark: characterization of M-definability of a model}
A locally compact model $f \colon \langle \bar X \rangle \to S$ of $\bar X$ is $M$-definable if and only if for any $C \subseteq U \subseteq S$, where $C$ is compact and $U$ is open, there exists an $M$-definable set $Y$ such that $f^{-1}[C] \subseteq Y \subseteq f^{-1}[U]$.
\end{remark}

\begin{proposition}\label{proposition: model is onto}
Let $f \colon \langle X \rangle \to S$ be a definable locally compact model and $\bar f \colon \langle \bar X \rangle \to S$ the unique $M$-definable extension (see Fact \ref{fact: extension to the monster}). Then:
\begin{enumerate}
\item $\bar f[\bar X_n]= \cl(f[X_n])$ is compact for every $n \in \mathbb{N}$;
\item if $f$ has dense image, then $\bar f$ is surjective.
\end{enumerate}
\end{proposition}

\begin{proof}
(1) We first argue that $\bar f [\bar X_n]$ is compact. Note that $\bar f^{-1}[\bar f [\bar X_n]] = \ker(\bar f) + \bar X_n$ and $\ker(\bar f)$ is $M$-type-definable and contained in some $\bar X_m$, so $\bar f^{-1}[\bar f [\bar X_n]]$ is also $M$-type-definable and contained in $\bar X_{\max(m,n)+1}$. Hence, by $M$-definability of $\bar f$, for every nonempty closed subset $F$ of $\bar f [\bar X_n]$, we have that $\bar f^{-1}[F]$ is $M$-type-definable. Now, compactness of $\bar f [\bar X_n]$ follows from $|M|^+$-saturation of $\C$, as any family of $M$-type-definable sets with the finite intersection property has nonempty intersection.

As explained in the proof of Lemma 3.2(1) of \cite{Kru}, $\bar f[\bar X_n]\subseteq \cl(f[X_n])$. Since $f[X_n] \subseteq \bar f[\bar X_n]$ and $\bar f[\bar X_n]$ is compact and so closed by the first part of the proof, we conclude that $\bar f[\bar X_n]= \cl(f[X_n])$.

(2) Take $s \in S$. Take a relatively compact, open neighborhood $U$ of $s$. Since $U$ is relatively compact, $f^{-1}[U] \subseteq X_n$ for some $n \in \mathbb{N}$ by 
Remark \ref{remark: premimages of compact sets}. Since $U$ is open and $f$ has dense image, we conclude that $f[X_n]$ is dense in $U$. So, by (1), $s \in \cl(f[X_n])=\bar f [\bar X_n]$.
\end{proof}

\begin{remark}\label{remark: pi f is a model}
If $f \colon \langle X \rangle \to S$ is a definable locally compact model of $X$ and $I$ is a compact ideal of $S$, then $\pi \circ f \colon \langle X \rangle  \to S/I$ is also a definable locally compact model of $X$, where $\pi \colon S \to S/I$ is the quotient map.
\end{remark}

\begin{proof}
$\pi \circ f$ is clearly a ring homomorphism and $\pi[f[X]]$ is relatively compact, because $\pi[\cl(f[X])]$ is compact as a continuous image of a compact set. Now, take a compact neighborhood $U$ of $0$ in $S/I$. By compactness of $I$, we get that $\pi^{-1}[U]$ is compact in $S$, so $f^{-1}[\pi^{-1}[U]]$ is contained in some $X_n$. To see definability of $\pi \circ f$, consider any $C \subseteq U \subseteq S/I$, where $C$ is compact and $U$ is open. Then $\pi^{-1}[C] \subseteq \pi^{-1}[U] \subseteq S$, where $\pi^{-1}[C]$ is compact  and $\pi^{-1}[U]$ is open. Hence,  since $f$ is definable, $f^{-1}[\pi^{-1}[C]] \subseteq D \subseteq f^{-1}[\pi^{-1}[U]]$ for some definable $D$.
\end{proof}

\begin{corollary}\label{corollary: <Y>=<X>}
Let $f \colon  \langle X \rangle \to S$ be a definable locally compact model with dense image and connected $S$, and let $Y \subseteq X$ be a definable approximate subring which contains $f^{-1}[U]$ for some open neighborhood $U$ of $0$ in $S$. 
Then $\langle Y \rangle = \langle X \rangle$ and $h:=f|_{\langle Y \rangle}\colon \langle Y \rangle \to S$ is a definable locally compact model of $Y$.
\end{corollary}

\begin{proof}
Let $h:=f|_{\langle Y \rangle}\colon \langle Y \rangle \to S$. It is clearly a definable locally compact model (for definability use Remark \ref{remark: premimages of compact sets}). Take the extension $\bar h \colon \langle \bar Y \rangle \to S$ of $h$ from Fact \ref{fact: extension to the monster}. Using Fact \ref{proposition: model is onto}, $\bar h[Y] =\cl(h[Y])$ which contains $U$ by assumptions. Hence, $\bar h$ is surjective by connectendenss of $S$. On the other hand, $\bar{h}^{-1}[U] \subseteq \bar Y$ by the second paragraph of the proof of Fact \ref{fact: extension to the monster}, and so $\ker(\bar h) \subseteq \bar Y$. Therefore, $\langle \bar X \rangle = \langle \bar Y \rangle$, which implies that $\langle X \rangle = \langle Y \rangle$.
\end{proof}

\begin{fact}\label{fact: passing to a connected quotient}
\begin{enumerate}
\item Let $f \colon \langle X \rangle \to S$ be a definable locally compact model such that $f^{-1}[U] \subseteq X$  and $f[X] \cap U$ is dense in $U$ for some open neighborhood $U$ of $0$ in $S$. Suppose that $S'$ is an open subring of $S$ and $I$ is a compact two-sided ideal of $S'$ contained in $U$ such that $S'/I$ is connected. 
Then $Y:=X \cap f^{-1}[S']$ is a definable approximate subring commensurable to $X$ and contained in $X$, and there is a definable  locally compact model $f' \colon \langle Y \rangle \to S'/I$ with dense image and satisfying $f'^{-1}[V] \subseteq  Y$ for some neighborhood $V$ of $0$ in $S'/I$.
\item Let $f \colon \langle \bar X \rangle \to S$ be an $M$-definable locally compact model  such that $f^{-1}[U] \subseteq \bar X$ and $U \subseteq f[\bar X]$ for some open neighborhood $U$ of $0$ in $S$. Suppose that $S'$ is an open subring of $S$ and $I$ is a compact two-sided ideal of $S'$ contained in $U$ such that $S'/I$ is connected. 
Then $Y:=\bar X \cap f^{-1}[S']$ is an $M$-definable approximate subring commensurable to $\bar X$ and contained in $\bar X$, and there is a surjective $M$-definable locally compact model $f' \colon \langle Y \rangle \to S'/I$ satisfying $f'^{-1}[V] \subseteq  Y$ for some neighborhood $V$ of $0$ in $S'/I$.
\end{enumerate}
\end{fact}

\begin{proof}
Bearing in mind Proposition \ref{proposition: model is onto}, item (1) easily follows from item (2). A proof of item (2) is essentially explained in \cite[Theorem 5.5]{Dries}, but we give the details. Put 
$$Y:=\bar X \cap f^{-1}[S'].$$
Since $S'$ is clopen and $f$ is $M$-definable, we get that $Y$ is $M$-definable. Since $f^{-1}[U \cap S'] \subseteq Y$, using Fact \ref{fact: D is approximate} and its proof, we conclude that $Y$ is an approximate subring commensurable with $X$ and $f|_{\langle Y \rangle} \colon \langle Y \rangle \to S$ is an $M$-definable locally compact model of $Y$.  

Let $f' \colon \langle Y \rangle \to S'/I$ be the composition of $f|_{\langle Y \rangle}$ with the quotient map. By Remark \ref{remark: pi f is a model} and its proof, $f'$ is an $M$-definable locally compact model of $Y$. Moreover, $f'$ is surjective, which follows from the assumption that $S'/I$ is connected and the observation that $U \cap S'/I$ is an open subset of $f'[Y] \subseteq S'/I$.

By compactness of $I\subseteq U$, there exists a neighborhood $V$ of $0$ in $S'$ such that $I+V \subseteq U$. Then $V/I$ is a neighborhood of $0$ in $S'/I$ and $f'^{-1}[V/I] \subseteq f^{-1}[V+I] \subseteq f^{-1}[U \cap S'] \subseteq Y$.
\end{proof}

By $M_n(\mathbb{R})$ we denote the algebra of $n\times n$ matrices over $\mathbb{R}$. Recall that the Frobenius norm on $M_n(\mathbb{R})$ is given by $||(a_{ij})||_{\textrm{F}}:= \sqrt{\sum_{i,j} a_{ij}^2}$. 

The next fact is folklore.

\begin{fact}\label{fact: submultiplicativity}
Besides $||A+B||_{\textrm{F}} \leq ||A||_{\textrm{F}} + ||B||_{\textrm{F}}$ the Frobenius norm satisfies $||A\cdot B||_{\textrm{F}} \leq ||A||_{\textrm{F}} \cdot ||B||_{\textrm{F}}$ and induces the unique topology on $M_n(\mathbb{R})$ under which $M_n(\mathbb{R})$ is a topological vector space over $\mathbb{R}$.
\end{fact}

The next fact is \cite[Theorem 1]{Kap}.

\begin{fact}\label{fact: Thm. 1 from Kaplansky}
If $S$ is a locally compact ring and $C \leq (S,+)$ is compact, then $C \cdot S^0=S^0 \cdot C = \{0\}$, where $S^0$ denotes the connected component of $0$ in $S$.
\end{fact}

Let us finish with the celebrated Gleason-Yamabe theorem \cite{Gle,Yam}. In fact, we will only need it in the context of compact groups in which case it follows from the fact (a consequence of Peter-Weyl theorem) that each compact group is an inverse limit of Lie groups.

\begin{fact}\label{fact: Gleason-Yamabe}
For any locally compact group $G$ and any neighborhood $U$ of the neutral element in $G$, there is an open subgroup $G'$ of $G$ and a compact normal subgroup $N \subseteq U$ of $G'$ such that $G'/N$ is a connected Lie group.
\end{fact}

\section{Improving the target space}\label{section: target algebra}

In this section, we show that for every definable approximate subring $X$ we can find a commensurable definable approximate subring $Y\subseteq 4X+X \cdot 4X$ with a definable locally compact model $f \colon \langle Y \rangle \to S$ with dense image, where $S$ is a connected locally compact ring whose additive group is of the form $\mathbb{R}^n \times C$ for a connected compact Lie group $C$, and with $f^{-1}[U] \subseteq Y$ for some neighborhood $U$ of $0$ in $S$. Dropping the last requirement, the target space can be improved to be a finite-dimensional real algebra, so a subalgebra of $M_n(\mathbb{R})$ for some $n$. Using the Frobenius norm on $M_n(\mathbb{R})$, we deduce that every definable approximate subring $X$  is commensurable with a definable  approximate subring $Y \subseteq 4X+X\cdot 4X$ closed under multiplication.

We will need the following result which is a ring-theoretic strengthening of Gleason-Yamabe theorem for compact rings. Corollary \ref{corollary: Gleason-Yamabe for loc. compact rings} extends it also to locally compact rings.

\begin{proposition}\label{proposition: Gleason-Yamabe for compact rings}
Let $R$ be a compact ring and $U \subseteq R$ a neighborhood of $0$. Then there exists an open subring $S$ of $R$ and a closed ideal $J \lhd S$ contained in $U$ such that the additive group of $S/J$ is a (compact) connected Lie group.
\end{proposition}

\begin{proof}

By Fact \ref{fact: Gleason-Yamabe}, there is an open subgroup $G$ of $(R,+)$ and a closed subgroup $H \leq G$ contained in $U$ and such that $G/H$ is a connected Lie group. Replacing $G$ by an open subring $S$ of $R$ contained in $G$ (for the existence of $S$ see e.g. the first paragraph of the proof of Theorem 5.2 in \cite{Kru}) and $H$ by $H \cap S$, we can assume that $G$ is an open subring, so without loss of generality $G=R$. Thus, we can assume that $R/H$ is a connected Lie group. Then there is an open neighborhood $O$ of $\{0/H\}$ in $R/H$ which contains no subgroups other than $\{0/H\}$. Let $\pi_H \colon R \to R/H$ be the quotient map. Then $\pi_H^{-1}[O]$ is an open neighborhood of $0$ in $R$.

As $R$ is a compact ring, by Fact \ref{fact: Thm. 1 from Kaplansky}, $R^0 \cdot R = R \cdot R^0 =\{0\}$. Since $(R,+)$ is a compact group, $R/R^0$ is a profinite group (e.g. see \cite[Proposition 4.1.6(b)]{DeEc}). Since it is also a topological ring, it must be a profinite ring (e.g. see the claim in the proof of \cite[Proposition 2.10]{GJK}), i.e. $R^0$ is the intersection of a downward directed family $\{I_i:i \in I\}$ of open ideals of $R$ (so $(I,\leq)$ is a directed set and $I_j \subseteq I_i$ whenever $i \leq j$).

\begin{clm}
There exists $i \in I$ for which $R \cdot I_i \cup I_i \cdot R \subseteq \pi_{H}^{-1}[O]$.
\end{clm}

\begin{clmproof}
Suppose not. Then there are nets $(r_i)_{i \in I}$ in $R$ and $(s_i)_{i \in I}$ with $s_i \in I_i$ such that $r_is_i \notin \pi_{H}^{-1}[O]$ for all $i \in I$ or $s_ir_i \notin \pi_{H}^{-1}[O]$ for all $i \in I$. By compactness of $R$, passing to subnets, without loss of generality we can assume that the nets $(r_i)_{i \in I}$ and $(s_i)_{i \in I}$ converge to some $r$ and $s$, respectively. Then $rs=\lim r_is_i \notin \pi_{H}^{-1}[O]$ or $sr=\lim s_ir_i \notin \pi_H^{-1}[O]$. So $rs \ne 0$ or $sr \ne 0$. Hence, it is enough to show that $s \in R^0$, as then the last inequalities contradict the fact that $R \cdot R^0=R^0 \cdot R = \{0\}$.

Suppose for a contradiction that $s \notin R^0$. Then there is a compact neighborhood $V$ of $R^0$ such that $s \notin V$. By compacntess of $R$ and by the choice of the $I_i$'s, we get that there is $i_0 \in I$ such that  $I_i \subseteq V$ for all $i >i_0$. Then for all $i>i_0$  we have $s_i \in V$, and so $s = \lim_i s_i \in V$, a contradiction. 
\end{clmproof}

Pick $i$ from the claim. Then $R\cdot I_i/H \subseteq O$ and $I_i\cdot R/H \subseteq O$. On the other hand, since for every $r \in R$ the functions $l_r \colon R \to R/H$ and $r_r \colon R \to R/H$ given by $l_r(x):=rx/H$ and $r_r(x)=xr/H$ are group homomorphisms, we have that $l_r[I_i]$ and $r_r[I_i]$ are subgroups of $R/H$. By the preceding sentence, they are also contained in $O$, so they must be trivial by the choice of $O$. Hence, $R \cdot I_i \cup I_i \cdot R \subseteq H$. Therefore, replacing $R$ by the open ideal $I_i$ and $H$ by $H \cap I_i$, we get that $R \cdot R \subseteq H$, so $H$ is an ideal in $R$ and we are done.
\end{proof}

\begin{corollary}\label{corollary: Gleason-Yamabe for loc. compact rings}
Let $R$ be a locally compact ring and $U \subseteq R$ a neighborhood of $0$. Then there exists an open subring $S$ of $R$ and a compact ideal $J \lhd S$ contained in $U$ such that the additive group of $S/J$ is a connected Lie group of the form $\mathbb{R}^n \times C$ for a compact connected Lie group $C$.
\end{corollary}

A proof of Corollary \ref{corollary: Gleason-Yamabe for loc. compact rings} can be easily extracted from the proof of Proposition \ref{proposition: improved target space 2} below which uses Proposition \ref{proposition: Gleason-Yamabe for compact rings}.

\begin{proposition}\label{proposition: improved target space 2}
Let $X$ be a definable (in $M$) $K$-approximate subring. Then there exists a definable $(K^{510}+K^{22})$-approximate subring $Y \subseteq 4X +X \cdot 4X \subseteq X_3$ commensurable to $X$ and a definable locally compact model $f \colon \langle Y \rangle \to \mathcal{A}$ with dense image, where $\mathcal{A}$ is a connected locally compact ring whose additive group is a Lie group of the form $\mathbb{R}^n \times C$ for a connected compact Lie group $C$, and with $f^{-1}[U]\subseteq Y$ for some neighborhood $U$ of $0$ in $\mathcal{A}$.  
Moreover, if $M$ is at least $(|\mathcal{L}| + \aleph_0)^+$-saturated, then $f$ can be chosen surjective, and whenever $X$ is definable over $M_0$ (as an approximate subring) where $M_0 \prec M$ is of cardinality at most $|\mathcal{L}| + \aleph_0$, then $f$ can be chosen to be surjective and $M_0$-definable and $Y$ also $M_0$-definable.
\end{proposition}

\begin{proof}
By Fact \ref{fact: locally compact model exists}, there is a definable locally compact model $f \colon \langle X \rangle\to S$ of $X$ with $f^{-1}[U] \subseteq 4X + X \cdot 4X$ for some neighborhood $U$ of $0$ in $S$. Replacing $S$ by the closure of the image, we get that $f$ has dense image.

As a locally compact abelian group $(S,+)$ is topologically isomorphic to $\mathbb{R}^n \times D$, where $D$ is a locally compact group with compact connected component $P$ at $0$ (e.g., see \cite[Theorem 4.2.1]{DeEc}). Identify $S$ with $\mathbb{R}^n \times D$, and equip  $\mathbb{R}^n \times D$ with the induced topological ring structure.  Since $\mathbb{R}^n \times P$ is the connected component of $0$ in $S$, we know that it is a closed (two-sided) ideal, so $S/(\mathbb{R}^n \times P)$ is a locally compact ring.

We have the obvious topological groups  isomorphism $\rho \colon S/(\mathbb{R}^n \times P) \to D/P$ given by $\rho((x,y)/(\mathbb{R}^n \times P)):=y/P$. Since $D/P$ is totally disconnected, van Dantzig theorem (e.g., see \cite[Proposition 4.1.6(a)]{DeEc}) tells us that that it has a compact open subgroup. 
So let $H$ be a compact open subgroup of $S/(\mathbb{R}^n \times P)$. Then there is a compact open subring $I$ of $S/(\mathbb{R}^n \times P)$ contained in $H$ (e.g. see the first paragraph of the proof of Theorem 5.2 in \cite{Kru}).

Let $\pi \colon S \to S/(\mathbb{R}^n \times P) $ and $\pi_2 \colon D \to D/P$ be the quotient maps. We clearly have
$$\pi^{-1}[I] = \mathbb{R}^n \times \pi_2^{-1}[\rho[I]].$$
Put $A:=\pi_2^{-1}[\rho[I]]$.
Since $\rho[I]$ and $P$ are compact, $A$ is also compact. Since $\rho[I]$ is open, we also have that $A$ is open.


The set $U$ from the first paragraph of the proof can be decreased to be of the form $U_1 \times U_2$ for some neighborhoods $U_1$ and $U_2$ of $0$ in $\mathbb{R}^n$ and $A$, respectively.

\begin{clm}
$\{0\} \times A$ is an ideal of $\pi^{-1}[I]=\mathbb{R}^n \times A$.
\end{clm}

\begin{clmproof}
It follows from the fact that $\{0\} \times A$ is the largest (in the sense of inclusion) compact subgroup of  $\pi^{-1}[I]$.
\end{clmproof}

By Claim 1, $A$ has a compact ring structure induced by the identification with $\{0\} \times A$.
By Proposition \ref{proposition: Gleason-Yamabe for compact rings}, there is an open subring $A'$ of $A$ and its closed ideal $B$ contained in $U_2$ such that the additive group of $A'/B$ is a compact connected Lie group. 

\begin{clm}
$\mathbb{R}^n \times A'$ is an open subring of $\pi^{-1}[I]$ (so also of $S$), and $\{0\} \times B$ is its compact ideal contained in $U_1 \times U_2$.
\end{clm}

\begin{clmproof}
For the first part it is enough to show that $\mathbb{R}^n \times A'$ is closed under multiplication. First of all, $\{0\} \times A'$ is a subring of $\mathbb{R}^n \times A$, so it is closed under $\cdot$. Secondly,  $(\mathbb{R}^n \times \{0\}) \cdot (\mathbb{R}^n \times \{0\})$ is connected, so contained in the open subgroup $\mathbb{R}^n \times A'$. Thus, the conclusion follows using Fact \ref{fact: Thm. 1 from Kaplansky}, because, by this fact, we have $(\mathbb{R}^n \times \{0\}) \cdot (\{0\} \times A') = (\{0\} \times A') \cdot (\mathbb{R}^n \times \{0\}) =\{(0,0)\}$. 

The second part follows because $(\mathbb{R}^n \times \{0\})\cdot (\{0\} \times B)=(\{0\} \times B) \cdot (\mathbb{R}^n \times \{0\})=\{(0,0)\}$ and $(\{0\} \times A') \cdot (\{0\} \times B) \cup (\{0\} \times B) \cdot (\{0\} \times A') \subseteq  \{0\} \times B$ (the last inclusion holds since  $B$ is an ideal in $A'$).
\end{clmproof}

Now, applying Fact \ref{fact: passing to a connected quotient} and Lemma \ref{lemma: K510}, we get that $Y:= (4X + X \cdot 4X) \cap f^{-1}[\mathbb{R}^n \times A']$ is a definable $(K^{510}+K^{22})$-approximate subring commensurable to $X$ and there is a definable locally compact model $f' \colon \langle Y \rangle \to (\mathbb{R}^n \times A')/(\{0\} \times B)$ with dense image and with $f'^{-1}[V] \subseteq Y$ for some neighborhood $V$ of $0$, where the additive group of the target is topologically  isomorphic to $\mathbb{R}^n \times A'/B$ and $A'/B$ is a compact connected Lie group.


For the moreover part, consider any $M_0 \prec M$ of cardinality $|\mathcal{L}| + \aleph_0$ over which $X$ is a definable approximate subring. Apply the first part of the proposition to $M_0$ and $X(M_0)$ in place of $M$ and $X$, respectively, to produce a definable  locally compact model $f \colon \langle Y(M_0) \rangle \to \mathcal{A}$ with dense image for some definable in $M_0$ $(K^{510}+K^{22})$-approximate subring $Y(M_0) \subseteq 4X(M_0) +X(M_0) \cdot 4X(M_0)$ commensurable to $X(M_0)$, satisfying $f^{-1}[U] \subseteq Y(M_0)$ for some neighborhood $U$ of $0$ in $\mathcal{A}$, where $\mathcal{A}$ is as in the statement. Let $Y:=Y(M)$. By Fact \ref{fact: extension to the monster}, $f$ extends uniquely to an $M_0$-definable locally compact model $\bar f \colon \langle Y \rangle \to S$, which satisfies $\bar{f}^{-1}[U] \subseteq Y$ by the second paragraph of the proof of Fact \ref{fact: extension to the monster}. This unique $\bar f$ is onto by Proposition \ref{proposition: model is onto}. Since $Y(M_0) \subseteq 4X(M_0) +X(M_0) \cdot 4X(M_0)$, we have $Y \subseteq 4X+ X \cdot 4X$, and $Y$ is commensurable to $X$ as $Y(M_0)$ is commensurable to $X(M_0)$. 
\end{proof}



\begin{proposition}\label{proposition: improved target space}
The version of Proposition \ref{proposition: improved target space 2} with the stronger conclusion that the target space $\mathcal{A}$ is a finite-dimensional real algebra, but with the property $f^{-1}[U] \subseteq Y$ replaced by the weaker condition that $f^{-1}[U] \subseteq Y_m$ for some $m$, remains valid. 

In fact, whenever one has a definable locally compact model $f \colon \langle Y \rangle \to \mathcal{A}$ [with dense image] of a definable approximate subring $Y$, where $\mathcal{A}$ is a locally compact ring whose additive group is of the form $\mathbb{R}^n \times C$ for a compact group $C$, and with $f^{-1}[U]\subseteq Y$ for some compact  neighborhood $U$ of $0$ in $\mathcal{A}$, then $\{0\} \times C$ is an ideal in $\mathbb{R}^n \times C$, the induced ring structure on $\mathbb{R}^n$ identified with $(\mathbb{R}^n \times C)/(\{0\} \times C)$ turns it into an $n$-dimensional real algebra, and the composition $\pi_1 \circ f \colon \langle Y \rangle \to \mathbb{R}^n $ (where $\pi_1 \colon \mathbb{R}^n \times C \to \mathbb{R}^n$ is the projection) is  a definable locally compact model [with dense image] with target space being an $n$-dimensional real algebra (which can be identified with a subalgebra of the matrix algebra $M_{n+1}(\mathbb{R})$ with the inherited topology) and with $(\pi_1 \circ f)^{-1}[V] \subseteq Y_m$ for some $m$ where $V:=\pi_1[U]$ is a compact neighborhood of $0$ in $\mathbb{R}^n$. Moreover, if $M$ is at least $(|\mathcal{L}| + \aleph_0)^+$-saturated, and $f$ is $M_0$-definable [surjective] for some $M_0 \prec M$ of cardinality at most $|\mathcal{L}| + \aleph_0$ over which $Y$ is a definable approximate subring, then so is $\pi_1 \circ f$.
\end{proposition}

\begin{proof}
By Claim 1 in the proof of Proposition \ref{proposition: improved target space 2}, $\{0\} \times C$ is an ideal in $\mathbb{R}^n \times C$, so we indeed obtain the induced topological ring structure on $\mathbb{R}^n$. Since the obtained ring multiplication on $\mathbb{R}^n$ is continuous, it is easy (see \cite[Lemma 2]{JaTa}) that it  turns $\mathbb{R}^n$ into a (topological) real algebra of dimension $n$. As such, it is topologically isomorphic to a subalgebra of $M_{n+1}(\mathbb{R})$ (first, adjoining $1$ to our algebra, we obtain a topological real unital algebra structure on $\mathbb{R}^{n+1}$, and then we embed it into $M_{n+1}(\mathbb{R})$ via $v \mapsto A_v$, where $A_v(x):=v \cdot x$ for $x \in \mathbb{R}^{n+1}$). 

Since $\ker(\pi_1)=\{0\} \times C$ is compact, $\pi_1 \circ f$ is a definable locally compact model by Remark \ref{remark: pi f is a model} [which has dense image if $f$ had]. We have $(\pi_1 \circ f)^{-1}[V] =f^{-1}[U\cdot (\{0\} \times C)]$, and since $U\cdot (\{0\} \times C)$ is a compact, its preimage is indeed contained in some $Y_m$ by Remark \ref{remark: premimages of compact sets}. 

Finally, $M_0$-definability of $\pi_1 \circ f$ in the moreover part is obtained by noticing that in the proof of Remark \ref{remark: pi f is a model} the set $D$ at the end can be found to be $M_0$-definable.
\end{proof}

\begin{corollary}\label{corollary: improved Cor. 3.2}
Any definable approximate subring $X$ is commensurable with a definable approximate subring $Z\subseteq 4X + X \cdot 4X$ which is closed under multiplication (i.e., $Z \cdot Z \subseteq Z$). 
\end{corollary}

\begin{proof}
By Proposition \ref{proposition: improved target space 2}, we can find a commensurable to $X$ definable approximate subring $Y \subseteq 4X + X \cdot 4X$ with a definable locally compact model 
$f \colon \langle Y \rangle \to \mathcal{A}$, where $\mathcal{A}$ is a connected locally compact ring whose additive group is of the form $\mathbb{R}^n \times C$ for a compact connected Lie group $C$, such that $f^{-1}[U \times V] \subseteq Y$ for some neighborhoods $U$ and $V$ of $0$ in $\mathbb{R}^n$ and $C$, respectively.

As in the proof of Proposition \ref{proposition: improved target space}, $\{0\} \times C$ is an ideal, and $\mathbb{R}^n$ naturally identified with $(\mathbb{R}^n \times C)/(\{0\} \times C)$ becomes an $n$-dimensional real algebra with multiplication denoted by $*$, so can be embedded in matrices and considered with the Frobenius norm. Note, however, that $\mathbb{R}^n \times \{0\}$ need not be closed under multplication in the ring $\mathbb{R}^n \times C$, i.e. may produce nonzero elements on the second coordinate. By Fact \ref{fact: Thm. 1 from Kaplansky}, $\{0\} \times C$ is null in $\mathbb{R}^n \times C$, hence 
$$(*)\;\;\;\;\;\; (a_1,c_1) \cdot (a_2,c_2) = (a_1,0) \cdot (a_2,0) \in \{a_1*a_2\} \times C \textrm{ for any } a_1,a_2 \in \mathbb{R}^n \textrm{ and } c_1,c_2 \in C.$$

Pick a compact symmetric neighborhood $V' \subseteq \textrm{int}(V)$ of $0$ in $C$. By continuity of multiplication, there is a neighborhood $U' \subseteq U$ of $0$ in $\mathbb{R}^n$ such that $(U' \times \{0\}) \cdot (U' \times \{0\}) \subseteq \mathbb{R}^n \times V'$. 

Now, pick $r \in (0,\frac{1}{2})$ so small that $B(0,r) \subseteq U'$, where $B(0,s)$ denotes the closed ball centered at $0$ and of radius $s$ with respect to the Frobenius norm on $\mathbb{R}^n$ mentioned above.

Using definability of $f$, choose  a definable symmetric set $Z$ between $f^{-1}[B(0,\frac{1}{2}r) \times V']$ and $f^{-1}[B(0,r) \times V]$. By Fact \ref{fact: D is approximate}, $Z$ is an approximate subring commensurable with $X$. It is clear that $Z \subseteq Y$.
Moreover, by $(*)$, submultiplicativity of the Frobenius norm (see Fact \ref{fact: submultiplicativity}), and the choice of $r$, we have $(B(0,r) \times V)\cdot (B(0,r) \times V) \subseteq B(0,r^2) \times V' \subseteq B(0,\frac{1}{2}r) \times V'$. Therefore, $Z \cdot Z \subseteq f^{-1}[B(0,r) \times V]\cdot f^{-1}[B(0,r) \times V] \subseteq f^{-1}[(B(0,r) \times V)\cdot (B(0,r) \times V)]\subseteq f^{-1}[B(0,\frac{1}{2}r) \times V'] \subseteq Z$.
\end{proof}

The final goal of this section is to give a logarithmic bound on $n$ in Propositions \ref{proposition: improved target space 2} and \ref{proposition: improved target space}.
The following lemma follows from the proof of \cite[Corollary 7.8]{Dries}

\begin{lemma}\label{lemma: Lou's 7.8}
Let $X$ be a definable $K$-approximate subgroup and $f \colon \langle X \rangle \to G$ a definable locally compact model, where  $G$ is a locally compact abelian group such that there is a neighborhood $U$ of the neutral element in $G$ with $f^{-1}[U] \subseteq X^4$ and $\im(f) \cap U$ dense in $U$. Suppose $T$ is a closed subgroup of $G$ such that $G/T$ is a connected Lie group without non-trivial compact subgroups. Then $d:=\dim_{\textrm{Lie}}(G/T) \leq 2 \log_2 (K)$.
\end{lemma}

\begin{proof}
Passing to $\C \succ M$ at least $|M|^+$-saturated, and taking the unique extension $\bar f \colon \langle \bar X \rangle \to G$ from the obvious counterpart of Fact \ref{fact: extension to the monster} for approximate subgroups, we get (using the obvious counterpart of Proposition \ref{proposition: model is onto}) that $\bar f[\bar X^4]$ is closed. 
On the other hand, by assumption, $f[X^4]$ is a dense subset of $U$ which implies that $\bar f[\bar X^4] \cap U$ is dense in $U$. Thus, $U \subseteq \bar f[\bar X^4]$. 
This context is appropriate to apply the proof of Corollary 7.8 from \cite{Dries}. 
(By the way, it seems to us that the bound $3\log_2(K)$ is not really proved in \cite{Dries}. In the proof of \cite[Corollary 7.8]{Dries}, it is not clear why $\pi(X)^2 \cap \mathcal{G}'$ is a neighborhood of the identity. One should use the fact that $\pi(X)^4 \cap \mathcal{G}'$ is a neighborhood, and this leads to the bound  $6\log_2(K)$.) 

Let us give the details in our context. $f[\bar X]^2$ is compact by the obvious counterpart of Proposition \ref{proposition: model is onto}, and so is $C:=f[\bar X]^2/T$ as a subset of $G/T$. Hence, by \cite[Lemma 7.7]{Dries}, $\mu(C^2) \geq 2^d \mu(C)$, where $\mu$ is a Haar measure on $G/T$. 
On the other hand, since $\bar{X}^2$ is a $K^2$-approximate subgroup, so is $f[\bar X]^2$, and so is $C$. Thus, $\mu(C^2) \leq K^2\mu(C)$. Moreover, since $U/T \subseteq C^2$ and $U/T$ has a nonempty interior, $\mu(C^2) >0$. Therefore, $\mu(C) >0$. Thus, we conclude that $d \leq 2 \log_2(K)$.   
\end{proof}

\begin{remark}\label{remark: X +XX is approximate}
Let $X$ be a $K$-approximate subring of $(R,+,\cdot)$. Then $X+ X  \cdot X$ is a $K^2$-approximate subgroup of $(R,+)$.
\end{remark}

\begin{proof}
By assumption $X \cup (X\cdot X) \subseteq F + X$ for some $F$ of size at most $K$. Then $X + X\cdot X \subseteq  X + F +X  \subseteq 2F+X$, and clearly $|2F| \leq K^2$. The fact that $X + X \cdot X$ is additively symmetric is clear.
\end{proof}

\begin{corollary}\label{corollary: small dimension of the algebra}
If $X$ is a definable $K$-approximate subring, then the number $n$ in Proposition \ref{proposition: improved target space 2} is bounded by $4 \log_2 (K)$. Thus,
in Proposition \ref{proposition: improved target space}, we can additionally require that $\dim(\mathcal{A}) \leq 4 \log_2 (K)$ (where $\mathcal{A}$ is the finite-dimensional real algebra being the target space of the obtained locally compact model).
\end{corollary}

\begin{proof}
Take $f$ and all the notation as in the proof of Proposition \ref{proposition: improved target space 2}. Let $Z:=X+X \cdot X$. By Remark \ref{remark: X +XX is approximate}, $Z$ is a definable $K^2$-approximate subgroup. Moreover, $f|_{\langle Z \rangle_+} \colon \langle Z \rangle_+ \to (S,+)$ is a definable locally compact model of $Z$ (where $\langle Z \rangle_+$ denotes the additive subgroup generated by $Z$) with $f^{-1}[U] \subseteq 4X + X \cdot 4X \subseteq 4Z$, and $f[4Z]$ is dense in $U$. As $T$ take $\{0\} \times D$. Then all the assumptions of Lemma \ref{lemma: Lou's 7.8} are satisfied, so $n=\dim(\mathbb{R}^n) = \dim_{\textrm{Lie}} (S/T) \leq 2 \log_2(K^2) = 4 \log_2 (K)$ (and clearly $\mathbb{R}^n$ is the target algebra of the model produced in Proposition \ref{proposition: improved target space}).
\end{proof}




\section{Escape norm}\label{section: escape norm}

The goal of this section is to show that for a given pseudofinite approximate subring, after passing to a suitable commensurable definable approximate subring, the escape norm satisfies certain good inequalities.

See Appendix \ref{appendix: pseudofiniteness} for the definition and formal treatment of the notion of pseudofinite approximate subring.

Let $X$ be a pseudofinite approximate subring (definable in a model $M=(G,\mathbb{R}^*)$ of $T_0$ in the notation from Appendix \ref{appendix: pseudofiniteness}). For every element $r \in \langle X \rangle^*$ (see Corollary \ref{corollary: <X>*}), we define the {\em escape norm of $r$ with respect to $X$} by:
$$||r||_X:= \inf\left\{\frac{1}{\nu +1}: \nu \in \mathbb{N}^* \textrm{ and } i r \in X \textrm{ for all } i \in \{0,\dots, \nu\}\right\}.$$

Note that $||r||_X$ belongs to the interval $[0,1]^*$ in $\mathbb{R}^*$, and:
$$||r||_X \leq \frac{1}{2} \iff r \in X,$$
$$||r||_X =1 \iff r \notin X,$$
$$||r||_X =0 \iff \nu r \in X \textrm{ for all } \nu \in \mathbb{N}^*.$$

Note also that if $||r||_X \ne 0$, then then infimum in the definition of $||r||_X$ is attained, i.e. instead of infimum we can write minimum. (This is because this property holds in every model from the class $\mathcal{K}$ from Appendix \ref{appendix: pseudofiniteness}.)

Note that the language $\mathcal{L}$ is countable in the present context, so $(\mathcal{L} + \aleph_0)^+$-saturation is just $\aleph_1$-saturation, and there exists a countable model $M_0 \prec M$. In all statements below, $M_0$ will be be countable.

Before we turn to the main result of this section, we need the following observation.

\begin{lemma}\label{lemma: escape norm in Lie groups}
Let $(G,+,0)$ be an abelian Lie group. Then for any neighborhood $V$ of $0$ there are symmetric neighborhoods $V_1 \subseteq V_2 \subseteq V$ of $0$ such that $V_1$ is compact, $V_2$ is open, and for 
any $g,h \in G$ such that $g,h,4g,4h \in V_2$ we have $g+h \in V_1$.
\end{lemma}

\begin{proof}
Let $\exp \colon \mathfrak{g} \to G$ be the exponential map of the Lie algebra $\mathfrak{g}$ of $G$. Let $B_r$ denote the closed ball centered at $0$ of radius $r$ in $\mathfrak{g}$ with respect to some vector space norm on $\mathfrak{g}$. Take $r>0$ so small that $\exp |_{\textrm{int}(B_{8r})}$ is a diffeomorphism onto an open neighborhood of $0$ in $G$ and $\exp[B_r] \subseteq V$. Then $V_1:= \exp[B_{\frac{1}{2}r}]$ and $V_2:=\exp[\textrm{int}(B_r)]$ work. To see this, consider any $g,h \in G$ such that $g,h,4g,4h \in V_2$.
Then $4g =\exp(x)$, $4h =\exp(y)$, $g=\exp(z)$, $h=\exp(t)$ for some $x,y,z,t \in B_r$. Thus, $\exp(4z)=4g=\exp(x)$ and $\exp(4t)=4h=\exp(y)$, so $x=4z$ and $y=4t$. Hence, $g+h=\exp(\frac{1}{4}x+\frac{1}{4}y) \in \exp[B_{\frac{1}{2}r}] =V_1$.
\end{proof}

\begin{proposition}\label{corollary: model with the same target space}

Let $X$ be a pseudofinite approximate subring which has a definable locally compact model $f\colon  \langle X \rangle \to \mathcal{A}$ with dense image [surjective if $M$ is $\aleph_1$-saturated], with target space being a connected locally compact ring whose additive group is a Lie group of the form $\mathbb{R}^n \times C$ for a compact connected Lie group $C$, and such that $f^{-1}[W] \subseteq X$ for some neighborhood $W$ of $0$ in $\mathcal{A}$. 

Then there is a definable approximate subring $Y \subseteq X$ commensurable with $Y$ satisfying $\langle Y \rangle =\langle X \rangle$ and $Y^2 \subseteq Y$, such that $f \colon \langle Y \rangle \to \mathcal{A}$ is a definable locally compact model with dense image [surjective if $M$ is $\aleph_1$-saturated, in which case if we fix a countable $M_0 \prec M$ such that $X$ and $f$ are $M_0$-definable, we can also require that $Y$ is $M_0$-definable], with $f^{-1}[U]\subseteq Y$ for some neighborhood $U$ of $0$ in $\mathbb{R}^n \times C$, and for which the escape norm $|| \cdot ||_Y$ satisfies:
\begin{enumerate}
\item $||x+y||_Y \leq 4\max(||x||_Y,||y||_Y) \leq 4(||x||_Y+||y||_Y)$,
\item if $x,y \in Y$, then $||xy||_Y \leq 2||x||_Y\cdot ||y||_Y$,
\item if $||x||_Y =0$ or $||y||_Y =0$, then $||xy||_Y=0$,
\end{enumerate}
for all $x,y \in \langle Y \rangle^*$. Moreover, $I:=\{y \in \langle Y \rangle^*: ||y||_Y =0\}$ is a definable [$M_0$-definable if $M$ is $\aleph_1$-saturated and $Y$ is $M_0$-definable] (two sided) ideal of $\langle Y \rangle^*$ which is contained in $\ker(f)$. 
Furthermore, $Y$ can be chosen so that $\cl(f[Y])$ contains no subgroup other than $\{0\}$. In fact, for any given $k \in \mathbb{N}$, $Y$ can be chosen so that $\cl(f[Y_k])$ contains no subgroup other than $\{0\}$.
\end{proposition}

\begin{proof}
We produce $Y$ from $X$ and $f$ in a similar way to how $Z$ was obtained from $Y$ and $f$ in the proof of Corollary \ref{corollary: improved Cor. 3.2}, but the choice of $V'$ and $r$ needs to be done more carefully. 

By assumption there exist neighborhoods $U$ and $V$ of $0$ in $\mathbb{R}^n$ and $C$, respectively, such that $f^{-1}[U \times V] \subseteq X$.
Since $\mathbb{R}^n \times C$ is a Lie group, we can find a compact neighborhood $O$ of $0$ in $\mathbb{R}^n \times C$ which contains no subgroup other than $\{0\}$. Shrink $U$ and $V$ to guarantee that $U \times V \subseteq O$ or even $(U \times V)_k \subseteq O$ for a given $k \in \mathbb{N}$.

Applying Lemma \ref{lemma: escape norm in Lie groups} to $G:=C$ and $V:=V$, we obtain $V_1 \subseteq V_2 \subseteq V$ as in the lemma.
There is a neighborhood $U' \subseteq U$ of $0$ in $\mathbb{R}^n$ such that $(U' \times \{0\}) \cdot (U' \times \{0\}) \subseteq \mathbb{R}^n \times V_1'$, where $V_1'$ a neighborhood of $0$ in $C$ such that   $V_1'+V_1' \subseteq V_1$.

Pick $r \in (0,\frac{1}{4})$ so small that $B(0,r) \subseteq U'$, where $B(0,s)$ denotes the closed ball centered at $0$ and of radius $s$ with respect to the Frobenius norm on $\mathbb{R}^n$ (see the proof of Corollary \ref{corollary: improved Cor. 3.2}).

Now, choose a definable [$M_0$-definable if $M$ is $\aleph_1$-saturated and $X$ and $f$  were assumed to be $M_0$-definable] symmetric set $Y$ between $f^{-1}[B(0,\frac{1}{2}r) \times V_1]$ and $f^{-1}[B(0,r) \times V_2]$. As in the proof of Corollary \ref{corollary: improved Cor. 3.2}, $Y \subseteq X$ is a definable 
[resp. $M_0$-definable]
approximate subring commensurable to $X$ and satisfying $Y^2 \subseteq Y$. By Fact \ref{fact: D is approximate}, $f|_{\langle Y \rangle} \colon \langle Y \rangle \to \mathcal{A}$ is a definable
[resp. $M_0$-definable] 
locally compact model of $Y$. We clearly have that $W' := \textrm{int}(B(0,\frac{1}{2}r)) \times \textrm{int}(V_1)$ is an open  neighborhood of $0$ with $f^{-1}[W'] \subseteq Y$. Hence, $\langle Y \rangle = \langle X \rangle$ by Corollary \ref{corollary: <Y>=<X>}. This implies that $f|_{\langle Y \rangle}$ has dense image [is surjective if $f$ was surjective].


The fact that $\cl(f[Y])$ contains no subgroup other than $\{0\}$ follows from the inclusion $\cl(f[Y]) \subseteq O$. More generally, the fact that $\cl(f[Y_k])$ contains no subgroup other than $\{0\}$ follows from the inclusions $\cl(f[Y_k]) =\cl(f[Y]_k)=(\cl(f[Y]))_k \subseteq (\cl(U \times V))_k \subseteq \cl((U \times V)_k) \subseteq \cl(O)=O$ (which hold by continuity of ring operations and compactness of $\cl[Y]$).

It remains to justify (1), (2), (3), and the moreover part. 

(1) It is enough to show that for every $\nu \in \mathbb{N}^*$, if $\max(||x||_Y, ||y||_Y) \leq \frac{1}{\nu+1}$, then $||x+y||_Y \leq \frac{4}{\nu+1}$.
So assume that $\max(||x||_Y, ||y||_Y) \leq \frac{1}{\nu+1}$. Then $nx,ny \in Y$ for every $n \leq \nu$ in $\mathbb{N}^*$. Consider any $n \leq \nu$ in $\mathbb{N}^*$ which is divisible by $4$. Then $nx,ny,\frac{n}{4}x,\frac{n}{4}y \in Y$. Hence, $f(nx),f(ny),f(\frac{n}{4}x), f(\frac{n}{4}y) \in B(0,r) \times V_2$. So, since $f(nx)=4f(\frac{n}{4}x)$ and  $f(ny)=4f(\frac{n}{4}y)$, using the choice of $V_1$ and $V_2$, we have $f(\frac{n}{4}(x+y))=f(\frac{n}{4}x) + f(\frac{n}{4}y) \in B(0,\frac{1}{2}r) \times V_1$, hence $\frac{n}{4}(x+y) \in Y$.

Let $\nu' \in \{ \nu,\nu-1,\nu-2,\nu-3\} \cap \mathbb{N}^*$ be divisible by $4$.  By the above paragraph, $||x+y||_Y \leq \frac{1}{\frac{\nu'}{4}+1}$. On the other hand, by the choice of $\nu'$, we have $\frac{1}{\frac{\nu'}{4}+1} \leq \frac{4}{\nu+1}$.

(2) It is enough to show that if $||x||_Y \leq \frac{1}{\mu +1}$ and $||y||_Y \leq \frac{1}{\nu+1}$ for some $\mu,\nu \geq 1$ in $\mathbb{N}^*$, then $||xy||_Y \leq \frac{2}{(\mu+1)(\nu +1)}$. 

Consider any $m\leq \mu$ and $n \leq \nu$. Then $mx,ny \in Y$, so $f(mx),f(ny) \in B(0,r) \times V_2$. In general, we have $(mx) \cdot  (ny)=(mn)(xy)$ (because this holds in the structures from the class $\mathcal{K}$ from  Appendix \ref{appendix: pseudofiniteness} working with standard natural numbers $m,n$). Hence, $(mn)(xy) \in Y^2 \subseteq \textrm{dom}(f)$, and,  by the choice of $r$, by $(*)$ in the proof of Corollary \ref{corollary: improved Cor. 3.2}, and by submultiplicativity of the Frobenius norm, we get $f((mn)(xy))=f((mx)(ny)) = f(mx)f(ny) \in (B(0,r) \times V_2)\cdot (B(0,r) \times V_2) \subseteq B(0,r^2) \times V_1' \subseteq B(0,\frac{1}{4}r) \times V_1'$.

Consider any element $k \leq \mu \nu$ of $\mathbb{N}^*$. It can be written as a sum $m_1n_1 + m_2n_2$ with $m_1,m_2 \leq \mu$ and $n_1,n_2 \leq \nu$ (all from $\mathbb{N}^*$). By the previous paragraph,  $f(kxy) =f((m_1n_1)(xy) + (m_2n_2)(xy)) = f((m_1n_1)(xy)) + f((m_2n_2)(xy))\in B(0,\frac{1}{2}r) \times (V_1'+V_1') \subseteq B(0,\frac{1}{2}r) \times V_1$. Thus, $kxy \in Y$.

Therefore, $||xy||_Y \leq \frac{1}{\mu \nu +1} \leq \frac{2}{(\mu+1)(\nu +1)}$ (note that the last inequality requires the assumption that $\mu,\nu \geq 1$).

(3) and the ``moreover'' part.
%
It is clear that $I$ is a definable [resp. $M_0$-definable] subset of $Y$. By (1), $I$ is a subgroup of $R$, so $f[I]$ is a subgroup of $\mathcal{A}$ contained in $f[Y]$. Since $f[Y]$ does not contain any subgroup of $\mathcal{A}$ other than $\{0\}$, we conclude that $f[I]=\{0\}$, i.e. $I \subseteq \ker(f)$.

For (3) it is enough to prove that $I$ is an ideal of $\langle Y \rangle^*$ (note that, conversely, (1) and (3) imply that $I$ is an ideal).  We will show that $I$ is a left ideal (the right version is symmetric). 
By (1), $I$ is a subgroup. So the set $Z:=\{x \in \langle Y \rangle^*: xI \subseteq I\}$ is a definable subring. Hence, in order to show that $Z= \langle Y \rangle^*$ (which is our goal), it is enough to prove that $\langle Y \rangle \subseteq Z$ (by virtue of Corollary \ref{corollary: <X>*}). 

So pick $y \in I$ and $x \in \langle Y \rangle$. Then $x \in Y_m$ for some $m$, hence $f(x) \in f[Y_m]$ which is relatively compact in $\mathcal{A}$ so contained in $B(0,N) \times C$ for some $N \in \mathbb{N}$. Choose $N'\geq N$ such that $(B(0,N) \times \{0\}) \cdot (B(0,\frac{1}{2N'}r) \times \{0\}) \subseteq \mathbb{R}^n \times V_1$. Consider any $n\in \mathbb{N}^*$. Since $||y||_Y=0$, we get that $2N'ny \in Y$, so $2N'f(ny) \in B(0,r) \times C$, hence $f(ny) \in B(0,\frac{1}{2N'}r) \times C$. Therefore,  by $(*)$ in the proof of Corollary \ref{corollary: improved Cor. 3.2} and by submultiplicativity of the Frobenius norm, $f(nxy)=f(x)f(ny) \in (B(0,N) \times C) \cdot (B(0,\frac{1}{2N'}r) \times C) \subseteq B(0,\frac{1}{2}r) \times V_1$. Thus, $nxy \in Y$. As $n\in \mathbb{N}^*$ was arbitrary, we conclude that $||xy||_Y=0$. i.e. $xy \in I$.
\end{proof}

\begin{corollary}\label{proposition: good escape norm  2}
Let $X$ be a pseudofinite $K$-approximate subring. 
Consider a definable $(K^{510}+K^{22})$-approximate subring $Y \subseteq 4X + X \cdot 4X$ and a definable locally compact model  $f \colon \langle Y \rangle \to \mathcal{A}$ with all the properties from Proposition \ref{proposition: improved target space 2}. In particular, the additive group of $\mathcal{A}$ is of the form $\mathbb{R}^n \times C$ for a compact connected Lie group $C$ and $n \leq 4\log_2(K)$ by  Corollary \ref{corollary: small dimension of the algebra}.

Then there is a definable approximate subring $Z \subseteq Y$ commensurable with $Y$ (equivalently with $X$) satisfying $\langle Z \rangle =\langle Y \rangle$ and $Z^2 \subseteq Z$, such that $f \colon \langle Z \rangle \to \mathcal{A}$ is a definable locally compact model with dense image [surjective if $M$ is $\aleph_1$-saturated, in which case if we fix a countable $M_0 \prec M$ such that $X$, $Y$ and $f$ are $M_0$-definable, we can also require that $Z$ is $M_0$-definable], with $f^{-1}[U]\subseteq Z$ for some neighborhood $U$ of $0$ in $\mathbb{R}^n \times C$, and for which the escape norm $|| \cdot ||_Z$ satisfies:
\begin{enumerate}
\item $||x+y||_Z \leq 4\max(||x||_Z,||y||_Z) \leq 4(||x||_Z+||y||_Z)$,
\item if $x,y \in Z$, then $||xy||_Z \leq 2||x||_Z\cdot ||y||_Z$,
\item if $||x||_Z =0$ or $||y||_Z =0$, then $||xy||_Z=0$,
\end{enumerate}
for all $x,y \in \langle Z \rangle^*$. Moreover, $I:=\{z \in \langle Z \rangle^*: ||z||_Z =0\}$ is a definable [$M_0$-definable if $M$ is $\aleph_1$-saturated and $Z$ is $M_0$-definable] (two sided) ideal of $\langle Z \rangle^*$ which is contained in $\ker(f)$. 
Furthermore, $Z$ can be chosen so that $\cl(f[Z])$ contains no subgroup other than $\{0\}$. In fact, for any given $k \in \mathbb{N}$, $Z$ can be chosen so that $\cl(f[Z_k])$ contains no subgroup other than $\{0\}$.
\end{corollary}

\begin{proof}
It follows from Proposition \ref{corollary: model with the same target space} applied to $Y$ in place of $X$.
\end{proof}




\begin{corollary}\label{corollary: good escape norm with finite dimensional algebra}
Let $X$ be a pseudofinite $K$-approximate subring. 
Consider a definable $(K^{510}+K^{22})$-approximate subring $Y \subseteq 4X + X \cdot 4X$ and a definable locally compact model  $f \colon \langle Y \rangle \to \mathcal{A}$ with all the properties from Proposition \ref{proposition: improved target space}. In particular, $\mathcal{A}$ is a finite-dimensional real algebra (of dimension at most $4\log_2(K)$ by  Corollary \ref{corollary: small dimension of the algebra}) and there is an open set $W$ and $m \in \mathbb{N}$ such that $f^{-1}[W] \subseteq Y_m$. Then we have the same conclusion as in Corollary \ref{proposition: good escape norm  2}, but with the property $Z \subseteq Y$ replaced by the weaker condition that $Z \subseteq Y_m$.
%
\end{corollary}

\begin{proof}
It follows from Proposition \ref{corollary: model with the same target space} applied to $Y_m$ in place of $X$.
\end{proof}


The next corollary follows from Proposition \ref{proposition: good escape norm 2} by model-theoretic compactness.

\begin{corollary}
Let $K \in \mathbb{N}$. There exists a constant $N(K)$ such that for every finite $K$-approximate subring there exists a finite $N(K)$-approximate subring $Y \subseteq 4X + X \cdot 4X$ which is $N(K)$-commensurable with $X$, satisfies $Y^2 \subseteq Y$, and such that the escape norm $|| \cdot ||_Y$ satisfies:
\begin{enumerate}
\item $||x+y||_Y \leq 4\max(||x||_Y,||y||_Y) \leq  4(||x||_Y+||y||_Y)$,
\item if $x,y \in Y$, then $||xy||_Y \leq 2||x||_Y\cdot ||y||_Y$,
\item if $||x||_Y =0$ or $||y||_Y =0$, then $||xy||_Y=0$,
\end{enumerate}
for all $x,y \in \langle Y \rangle$.
\end{corollary}

\section{Structure of finite approximate subrings}\label{section: structure of finite approximate rings}

In this section, we prove the main theorem on the structure of finite approximate subrings, i.e. Theorem \ref{theorem: main theorem intro}. The two regimes in Theorem \ref{theorem: main theorem intro} are split into Theorems \ref{theorem: structure of finite approximate rings 1} and \ref{theorem: structure of finite approximate rings 2} below.

The pseudofinite context with the notation from Appendix \ref{appendix: pseudofiniteness} will be present all the time without mentioning. The engine will be Proposition \ref{corollary: model with the same target space} and Corollaries \ref{proposition: good escape norm  2} and \ref{corollary: good escape norm with finite dimensional algebra}.

Let us recall a few useful lemmas from \cite{Dries}. The first lemma is \cite[Lemma 6.1]{Dries}.

\begin{lemma}\label{lemma: Lao 6.1}
Let $\mathcal{G}$ be a topological (Hausdorff) group and $V$ a compact symmetric subset of $\mathcal{G}$ which contains no subgroup of $\mathcal{G}$ other than $\{e\}$. Then, for every neighborhood $U$ of $e$ there is $n \geq 1$ such that for every $a \in \mathcal{G}$: if $a^i \in V$ for all $i \in\{0,\dots,n\}$, then $a \in U$.
\end{lemma}

\begin{proof}
Suppose not, i.e. there is a sequence $(a_n)_{n \in \mathbb{N}}$ of elements of $\mathcal{G} \setminus U$ such that $a_n^i \in V$ for all $i \in \{0,\dots,n\}$. Since $V$ is compact, there is a subnet $(b_j)_j \in J$ of $(a_n)_{n \in \mathbb{N}}$ which converges to some $a \in V$. Then $a \ne e$, as $U$ is a neighborhood of $e$. By continuity of multiplication and the choice of the $a_n$'s, we get that $a^i \in V$ for all $i \in \mathbb{N}$. Since $V$ is symmetric, $a^i \in V$ for all $i \in \mathbb{Z}$, i.e. $V$ contains the subgroup $\langle a \rangle$ which is nontrivial as $a \ne e$, a contradiction.
\end{proof}

The next observation is an extension of \cite[Lemma 6.7]{Dries}.

\begin{lemma}\label{lemma: Lou's 6.7}
Let $X$ be a pseudofinite approximate subring and $f \colon \langle X \rangle \to S$ a locally compact model with $f^{-1}[U] \subseteq X$ for some neighborhood $U$ of $0$ in $S$. If $X \subseteq Y \subseteq X_k$ with $Y$ definable and symmetric, and $\cl(f[Y])$ contains no subgroup of $(S,+)$ other than $\{0\}$, then there is $n \in \mathbb{N}$ such that for every $r \in \langle X \rangle^*$ we have $||r||_X \leq n||r||_{Y}$.
\end{lemma}

\begin{proof}
Since $f[Y] \subseteq f[X_k]$ is relatively compact, $\cl(f[Y])$ is compact and clearly symmetric. By Lemma \ref{lemma: Lao 6.1}, there is a (standard) natural number $n \geq 1$ such that for every $s \in S$, if $is\in f[Y]$ for all $i \in\{0,\dots,n\}$, then $s \in U$. It follows that if $r$ is such that $ir \in Y$ for all $i \in \{0,\dots,n\}$, then $r \in X$. This easily implies that $||r||_X \leq n||r||_{Y}$.
\end{proof}

The next lemma is  essentially \cite[Lemma 6.8]{Dries}.

\begin{lemma}\label{lemma: nondiscrete image}
Let $X$ be a pseudofinite approximate subring and $f \colon \langle X \rangle \to S$ a locally compact model with $f^{-1}[U] \subseteq X$ for some neighborhood of $0$ in $S$.
Assume that $\cl(f[X])$ contains no subgroup of $(S,+)$ other than $\{0\}$. Let $u \in X$ be such that $||u||_X =\frac{1}{N}$ for some $N \in \mathbb{N}^*$ with $N>\mathbb{N}$ (i.e. $||u||_X$ is a nonzero infinitesimal). Then:
\begin{enumerate}
\item the map from $\mathbb{R}$ to $(S,+)$ given by $t \mapsto f(\lfloor tN \rfloor u)$ is a continuous group homomorphism;
\item the image of  the map from (1) is a non-discrete space.
\end{enumerate}
\end{lemma}

\begin{proof}
(1) is proved in \cite{Dries}. So let us prove (2). By (1), $\lim_{k \to \infty}  f(\lfloor \frac{1}{k}N \rfloor u) =0$, so it is enough to show that $f(\lfloor \frac{1}{k}N \rfloor u) \ne 0$ for every positive integer $k$. Suppose for a contradiction that $f(\lfloor \frac{1}{k}N \rfloor u) =0$ for some $k$. Write $N$ as $kM +r$ for some $M \in \mathbb{N}^*$ and $r \in \{0,\dots,k-1\}$. Then $\frac{1}{k}N=M + \frac{r}{k}$, so $\lfloor \frac{1}{k}N \rfloor = M$. Thus, $f(Mu)=0$, and so $f((kM)u)=0$. Since $f[\mathbb{Z}u]$ is a subgroup of $f[X]$, we get that $f[\mathbb{Z}u]=\{0\}$, so $f(ru)=0$. Hence, $f(Nu)=f((kM + r)u)=0$. Thus, $Nu \in \ker(f) \subseteq X$, which contradicts the assumption that $||u||_X = \frac{1}{N}$.
%
\end{proof}

The following lemma will be crucial in the proof of the main theorem.

\begin{lemma}\label{lemma: passing to quotients}
Let $X$ be an $M_0$-definable approximate subring of a definable (in $M$) ring $R$ with a surjective $M_0$-definable locally compact model $f \colon \langle X \rangle \to S$ for some $M_0 \prec M$ such that $M$ is $|M_0|^+$-saturated. Let $D$ be an $M_0$-definable ideal of $R$. Let $\eta \colon R \to R/D$ be the quotient map and $\mathcal{D}:=\cl(f[D \cap \langle X \rangle])\subseteq S$.
Then $\mathcal{D}$ is a closed ideal of $S$, so $S/\mathcal{D}$ is a locally compact ring. $\eta[X]=X/D$ is an $M_0$-definable approximate subring of $R/D$ and $\hat{f} \colon \langle X/D \rangle \to S/\mathcal{D}$ given by $\hat{f}(\eta(x)):= f(x)/\mathcal{D}$ is a surjective $M_0$-definable locally compact model of $X/D$ with $\ker(\hat{f}) = \eta[\ker(f)]$. In particular, if $U$ is a neighborhood of $0$ in $S$ such that $f^{-1}[U] \subseteq X_n$, then $U/\mathcal{D}$ is a neighborhood of $0/\mathcal{D}$ in $S/\mathcal{D}$ and $\hat{f}^{-1}[U/\mathcal{D}] \subseteq X_n/D$.
\end{lemma}

\begin{proof}
The facts that  $\mathcal{D}$ is a closed ideal of $S$ (so $S/\mathcal{D}$ is a locally compact ring) and $\eta[X]=X/D$ is an $M_0$-definable approximate subring of $R/D$ are obvious.  It is also clear that $\hat{f}$ is a well-defined ring epimorphism. $\hat{f}[X/D]=f[X]/\mathcal{D}$ is relatively compact, because $f[X]$ is relatively compact (in fact compact by Proposition \ref{proposition: model is onto} and uniqueness in Proposition \ref{fact: extension to the monster}).

\begin{clm}
$\ker(\hat{f}) = \eta[\ker(f)]$.
\end{clm}

\begin{clmproof}
The inclusion $(\supseteq)$ is clear, so let us focus on $(\subseteq)$. Consider any $x\in \langle X \rangle$ with $x/D \in \ker(\hat{f})$, i.e. $f(x) \in \mathcal{D}$. We need to show that $x/D \in \eta[\ker(f)]$.

Take any $M_0$-definable $F$ with $\ker(f) \subseteq F\subseteq \langle X \rangle$ (there exists such an $F$ by $M_0$-definability of $f$). Put
$$U_F:=\{ s \in S: f^{-1}(s) \subseteq F\}.$$
Let us show that $U_F$ is a neighborhood of $0$. If not not, then for every compact neighborhood $V$ of $0$, $f^{-1}[V] \cap (R \setminus F) \ne \emptyset$.  By Remark \ref{remark: premimages of compact sets} and $M_0$-definability of $f$, each $f^{-1}[V]$ is $M_0$-type-definable. So (by $|M_0|^+$-saturation of $M$) the intersection of all such $f^{-1}[V]$ with $R \setminus F$ is nonempty. But this intersection coincides with $\ker(f) \cap (R \setminus F)$, which contradicts the choice of $F$.

As $U_F$ is a neighborhood of $0$, we get that $f(x) \in U_F + f[D \cap \langle X \rangle]$, so $f(x)=f(s)+f(y) =f(s+y)$ for some $s \in F$ and $y \in D \cap \langle X \rangle$. Then $x \in \ker(f) + F + (D \cap \langle X \rangle) \subseteq \ker(f) + F + D$. By $M_0$-definability of $f$, $\ker(f)$ is the intersection of all $M_0$-definable sets $F$ containing $\ker(f)$, so, by $|M_0|^+$-saturation of $M$, $x \in \ker(f) + \ker(f) + D = \ker(f) + D$. Thus, $x/D \in \eta[\ker(f)]$. 
\end{clmproof}

Now, take a neighborhood $U$ of $0$ in $S$ such that $f^{-1}[U] \subseteq X_n$. Then $U/\mathcal{D}$ is a neighborhood of $0/\mathcal{D}$ in $S/\mathcal{D}$, and we will show that $\hat{f}^{-1}[U/\mathcal{D}] \subseteq X_n/D$. Take any $s \in U$. Then there is $r\in X_n$ with $f(r)=s$. Consider any $x \in \langle X \rangle$ such that $\hat{f}(x/D)= s/\mathcal{D}$, i.e. $f(x) \in s+\mathcal{D}$. Then, $x \in f^{-1}[\mathcal{D}]+r$, so, by the claim, $x \in D + \ker(f) +r$ which is contained in $D +X_n$ (because $f[\ker(f)+r] = \{f(r)\} \subseteq U$ and $f^{-1}[U]\subseteq X_n$). Hence, $x/D \in X_n/D$ as required. 

For the definability of $\hat{f}$ consider any $C \subseteq U \subseteq S/\mathcal{D}$, where $C$ is compact and $U$ is open. Since $\hat{f}$ is a locally compact model, by Remark \ref{remark: premimages of compact sets}, $\hat{f}^{-1}[C] \subseteq X_m/D$ for some $m$. So there is $Y\subseteq X_m$ with $\eta[Y] =Y/D =\hat{f}^{-1}[C]$.  Let $\pi \colon S \to S/\mathcal{D}$ be the quotient map; so $\pi \circ f = \hat f \circ \eta$.

Then $\pi[\cl(f[Y])] \subseteq \cl(\pi[f[Y]]) = \cl(\hat{f}[\hat{f}^{-1}[C]])=\cl(C)=C \subseteq U$, so $\cl(f[Y]) \subseteq \pi^{-1}[U]$. On the other hand,  $\cl(f[Y]) \subseteq \cl(f[X_m])$ is compact and $\pi^{-1}[U]$ is open. Hence, by $M_0$-definability of $f$, there is an $M_0$-definable set $F$ between $f^{-1}[\cl(f[Y])]$ and $f^{-1}[\pi^{-1}[U]]$. 
Then $\hat{f}^{-1}[C] \subseteq F/D \subseteq \hat{f}^{-1}[U]$ and $F/D$ is $M_0$-definable, which shows $M_0$-definability of $\hat{f}$ by Remark \ref{remark: characterization of M-definability of a model}. Indeed, $\hat{f}^{-1}[C] =\eta[Y] \subseteq \eta[f^{-1}[\cl(f[Y])]] \subseteq F/D \subseteq \eta[f^{-1}[\pi^{-1}[U]] =\eta[\eta^{-1}[\hat{f}^{-1}[U]]]= \hat{f}^{-1}[U]$.
\end{proof}

We have all the tools to prove two versions of our main result. The proofs of both versions are almost the same. The second theorem requires a slight modification of the beginning of the proof of the first theorem. Everything is explained below.

\begin{theorem}\label{theorem: structure of finite approximate rings 1}
For any $K \in \mathbb{N}$ there exist $N_2(K),N_4(K) \in \mathbb{N}$ such that for every finite $K$-approximate subring $X$ there exists a $(K^{510}+K^{22})$-approximate subring $Y \subseteq 4X + X \cdot 4X$ which is $N_2(K)$-commensurable to $X$ for which there exists an ideal $I \lhd \langle Y \rangle$ contained in $Y$ such that $\langle Y \rangle/I$ is 
nilpotent of class at most $N_4(K)$.
\end{theorem}

\begin{theorem}\label{theorem: structure of finite approximate rings 2}
For any $K \in \mathbb{N}$ there exists $N_2(K),N_3(K) \in \mathbb{N}$ such that for every finite $K$-approximate subring $X$ there exists a $(K^{510}+K^{22})$-approximate subring $Y \subseteq 4X + X \cdot 4X$ which is $N_2(K)$-commensurable to $X$ for which there exists an ideal $I \lhd \langle Y \rangle$ contained in $Y_{N_3(K)}$ such that $\langle Y \rangle/I$ is 
nilpotent of class at most $\lfloor4\log_2(K)\rfloor$.
\end{theorem}

\begin{proof}[Proof of Theorem \ref{theorem: structure of finite approximate rings 1}]
Suppose for a contradiction that it is not true. Then, by model-theoretic compactness, there is an $\aleph_1$-saturated model $M=(G,\mathbb{R}^*)$ of the theory $T_0$ (in the notation from Appendix \ref{appendix: pseudofiniteness}) and a pseudofinite $K$-approximate subring $X$  definable in $M$ such that there is NO definable $(K^{510}+K^{22})$-approximate subring $Y \subseteq 4X +X \cdot 4X$ commensurable with $X$ for which there exists a definable ideal $I$ of $\langle Y \rangle^*$ contained in $Y$ 
and such that $\langle Y \rangle^*/I$ is nilpotent (of class at most $n$ for some $n \in \mathbb{N}$).

We will show that the approximate subring $Y$ from Corollary \ref{proposition: good escape norm 2} (which was obtained in Proposition \ref{proposition: improved target space 2}) yields a contradiction. Let $f \colon \langle Y \rangle \to \mathcal{A}$ be the witnessing locally compact model. Take the approximate subring $Z$ obtained in Corollary \ref{proposition: good escape norm 2} for $k:=2\dim_{\textrm{Lie}}(\mathcal{A})$. Although $Z$ need not be a $(K^{510}+K^{22})$-approximate subring, since $\langle Z \rangle = \langle Y \rangle$, we can replace $X$ by $Z$ and then the goal becomes to show that there is a definable ideal $I$ of $\langle X \rangle^*$ contained in $X$ and such that $\langle X \rangle^*/I$ is nilpotent of class at most $d$ for some $d \in \mathbb{N}$. We will obtain it with $d \leq \dim_{\textrm{Lie}}(\mathcal{A})$. After this replacement, the escape norm $||\cdot||_X$ has all the good properties and we have a locally compact model $f \colon \langle X \rangle \to \mathcal{A}$ with all the properties from  Corollary \ref{proposition: good escape norm 2}, and everything (i.e. $X$ and $f$) is chosen to be $M_0$-definable for some countable $M_0\prec M$. 
(After the above modifications of $X$ we forget about the choice of $Y$ and $Z$, and letters $Y$ and $Z$ may be used below to denote something else.)

To be more explicit, recall that $\mathcal{A}$ is a connected locally compact ring whose additive group is of the form $\mathbb{R}^n \times C$  for a compact connected Lie group $C$, and $f^{-1}[U] \subseteq X$ for some neighborhood $U$ of $0$ in $\mathcal{A}$. Note that $\mathcal{A} \ne \{0\}$, as otherwise $X = \langle X \rangle = \langle X \rangle^*$, and $I:=X$ would be an ideal of $\langle X \rangle^*$ with $\langle X \rangle^*/I=\{0\}$, a contradiction. This implies that $k \geq 2$, as otherwise $\mathcal{A}$ is 0-dimensional so trivial, a contradiction.


Let $R:=\langle X \rangle^*$. We will construct a finite sequence 
$$\{0\}=I_0 \lhd I_1 \lhd I_2 \lhd \dots \lhd I_{2d+1} =R$$ of (two-sided) ideals of $R$ together with several other sequences $(Y_i)_{i\leq 2d+1}$, $(\mathcal{A}_i)_{i \leq 2d+1}$, $(U_i)_{i\leq 2d+1}$, $(f_i)_{i\leq 2d+1}$, and $(u_i)_{i < d}$ of appropriate objects with certain properties which will be described at the end of the construction, where $d \leq \dim_{\textrm{Lie}}(\mathcal{A})$. (The sets $I_i$ as well as other indexed sets in the above sequences have nothing to do with the recursive definition of the sets $X_n$ for an approximate subring $X$ which is given in the introduction  on page \pageref{page: definition of  X_n}.) 

Set $Y_0:=X$, $\mathcal{A}_0:=\mathcal{A}$, $U_0:=U$, $f_0:=f$, and
$$I_1=I:= \{r \in R: ||r||_X=0\}.$$

By the ``moreover'' part of Proposition \ref{proposition: good escape norm 2}, $I$ is an $M_0$-definable ideal of $R$ contained in $\ker(f)$.
Note also that $\langle X/I \rangle =\langle X \rangle/I \subseteq R/I$. Thus, $f$ induces a surjective, $M_0$-definable locally compact model $f_1 \colon \langle X/I\rangle \to \mathcal{A}$ of $X/I$ given by $f_1(r/I):=f(r)$, with $f_1^{-1}[U] \subseteq X/I$. Put $Y_1:=X/I$, $\mathcal{A}_1:=\mathcal{A}$, and  $U_1:=U$.

Since $X+I \subseteq 2X$, we have $||r||_{2X} \leq ||r/I||_{X/I} \leq ||r||_X$. On the other hand, by Lemma \ref{lemma: Lou's 6.7} and the fact that $2X \subseteq X_2$ and $\cl(f[X_2])$ contains no subgroup other than $\{0\}$, there is $n \in \mathbb{N} \setminus \{0\}$ such that $||r||_X \leq n ||r||_{2X}$ for all $r \in \langle X \rangle^*$. Therefore, using also property (2) in Proposition \ref{proposition: good escape norm 2}, for all $r,s \in X$ we have:
\begin{equation}\label{equation: * in the main proof}
||(r/I)(s/I)||_{X/I} \leq ||rs||_X \leq 2||r||_X \cdot ||s||_X \leq 2n^2||r||_{2X}\cdot ||s||_{2X} \leq 2n^2||r/I||_{X/I}\cdot ||s/I||_{X/I}, \tag{*}
\end{equation}
\begin{equation}\label{equation: ** in the main proof}
||r/I||_{X/I}=0 \iff r/I=0/I.\tag{**}
\end{equation}

Since $U$ is a neighborhood of $0$, for every $m \in \mathbb{N}$ there is an open neighborhood $V \subseteq U$ of $0$ such that for every $s \in V$ we have that $is \in U$ for all $i \leq m$. Thus, taking any element $r \in f_1^{-1}[V] \setminus \{0/I\}$ (there is one since $f_1$ is onto $\mathcal{A}$ and $\mathcal{A} \ne \{0\}$ is connected), we get that $ir \in X/I$ for all $i \leq m$, and so $||r||_{X/I} \leq \frac{1}{m+1}$. On the other hand, since $X/I$ is pseudofinite, there is $ u \in (X/I)\setminus \{0/I\}$ with minimal possible $||u||_{X/I}$. Using (\ref{equation: ** in the main proof}), we conclude that $||u||_{X/I}$ is a positive infinitesimal (i.e., smaller than all standard positive reals).

Put
$$Z:=\left\{r \in X/I: ||r||_{X/I} < \frac{1}{2n^2}\right\}.$$
By the the previous paragraph and surjectivity of $f_1$, $f_1[Z]$ is a neighborhood of $0$ in $\mathcal{A}$. Since $\mathcal{A}$ is connected, we conclude that $f_1[\langle Z \rangle] = \mathcal{A}$. 
On the other hand, note that $\ker(f_1) \subseteq Z$ (because  $\ker(f_1)$ is a subgroup of $\langle X/I \rangle$ and so all elements of $\ker(f_1)$ have infinitesimal norm $||\cdot||_{X/I}$). All of this implies that $\langle Z \rangle = \langle X/I \rangle$.

By (\ref{equation: * in the main proof}), for any $r \in Z$, we have $||ru||_{X/I} \leq 2n^2||r||_{X/I}\cdot ||u||_{X/I} < ||u||_{X/I}$, and so $ru=0/I$ by minimality of $||u||_{X/I}$. So $Zu =\{0/I\}$. Similarly, $uZ = \{0/I\}$. Hence, $u \langle Z \rangle = \langle Z \rangle u = \{0/I\}$, which, by the previous paragraph, means that $u \langle X/I \rangle = \langle X/I \rangle u = \{0/I\}$. This implies that $(\mathbb{Z}^* u) \cdot R/I = R/I \cdot (\mathbb{Z}^* u)= \{0/I\}$ (where $\mathbb{Z}^*u:=\{\nu u: \nu \in \mathbb{Z}^*\}$ is the smallest definable additive subgroup of $R/I$ containing $u$). In particular, $\mathbb{Z}^* u$ is a (two-sided) ideal in $R/I$. Take any countable model $M_1 \prec M$ containing $M_0$ and $u$. Then $\mathbb{Z}^*u$ is $M_1$-definable.

Let $\eta \colon R/I \to (R/I)/\mathbb{Z}^*u$ be the quotient map and $\mathcal{D} :=\cl(f_1[ \mathbb{Z}^*u \cap \langle X/I\rangle])$. By Lemma \ref{lemma: passing to quotients}, $\eta[X/I]$ is an $M_1$-definable approximate subring and the function $\hat{f}_1 \colon  \langle \eta[X/I] \rangle \to \mathcal{A}/\mathcal{D}$ given by $\hat{f}_1(\eta(x)):=f_1(x)/\mathcal{D}$ is a surjective $M_1$-definable locally compact model of $\eta[X/I]$, and 
$U/\mathcal{D}$ is a neighborhood of $0/\mathcal{D}$ in $\mathcal{A}/\mathcal{D}$ with $\hat{f}^{-1}[U/\mathcal{D}] \subseteq (X/I)/\mathbb{Z}^*u$.
Put $\mathcal{A}_2:=\mathcal{A}/\mathcal{D}$.

Of course, $(R/I)/\mathbb{Z}^*u$ is naturally identified with $R/I_2$ for some $M_1$-definable ideal $I_2$ of $R$ with $I \subseteq I_2$, namely $I_2:=I + \mathbb{Z}^*u_0$ for some/any $u_0 \in X$ with $u_0/I=u$ (so $||u_0/I||_{X/I}$ is a positive infinitesimal). Then $\eta[X/I]$ is identified with $X/I_2$.

Since the additive group of $\mathcal{A}_2$ is a connected locally compact abelian Lie group, it is topologically isomorphic to $\mathbb{R}^m \times P$ for some $m \in \mathbb{N}$ and compact connected Lie group $P$ (using \cite[Theorem 4.2.1]{DeEc}, as in the argument at the beginning of the proof of Proposition \ref{proposition: improved target space 2}). Applying Proposition \ref{corollary: model with the same target space} to the $M_1$-definable approximate subring $X/I_2$ and its $M_1$-definable locally compact model $\hat{f}_1 \colon \langle X/I_2 \rangle \to \mathcal{A}_2$, we find an $M_1$-definable approximate subring $Y_2 \subseteq X/I_2$ commensurable to $X/I_2$ and satisfying all the good properties from Proposition \ref{corollary: model with the same target space} for $k_2:=2\dim_{\textrm{Lie}}(\mathcal{A}_2)$. In particular, $\langle Y_2 \rangle = \langle X/I_2 \rangle$ and $f_2:=\hat{f}_1 \colon \langle Y_2 \rangle \to \mathcal{A}_2$ is a surjective $M_1$-definable locally compact model of $Y_2$ with some neighborhood $U_2$ of $0$ in $\mathcal{A}_2$ such that $f_2^{-1}[U_2] \subseteq Y_2$.

By Lemma \ref{lemma: nondiscrete image}(2), $\mathcal{D}$ is non-discrete in $\mathcal{A}$. Thus, 
$$\dim_{\textrm{Lie}}(\mathcal{A}_2)<\dim_{\textrm{Lie}}(\mathcal{A}).$$ 

If $\mathcal{A}_2=\{0\}$, then $Y_2= R/I_2$, and we put $I_3:=R$, $Y_3:=R/I_3 =R/R$, $\mathcal{A}_3:=\{0\}$, $U_3:=\{0\}$, $f_3: \langle Y_3 \rangle \to \mathcal{A}_3$ the zero map, and stop the construction. 
Otherwise, we repeat the above two steps for $Y_2$, $\mathcal{A}_2$, $f_2$, $U_2$, and $M_1$ in place of $X$, $\mathcal{A}$, $f$, $U$, and $M_0$, respectively. Continuing this construction at most $2\dim_{\textrm{Lie}}(\mathcal{A})+1$ steps, we obtain a sequence
$$\{0\}=I_0 \lhd I_1 \lhd I_2 \lhd \dots \lhd I_{2d+1} =R$$ 
of (two-sided) ideals of $R$ with $I_1= \{ r \in R: ||r||_X=0\}$ and a countable $M_0'\prec M$ such that:
\begin{enumerate}
\item if $i<2d+1$ is even, then $I_{i+1}/I_i \subseteq X/I_i$;
\item if $i<2d+1$ is odd, then $I_{i+1}/I_i$ is null in $R/I_i$ (i.e., $(R/I_i) (I_{i+1}/I_i) = (I_{i+1}/I_i)(R/I_i) =\{0/I_i\}$);
\item all $I_i$'s are $M_0'$-definable;
\end{enumerate}
as well as sequences $(Y_i)_{i\leq 2d+1}$, $(\mathcal{A}_i)_{i \leq 2d+1}$, $(U_i)_{i\leq 2d+1}$, $(f_i)_{i\leq 2d+1}$, and $(u_i)_{i < d}$, where $Y_0=X$, $\mathcal{A}_0=\mathcal{A}$, $U_0=U$, $f_0=f$, and:
 
\begin{itemize}
\item $Y_i \subseteq X/I_i$ is an $M_0'$-definable approximate subring commensurable to $X/I_i$ and generating $\langle X/I_i \rangle$;
\item $Y_j \subseteq Y_{i}/(I_{j}/I_i)$ for all $j>i$ (after the identification of $R/I_j$ and $(R/I_i)/(I_j/I_i)$);
\item if $j<2d+1$ is even and $i\leq j$, then $I_{j+1}/I_j \subseteq Y_i/(I_j/I_i)$ (after the identification);
\item $I_{2i+2} = I_{2i+1} + \mathbb{Z}^* u_i$, $||u_i/I_{2i+1}||_{Y_{2i+1}}$ is a positive infinitesimal;
\item $\mathcal{A}_i$ is a connected locally compact ring whose additive group is of the form $\mathbb{R}^{n_i} \times C_i$ for a compact Lie group $C_i$;
\item $f_i \colon \langle Y_i \rangle \to \mathcal{A}_i$ is a surjective $M_0'$-definable locally compact model of $Y_i$;
\item $d-i \leq \dim_{\textrm{Lie}}(\mathcal{A}_{2i})$;
\item  $\cl(f_{2i}[(Y_{2i})_{2\dim_{\textrm{Lie}}(\mathcal{A}_{2i})}])$ contains no subgroup other than $\{0\}$;
\item $U_i$ is a neighborhood of $0$ in $\mathcal{A}_i$ such that $f_i^{-1}[U_i] \subseteq Y_i$;
\item For $i$ even, $||\cdot||_{Y_i}$ has properties (1)-(3) from Proposition \ref{proposition: good escape norm 2}.
\end{itemize}

We will show that whenever we have such sequences, then $R/I_1$ is nilpotent of class at most $d$. Since 
$I_1$ is a definable ideal of $R$ contained in $X$, this yields a contradiction with the choice of $X$ 
(see the second paragraph of the proof) so the proof will be finished.

The proof is by induction on $d$.
The base step $d=0$ is trivial. Now, assume that $d\geq 1$ and the conclusion holds for shorter sequences. Applying the induction hypothesis to the sequence $\{0\}=I_2/I_2 \lhd I_3/I_2 \lhd \dots \lhd I_{2d+1}/I_2=R/I_2$ (for $Y_2$ playing the role of $X$), we get that $R/I_3$ is nilpotent of class at most $d-1$, i.e. $R^d \subseteq I_3$.
Our goal is to show that for any $r \in R$ and $s \in R^d$, the product $rs \in I_1$. By the previous sentence, $s \in I_3$, so $s \in I_2+X$ by property (1) above. 

Before we continue, let us slightly modify the $u_i$'s. Since $||u_i/I_{2i+1}||_{Y_{2i+1}}$ is infinitesimal, $u_i/I_{2i+1} \in Y_{2i+1}$. So $u_i/I_{2i+1} \in X/I_{2i+1}$. Thus, changing the representative $u_i$ of the coset $u_i + I_{2i+1}$, we can and do assume that $u_i \in X$ (the fourth bullet above is preserved, because  $I_{2i+1} + \mathbb{Z}^* u_i$ is the smallest definable subgroup of $(R,+)$ containing $I_{2i+1}$ and $u_i$).  

By item (1) and the fourth bullet, $r \in I_2 + (r_1 + n_1u_1) + \dots + (r_{d-1} + n_{d-1}u_{d-1}) + r_d$ for some $r_i \in X$ and $n_i \in \mathbb{Z}^*$. Since $s \in I_2 +X$ and, by item (2), $I_2/I_1$ is null in $R/I_1$, we get
$$rs \in I_1 + \sum_{i=1}^d r_is + \sum_{i=1}^{d-1} n_iu_iI_3  \subseteq I_1 + dX^2 + \sum_{i=1}^{d-1} u_i I_3 \subseteq I_1 + dX^2 + \sum_{i=1}^{d-1} u_i X \subseteq I_1+ (2d-1)X^2,$$
where in the first two inclusions we also use the fact that for $a \in I_3$ we have $(n_iu_i)a =u_i (n_ia)$ and $n_ia \in I_3 \subseteq I_2 +X$ (as $I_3$ is a definable subgroup of $R$). 

So $rs \in X_{m}$, where $m=2d$. Then for any $n \in \mathbb{Z}^*$, $nrs=r(ns) \in X_m$, as $ns \in R^d$. Hence, $||rs||_{X_m}=0$.  Since $f^{-1}[U] \subseteq X$
and 
 $\cl([f[X_m])$ contains no subgroup other than $\{0\}$, 
by Lemma \ref{lemma: Lou's 6.7}, we conclude that $||rs||_X =0$, that is $rs \in I_1$, as required.
\end{proof}

\begin{proof}[Proof of Theorem \ref{theorem: structure of finite approximate rings 2}]
The proof is essentially the same as the proof of Theorem \ref{theorem: structure of finite approximate rings 1}. Only the first two paragraphs require a modification, which we now explain.

Suppose for a contradiction that it is not true. Then, by model-theoretic compactness, there is an $\aleph_1$-saturated model $M=(G,\mathbb{R}^*)$ of the theory $T_0$ (in the notation from Appendix \ref{appendix: pseudofiniteness}) and a pseudofinite $K$-approximate subring $X$ definable in $M$ such that there is NO definable $(K^{510}+K^{22})$-approximate subring $Y \subseteq 4X +X \cdot 4X$ commensurable with $X$ for which there exists a definable ideal $I$ of $\langle Y \rangle^*$ contained in $Y_m$ for some $m \in \mathbb{N}$ and such that $\langle Y \rangle^*/I$ is nilpotent of class at most $\lfloor 4\log_2(K) \rfloor$.

Next we do the same thing as in the second paragraph of the proof of Theorem \ref{theorem: structure of finite approximate rings 1}, but using Corollary \ref{corollary: good escape norm with finite dimensional algebra} in place of Corollary \ref{proposition: good escape norm 2}. Then the target space $\mathcal{A}$ of the obtained surjective definable locally compact model $f \colon \langle X \rangle \to \mathcal{A}$ (of the modified $X$) is of Lie dimension at most $4\log_2(K)$, and this is exactly what we need to get at the end of the construction that $\langle X \rangle^*/I$ is nilpotent of class at most $\lfloor 4\log_2(K) \rfloor$ which yields a contradiction.
\end{proof}

\begin{corollary}\label{corollary: structure by finite}
For any $K \in \mathbb{N}$ there exists $N(K) \in \mathbb{N}$ such that for every finite $K$-approximate subring $X$ there exists a $(K^{510}+K^{22})$-approximate subring $Y\subseteq 4X + X \cdot 4X$ which is $N(K)$-commensurable to $X$ and for which the ring $\langle Y \rangle$ is ($\lfloor 4 \log_2(K) \rfloor$-niloptent)-by-finite (in particular, (nilpotent of class at most $\lfloor 4 \log_2(K)\rfloor$)-by-finite).
\end{corollary}

Note first that a weaker version in which $\langle Y \rangle$ is only required to be (nilpotent of class at most $\lfloor 4 \log_2(K)\rfloor+1$)-by-finite follows immediately from Theorem \ref{theorem: structure of finite approximate rings 2}. Just take $Y$ and $I$ from Theorem \ref{theorem: structure of finite approximate rings 2}. Let $S:=\{r \in \langle Y \rangle: rI=Ir=\{0\}\}$ (i.e., the two-sided annihilator of $I$ in $\langle Y \rangle$). Since $I$ is a finite ideal, $S$ is a finite index ideal in $\langle Y \rangle$. And it is clear that $S$ is nilpotent of class at most $\lfloor 4 \log_2(K) \rfloor+1$.

The full version of Corollary \ref{corollary: structure by finite} is more delicate. First, observe that the proof of Theorem \ref{theorem: structure of finite approximate rings 2} yields the following

\begin{corollary}\label{corollary: of the proof of the main theorem}
For any $K \in \mathbb{N}$ there exists $N(K) \in \mathbb{N}$ such that for every finite $K$-approximate subring $X$ there exists a $(K^{510}+K^{22})$-approximate subring $Y \subseteq 4X + X \cdot 4X$ which is $N(K)$-commensurable to $X$ for which there exists an ideal $I_1 \lhd \langle Y \rangle$ contained in $Y_{N(K)}$ and a sequence of ideals
$$\{0\}=I_0\lhd I_1 \lhd \dots \lhd I_{2d+1}=\langle Y \rangle$$
for some $d \leq \lfloor 4 \log_2(K) \rfloor$, such that:
\begin{enumerate}
\item if $i$ is even, then $I_{i+1}/I_i \subseteq Y_{N(K)}/I_i$, so $I_{i+1}/I_i$ is finite;
\item if $i$ is odd, then $I_{i+1}/I_i$ is null in $\langle Y \rangle/I_i$ and $I_{i+1} = I_i + \mathbb{Z}u_i$ for some $u_i \in Y_{N(K)}$.
\end{enumerate}
\end{corollary}


We also need the following ring-theoretic variant of \cite[Lemma 7.4]{Dries}.

\begin{lemma}\label{lemma: ring-theoretic Lou's 7.4}
Suppose $I_1 \subseteq I_2$ are ideals of a ring $R$ such that $I_1$ is finite, $I_2/I_1$ is null in $R/I_1$,
and  the additive group of $I_2/I_1$ is cyclic generated by an element $u/I_1$. Then $I_2$ has a finite index additive subgroup $J$ generated by $nu$ for some $n \in \mathbb{N}$ such that $J$ is null in $R$ (in particular, $J$ is an ideal in $R$).
\end{lemma}

\begin{proof}
If $I_2$ is finite, then $J:=\{0\}$ works. So assume that $I_2$ is infinite. We have $I_2=\mathbb{Z}u+ I_1$, and $\mathbb{Z}u$ is a finite index subgroup of $I_2$, so $\mathbb{Z}u$ is isomorphic to $(\mathbb{Z},+)$.

Let $r_1,r_2$ below range over $R \sqcup \{1\}$, where $1$ is an external unit for $R$.

Define $f_{r_1,r_2} \colon I_2 \to I_2$ by $f_{r_1,r_2}(x) := r_1 x r_2$. These are endomorphisms of the additive group of $I_2$.
Put $A_{r_1,r_2}:= f_{r_1,r_2}^{-1}[\mathbb{Z}u]$. These are subgroups of $(I_2,+)$ of index at most $[I_2:\mathbb{Z}u]<\infty$. Hence, since $(I_2,+)$ is a finitely generated group, there are only finitely many possibilities for $A_{r_1,r_2}$. Define
$$J:=\bigcap_{r_1,r_2 \in R \sqcup \{1\}} A_{r_1,r_2}.$$

Taking $r_1=r_2=1$, we get that $J \subseteq \mathbb{Z}u$. We claim that $J$ is an ideal in $R$. The fact that it is an additive subgroup is obvious.  It remains to check that for any $s_1,s_2 \in  R \sqcup \{1\}$ we have $s_1Js_2 \subseteq J$, and for that it is enough to see that $s_1Js_2\subseteq A_{r_1,r_2}$ for all $r_1,r_2 \in  R \sqcup \{1\}$. The last inclusion holds, because we trivially have $s_1A_{r_1s_1,s_2r_2}s_2 \subseteq A_{r_1,r_2}$. By the previous paragraph, we also see that $J$ has finite index in $I_2$.

Since $J$ is a nontrivial subgroup of $\mathbb{Z}u$, it is of the form $\mathbb{Z}nu$ for some $n>0$. It remains to show that $J$ is null in $R$. Suppose not. Then $rnu \ne 0$ or $nur \ne 0$ for some $r \in R$. Both cases are similar, so consider the first one.  As $I_2/I_1$ is null in $R/I_1$ and $nu \in I_2$, we conclude that $rnu \in I_1$. On the other hand, as $J$ is an ideal in $R$ and $nu \in J$, we get that $rnu \in J=\mathbb{Z}nu$. Therefore, $\mathbb{Z}u \cap I_1 \ne \{0\}$ so infinite, which contradicts the assumption that $I_1$ is finite. 
\end{proof}

Now, the full version of Corollary \ref{corollary: structure by finite} follows by induction from Corollary \ref{corollary: of the proof of the main theorem} using Lemma \ref{lemma: ring-theoretic Lou's 7.4}. In fact, having the data from Corollary \ref{corollary: of the proof of the main theorem},  one obtains a finite index ideal $J$ in $\langle Y \rangle$ which is $d$-nilpotent with nilpotent base $n_1u_1,\dots, n_du_d$ for some $n_1,\dots,n_d \in \mathbb{N}$.

\medskip

\section{Applications of the main structural results}

In this section, we give two applications of Theorems \ref{theorem: structure of finite approximate rings 1} and \ref{theorem: structure of finite approximate rings 2}.

\subsection{Sum-product phenomenon}

We give a very general qualitative version of a sum-product phenomenon which contains meaningful information even in the case of ``many'' zero-divisors. 



For any nonempty subset $X$ of a ring, the subset $X-X$ is additively symmetric, and by $X_{n+1}'$ we will denote the set $(X-X)_n$ (see the recursive definition in the introduction).


The next theorem will be deduced from Theorem \ref{theorem: structure of finite approximate rings 1}. For $K \in \mathbb{N}$ let $N(K)$ be $N_2(K)$ from the conclusion of Theorem \ref{theorem: structure of finite approximate rings 1}. 

\begin{theorem}\label{theorem: sum-product phenomenon with any g}
Let $g \colon \mathbb{N} \to \mathbb{R}^+$ be any function 
such that the function $\frac{n}{g(n)}$ is non-decreasing and tends to infinity.
Then there is a non-decreasing unbounded function $f\colon \mathbb{N} \rightarrow \mathbb{N}$ such that the following holds. 

Let $R$ be a ring and $X \subseteq R$ be a finite subset. Then: 
\begin{itemize}
\item either, $|X+X+X\cdot X| \geq f(|X|) |X|$, or
\item there is a subring $R' \subseteq R$ and an ideal $I \subseteq R' \cap (4(X-X) + (X-X) \cdot 4(X-X)) \subseteq  R' \cap X_4'$ such that $R'/I$ is nilpotent and $|X_4' \cap R'| \geq g(|X|)$.\footnote{As mentioned in the introduction, in a forthcoming joined paper of the first author with Mateusz Rzepecki, the expression  $4X + X \cdot4X$ in Fact \ref{fact: locally compact model exists} will be improved to $4X + X \cdot 2X$. 
So here  $4(X-X) + (X-X) \cdot 4(X-X)$ can be replaced by $4(X-X) + (X-X) \cdot 2(X-X)$ and in consequence $X_4'$ by $X_3'$.}
\end{itemize}
\end{theorem}

\begin{proof}
Set $f(0):=0$, and from now on assume that $X \ne \emptyset$. 
By Remark \ref{remark: small n-pling for rings}, for any choice of a function $f$, if $|X+X+X\cdot X| \leq f(|X|) |X|$, then $X-X$ is an $(f(|X|)^{5} + f(|X|)^{19})$-approximate subring. Set $K(x):=x^{5} + x^{19}$ for $x \in \mathbb{N}^+$. Then, by Theorem \ref{theorem: structure of finite approximate rings 1}, there exists $Y \subseteq 4(X-X)+ (X-X) \cdot 4(X-X)$ which is $N(K(f(|X|))$-commensurable to $X-X$ and for which there exists an ideal $I \lhd \langle Y \rangle$ contained in $Y$ such that $\langle Y \rangle/I$ is nilpotent.
Put $R':=\langle Y \rangle$. To finish the proof, we need to choose $f$ so that $|X_4' \cap R'| \geq g(|X|)$. Let $M:=N \circ K$. Define $f$ as follows:

\begin{displaymath}
f(n):= \left\{
\begin{array}{l}
0 \textrm{ if } M^{-1}[(0,\frac{n}{g(n)}]] = \emptyset\\
\textrm{the largest element of }  M^{-1}[(0,\frac{n}{g(n)}]] \textrm{ if } M^{-1}[(0,\frac{n}{g(n)}]] \ne \emptyset \textrm{ is finite}\\
\textrm{any element of }  M^{-1}[(0,\frac{n}{g(n)}]] \textrm{ greater than } f(0),\dots,f(n-1) \textrm{ if }  M^{-1}[(0,\frac{n}{g(n)}]] \textrm{ is infinite.}
\end{array}\right.
\end{displaymath}
Any such $f$ is non-decreasing and unbounded. Moreover, still assuming that  $|X+X+X\cdot X| \leq f(|X|) |X|$ (which implies that $f(|X|) \ne 0$), we have $M(f(|X|)) \leq \frac{|X|}{g(|X|)}$, so $N(K(f(|X|)))\leq \frac{|X|}{g(|X|)}$, so $\frac{|X|}{N(K(f(|X|)))} \geq g(|X|)$. Since $|Y| \geq \frac{|X-X|}{N(K(f(|X|)))} \geq \frac{|X|}{N(K(f(|X|)))}$, we conclude that $|Y| \geq g(|X|)$. As $Y \subseteq X_4' \cap R'$, the last inequality implies $|X_4' \cap R'| \geq g(|X|)$ as required.
\end{proof}

\begin{corollary}\label{corollary: sum-product phenomenon with epsilon}
Let $\epsilon > 0$. There is a non-decreasing unbounded function $f\colon \mathbb{N} \rightarrow \mathbb{N}$ such that the following holds. 

Let $R$ be a ring and $X \subseteq R$ be a finite subset. Then: 
\begin{itemize}
\item either, $|X+X+X\cdot X| \geq f(|X|) |X|$, or
\item there is a subring $R' \subseteq R$ and an ideal $I \subseteq R' \cap (4(X-X) + (X-X) \cdot 4(X-X)) \subseteq  R' \cap X_4'$ such that $R'/I$ is nilpotent and $|X_4' \cap R'| \geq |X|^{1-\epsilon}$.
\end{itemize}
\end{corollary}

\begin{proof}
It follows from Theorem \ref{theorem: sum-product phenomenon with any g} applied for $g(n):=n^{1-\epsilon}$.
\end{proof}

A natural question arises, how fast the function $f$ in the above sum-product phenomena grows. 
We would get a satisfactory quantitative answer assuming that $N(K)$ could be polynomially bounded.

\begin{remark}\label{remark: polynomial bound on f}
If $N(K)$ could be polynomially bounded, say by $CK^d$, then the function $f$ in Theorem \ref{theorem: sum-product phenomenon with any g} could be chosen to satisfy $f(n) = \left\lfloor \frac{1}{2^{\frac{1}{19}}C^{\frac{1}{19d}}}\left(\frac{n}{g(n)}\right)^{\frac{1}{19d}} \right\rfloor$, which in the particular case of Corollary \ref{corollary: sum-product phenomenon with epsilon} yields $f(n) =  \left\lfloor \frac{1}{2^{\frac{1}{19}}C^{\frac{1}{19d}}}n^{\frac{\epsilon}{19d}} \right\rfloor$.
\end{remark}

In Theorem \ref{theorem: sum-product phenomenon with any g} and Corollary \ref{corollary: sum-product phenomenon with epsilon}, we used the sum-product condition $|X+X+X\cdot X| \geq f(|X|) |X|$. We remark that it can be replaced by any other sum-product condition for which Remark \ref{remark: small n-pling for rings} remains valid (potentially with different constant). 
For example, in \cite[Lemma 3.3.1]{Kow}, it is shown that if $X$ is a finite subset of a unital ring\footnote{In \cite[Lemma 3.3.1]{Kow}, it is assumed that the ring is commutative, but this assumption is not needed.} and $1\in X$, then the sum-product condition $|X\cdot X - X \cdot X|\leq K|X|$ implies that $X-X$ is a $K^{O(1)}$-approximate subring.
Thus, we have:

\begin{remark}
The variants of Theorem \ref{theorem: sum-product phenomenon with any g} and Corollary \ref{corollary: sum-product phenomenon with epsilon} with $R$ a unital ring, $X$ a finite subset containing $1$, and the first bullets replaced by $|X\cdot X-X \cdot X|\geq f(|X|) |X|$ remain valid.
\end{remark}

\subsection{A ring-theoretic counterpart of Gromov's theorem}

In this subsection, we show a ring-theoretic counterpart of Gromov's theorem on groups of polynomial growth. Our proof works for (finitely generated) torsion-free rings of polynomial growth. The main tool is Theorem  \ref{theorem: structure of finite approximate rings 2}.

For a subset $X$ of a ring $R$ by $X^{\leq n}$ we will denote the subset of $R$  consisting of the elements obtained from $X$ using $+$ and $\cdot$ so that the elements of $X$ are used at most $n$ times (counting potential repetitions).

\begin{definition}\label{definition: ring of polynomial growth}
A ring $R$ generated by a finite additively symmetric set $X$ has {\em polynomial growth} if there exists $d \in \mathbb{N}$ and constant $C$ such that $|X^{\leq n}| \leq Cn^d$ for all positive integers $n$.
\end{definition}

The following remark is straightforward.

\begin{remark}
For a finitely generated ring the property of being of polynomial growth does not depend on the choice of the finite set of generators.
\end{remark}

\begin{proposition}
A finitely generated virtually nilpotent ring has polynomial growth.
\end{proposition}

\begin{proof}
Let $I\lhd R$ be a finite index ideal which is nilpotent of class at most $c$. Choose a finite symmetric set $X=\{x_1,\dots,x_m\}$ of generators of $R$ so that $R/I=\{x_1/I,\dots,x_m/I\}$. Then for every $i,j\in \{1,\dots,m\}$ we can write $x_i+x_j = x_{f(i,j)} +r_{ij}$ and $x_ix_j=x_{g(i,j)} +s_{ij}$ for some $f(i,j),g(i,j) \in \{1,\dots,m\}$ and $r_{ij},s_{ij} \in I$. Let $a_1,\dots,a_{2m^2}$ be an enumeration of all the $r_{ij}$ and $s_{ij}$ (with possible repetitions). Then each element of $R$ can be written as the sum of one of the $x_i$'s and a sum of elements of the form
\begin{equation}\label{equation: generating elements}
z_{\bar \epsilon,\bar i, \bar k, \bar j}:=x_{i_1}^{\epsilon_1} a_{j_{1}^1} \dots a_{j_{k_1}^1} x_{i_2}^{\epsilon_2} a_{j_{1}^2} \dots a_{j_{k_2}^2}\dots x_{i_{l-1}}^{\epsilon_{l-1}} a_{j_{1}^{l-1}} \dots a_{j_{k_{l-1}}^{l-1}}x_{i_l}^{\epsilon_l}, \tag{*}
\end{equation}
where $l \geq 2$, $\epsilon_1,\dots,\epsilon_{l-1} \in \{0,1\}$, $k_1,\dots,k_{l-1} \geq 1$, and $k_1+\dots+k_{l-1} \leq c$.

We claim that for any $l \geq 2$, any $k_1,\dots,k_{l-1} \geq 1$ with $k_1+\dots+k_{l-1} \leq c$, and $\alpha,\beta \in \{0,1\}$ there exists $d(\alpha,\beta;k_1,\dots,k_{l-1}) \in \mathbb{N}$ such that for every positive $n \in \mathbb{N}$ each element of the set $X^{\leq n}$ can be written as 
$$x_i + \sum_{\bar \epsilon,\bar i, \bar k, \bar j} A_{\bar \epsilon,\bar i, \bar k,\bar j} z_{\bar \epsilon,\bar i, \bar k,\bar j},$$
for some $i \in \{1,\dots,m\}$ and natural numbers $A_{\bar \epsilon,\bar i, \bar k,\bar j} < n^{d(\epsilon_1,\epsilon_l;\bar k)}$.

The proof of the above claim is by induction on $k_1+\dots +k_{l-1}$, and in order to show the induction step there is another induction on $n$. It is a rather straightforward technical exercise which requires considering several cases, so we leave it to the reader.

Having this, we can choose $d$ to be the maximum of all the numbers $d(\alpha,\beta;k_1,\dots,k_{l-1})$. Then we see that $|X^{\leq n}| \leq mn^{d d'}$, where $d'$ is the number of elements of the form (\ref{equation: generating elements}), and so we have a polynomial bound.
\end{proof}

Our goal is the converse. We will show the following more general and more precise statement, but only in the torsion-free context. 
We do not know whether this result or its appropriate variant remains true for rings with torsion elements; see also Remark \ref{remark: Gromov with weekend assumption}.

\begin{theorem}\label{theorem: Gromov for torsion-free rings}
A finitely generated torsion-free ring of polynomial growth is virtually nilpotent.
More precisely, given $d > 0$, if $S$ is a torsion-free ring generated by a finite additively symmetric set $X$ for which $|X^{\leq n}| \leq n^d|X|$ for arbitrarily large $n$, then $S$ has an ideal of index at most $O_d(1)$ which is nilpotnent of class at most $4(57d+1)$. In particular, $S$ is virtually nilpotent.
\end{theorem}

\begin{proof}
Consider any positive integer $N$.
\begin{clm}
There is $n_0=n_0(N)>N$ such that whenever  $X$ is a subset of a ring $S$ and $n>n_0$ satisfies $|X^{\leq n}| \leq n^d|X|$, then there is $n'$ between $N$ and $n$ such that $|X^{\leq 4n'}| \leq 8^d|X^{\leq n'}|$.
\end{clm}

\begin{clmproof}
We will show that any $n_0 \geq 64N^3$ works. Take such an $n_0$ and suppose for a contradiction that the conclusion fails which is witnessed by some $S$, $X$, and $n$. Then $|X^{\leq 4n'}| > 8^d|X^{\leq n'}|$ for any $n'$ of the form $N\cdot 4^k$ where $k \in \mathbb{N}$ is arbitrary with $N\cdot 4^k \leq n$. Hence, $|X^{\leq N \cdot 4^k}| > 8^{kd}|X^{\leq N}|$ for all such $k$'s. Pick $k$ so that $N \cdot 4^k \leq n<N \cdot 4^{k+1}$. By the choice of $n$, we get $|X^{\leq N \cdot 4^k}| \leq |X^{\leq n}| \leq  n^d|X| <N^d \cdot 4^{(k+1)d} \cdot |X| =(4N)^{d} \cdot 4^{kd} \cdot |X|$. We conclude that $|X^{\leq N}| < (4N)^d \cdot (\frac{1}{2})^{kd} \cdot |X|$. Thus, we will get a contradiction if we show that $(4N)^d \cdot (\frac{1}{2})^{kd} \cdot |X|\leq |X|$. This is equivalent to showing that $2^k \geq 4N$. And this is true, because the inequality $N \cdot 4^{k+1} >n$ implies  $2^k > \sqrt{n/4N} > \sqrt{n_0/4N} \geq 4N$.
\end{clmproof}

Consider an additively symmetric subset $X$ of a ring and any positive $n'$ satisfying $|X^{\leq 4n'}| \leq 8^d|X^{\leq n'}|$. Put $X':=X^{\leq n'}$. Note that $X'$ is additively symmetric and $X' +X' + X' \cdot X' \subseteq X^{\leq 4n'}$, so $|X' +X' + X' \cdot X'| \leq 8^d |X'|$. Hence, by Remark \ref{remark: small n-pling for rings}, $2X'$ is a $C(d)$-approximate subring, where $C(d):= 8^{5d} +8^{19d}$. 

Let now $S$ and $X$ be as in the statement of the theorem. Applying the above paragraph to the situation from Claim 1 for bigger and bigger numbers $N$, we obtain a sequence $n_0'<n_1'<n_2'<\dots$ such that $X_i':=2X^{\leq 4 n_i'}$ is a $C(d)$-approximate subring of $S$, and clearly $S = \bigcup_{i \in \mathbb{N}} X_i'$.
Hence, in order to finish the proof of the theorem, it remains to prove the following general observation (see the explanation at the end of the proof).

\begin{clm}
Let $K \in \mathbb{N}$. If a torsion-free ring $S$ can be written as $\bigcup X_i'$ where $(X_i')_{i \in \mathbb{N}}$ is an increasing sequence of finite $K$-approximate subrings, then there exists a subring $S'$ of $S$ of index $O_K(1)$ which is nilpotent of class at most $\lfloor 4\log_2(K) \rfloor$.
\end{clm} 

\begin{clmproof}
Take $N_2(K),N_3(K) \in \mathbb{N}$ satisfying the conclusion of Theorem  \ref{theorem: structure of finite approximate rings 2}.

Let us work in an $\aleph_1$-saturated elementary extension $M$ of the two-sorted structure consisting of the sorts $(S,+,\cdot)$ and $(\mathbb{R},+,\cdot,0,1,\leq,\mathbb{Z})$ and a predicate $R$ defined as in Appendix \ref{appendix: pseudofiniteness}.  In particular, $M$ is a model of $T_0$ (in the notation from the appendix).

By the choice of $N_2(K)$ and $N_3(K)$, for every $i \in \mathbb{N}$ there exists a finite $(K^{510}+K^{22})$-approximate subring $Y_i'$ of $S$ which is $N_2(K)$-commensurable with $X_i'$ for which there exists an ideal $I \lhd \langle Y_i' \rangle$ contained in $(Y_i')_{N_3(K)}$ such that $\langle Y_i' \rangle/I$ is nilpotent of class at most $\lfloor4\log_2(K)\rfloor$. Since $S$ is torsion-free and $(Y_i')_{N_3(K)}$ is finite, $I = \{0\}$, and so $\langle Y_i' \rangle$ is nilpotent of class at most $\lfloor4\log_2(K)\rfloor$.

Thus, by $\aleph_1$-saturation of $M$ and uniform definability  in $M$ (using the predicate $R$) of the finite subsets of $S$, there exists a pseudofinite $(K^{510}+K^{22})$-approximate subring $Y$ of 
$S(M)$ (where $S(M)$ the interpretation of $S$ in $M$) such that: 
\begin{enumerate}
\item for every $i \in \mathbb{N}$, $N_2(K)$ additive translates of $Y$ cover $X_i'$;
\item $\langle Y \rangle$ is nilpotent of class at most $\lfloor4\log_2(K)\rfloor$.
\end{enumerate}

Define $S':= S \cap \langle Y \rangle$. By (2), $S'$ is nilpotent of class at most $\lfloor4\log_2(K)\rfloor$. On the other hand, by (1), for every $i$, $X_i'$ is covered by $N_2(K)$ additive translates of $S'$. Since $S$ is the increasing union of the sets $X_i'$, we conclude that $[S:S'] \leq N_2(K)$.
\end{clmproof}

Applying Claim 2 to the data from the paragraph preceding it for $K:=C(d)$, we obtain a subring $S'$ of $S$ of index at most $O_d(1)$ which is nilpotent of nilpotency class at most $4\log_2(2\cdot 8^{19d})= 4(57d+1)$. By \cite[Lemma 1]{Lew}, $S'$ contains an ideal $I$ of $S$ of index at most $O_d(1)$, and $I$ is clearly still nilpotent of nilpotency class at most $4(57d+1)$.
\end{proof}

The next remark  follows from the proof of Theorem \ref{theorem: Gromov for torsion-free rings} (see the proof of Claim 2).

\begin{remark}\label{remark: Gromov with weekend assumption}
In Theorem \ref{theorem: Gromov for torsion-free rings}, instead of torsion-freeness it is enough to assume that $S$ contains no subring which has a non-zero finite ideal.
\end{remark}

\section{Sum-product for real algebras}\label{Section: Sum-product for real algebras}

In this section, we are particularly interested in interactions of approximate rings with topologies. When the field $k$ is local (e.g. $\mathbb{R}, \mathbb{C}$ or $\mathbb{Q}_p$), finite-dimensional associative $k$-algebras are naturally equipped with the topology inherited from $k$ which makes them into locally compact rings. Sum-product phenomena \emph{à la} Meyer concern approximate subrings that are \emph{uniformly discrete} with respect to this topology. A subset $X$ of a $k$-algebra $A$ over a local field $k$ is called {\em uniformly discrete} if there is a neighbourhood $W$ of $0$ such that $(x + W) \cap (y + W) =\emptyset$ for all $x, y \in X$ distinct. Meyer's theorem asserts that discrete approximate subrings of locally compact fields must have a number-theoretic origin. 

The number-theoretic objects appearing are related to \emph{Pisot-Salem numbers} and rings of integers, and are defined as such. Let $K$ be a number field, let $S_K$ denote the set of all places of $K$ and let $S \subseteq S_K$ be a finite subset. We can define the \emph{set of $S$-integers} as
\begin{equation}
    \mathcal{O}_{K,S} := \{\alpha \in K : \forall w \in S_K \setminus S, \vert \alpha \vert_w \leq 1\}. \label{Eq: Pisot}
\end{equation}
The subsets $\mathcal{O}_{K,S}$ are approximate subrings (for they are approximate subgroups as shown in \cite[\S II.13.3]{meyer1972algebraic} and stable under multiplication) and they are rings if and only if $S$ contains all the Archimedean places of $K$. See \cite[\S 2.1.1]{machado2020apphigherrank} for more on these objects and \cite[Chapter II]{NeukirchAlgebraicNumberTheory} for background on number fields and their places.

Definition (\ref{Eq: Pisot}) encompasses a number of well-known objects. If $K = \mathbb{Q}$ and $S$ has the usual absolute value as its only element, then $\mathcal{O}_{K,S} = \mathbb{Z}$. Moreover, for any number field $K$, if $S$ is the set of all Archimedean valuations, then $\mathcal{O}_{K,S}$ is the ring of (algebraic) integers of $K$ and often denoted by $\mathcal{O}_K$. Finally, if $S = \{v\}$ and the completion $K_v$ of $K$ is isomorphic to $\mathbb{R}$, then $\mathcal{O}_{K,v}\setminus \{-1,1\}$ is the set of primitive Pisot--Salem numbers of $K$ (i.e. those Pisot-Salem numbers contained in $K$ that generate $K$ as a field)\footnote{In fact, here $K$ and $\mathcal{O}_{K,v}$ are identified with  $\sigma_v[K]$ and $\sigma_v[\mathcal{O}_{K,v}]$, respectively, as explained in the paragraph after Lemma \ref{Lemma: Arithmeticity}.}.  It is this last example that appears in Meyer's original work: 

\begin{theorem}\label{theorem: classical Meyer}[Meyer, \S II.13.3, \cite{meyer1972algebraic}]
    Let  $X\subseteq \mathbb{R}$ be an infinite uniformly discrete approximate subring. Then there are a real number field $K \subseteq \mathbb{R}$ and a place $v$ of $K$ such that $\mathbb{R} \cong K_v$ and through this identification, $X$ is commensurable with 
    $$\mathcal{O}_{K,v}= \{\alpha \in K : \forall w \in S_K \setminus \{v\}, |\alpha|_w \leq 1 \}.$$
    In other words, $X$ is commensurable with the set of primitive Pisot-Salem numbers of $K$. 
\end{theorem}

\begin{proof}
    According to Corollary \ref{Cor: approximate subring closed under mult}, there is an approximate subring $Y \subseteq \langle X \rangle$ commensurable with $X$ and closed under multiplication. By \cite[Prop. II.6.2]{schreiber1973approximations} applied to $Y$ and $\mathbb{R}$ seen as a $1$-dimensional $\mathbb{R}$-vector space, either $Y$ is contained in a compact neighbourhood of $0$ or it is syndetic in $\mathbb{R}$. Since $Y$ is uniformly discrete, either $Y$ is finite or $Y$ syndetic. So it is syndetic. Now, \cite[Thm. X]{meyer1972algebraic} implies that $Y$ is harmonious. And by \cite[Thm. VI]{meyer1972algebraic}, $Y$ is contained in a set $P$ of primitive Pisot--Salem numbers of some real number field $K \subseteq \mathbb{R}$. Since $Y$ is syndetic and $P$ is an approximate lattice, $Y$ is commensurable with $P$ (cf. \cite[Lem. A.4]{hru2}). 
\end{proof}

Our goal in this section is to produce general structure results about uniformly discrete approximate subrings inspired from and generalising Meyer's theorem. In particular, all our results will provide structural information \emph{up to commensurability}. For the sake of simplicity, we will stick to results concerning \emph{real} associative algebras, but these results should hold for finite-dimensional (associative) algebras over all local fields. Throughout this section, we will use elementary results regarding associative algebras, which can be found in \cite{PierceAssociativeAlgebras} (see also Appendix \ref{section: unitality}). Throughout, all algebras will be assumed associative. 

Our main result (Theorem \ref{Theorem: Gen. Pisot Meyer}) will extend Meyer's Theorem to semi-simple associative algebras. Note that by Wedderburn's theorem, finite-dimensional semi-simple associative algebras over the reals are finite products of the matrix algebras of the form $M_n(D)$ where $n \geq 1$, and (by Frobenius theorem)  $D$ is either the real, complex or quaternion division algebra.

\subsection{The cut-and-project construction}

In this section, we will exploit the existence of locally compact models for rings \cite{Kru} in the improved form obtained in Proposition \ref{proposition: improved target space}. When dealing with uniformly discrete approximate subrings, this can be understood using the notion of a \emph{cut-and-project scheme} borrowed from the work of Meyer on aperiodic order and uniformly discrete approximate subgroups of abelian groups. Namely:

\begin{definition}
    A \emph{(ring) cut-and-project scheme} is a triple $(A,B,\Delta)$ consisting of a finite-dimensional $\mathbb{R}$-algebra $A$, a finite-dimensional $\mathbb{R}$-algebra $B$ and a (uniformly) discrete subring $\Delta \subseteq A \times B$ projecting injectively to $A$ and whose projection to $B$ spans $B$. 

    Given a symmetric relatively compact neighbourhood $W_0$ of $0$ (the \emph{window}), the subset 
    $$ M:=\pi_A\left[ \Delta \cap \left(A \times W_0\right)\right]$$
    is called a \emph{weak model set},  where $\pi_A: A \times B \rightarrow A$ is the natural projection. We say it is a \emph{model set} if $\Delta$ is co-compact (i.e. if there is a compact subset $K \subseteq A \times B$ such that $\Delta + K = A \times B$). We call moreover $A$ the \emph{direct space} and $B$ the \emph{internal space}.
\end{definition}

Our first result is:

\begin{proposition}\label{Proposition: Weak models for discrete app rings}
    Let $\Lambda$ be a uniformly discrete approximate subring of a finite-dimensional $\mathbb{R}$-algebra $A$. Then there are a finite-dimensional $\mathbb{R}$-algebra $B$ and a (uniformly) discrete subring $\Delta \subseteq A \times B$ projecting injectively to $A$ and whose projection to $B$ spans $B$, such that for all symmetric relatively compact neighbourhoods $W_0$ of $0$, $\Lambda$ is commensurable with the weak model set 
    $$ M:=\pi_A\left[ \Delta \cap \left( A \times W_0\right)\right]$$
    where $\pi_A: A \times B \rightarrow A$ is the natural projection.
\end{proposition}

\begin{proof}
    By Proposition \ref{proposition: improved target space}, there is an approximate subring $\Lambda'$, commensurable with $\Lambda$, and a locally compact model $\phi: \langle \Lambda' \rangle \rightarrow B$ where $B$ is a finite-dimensional $\mathbb{R}$-algebra with dense image satisfying $\phi^{-1}\left[W\right] \subseteq \Lambda'$ for some neighbourhood $W$ of $0$ in $B$. Now, $\Lambda'$ is an approximate subgroup for the additive structure and $\phi$ is a good model -- in the sense of \cite[Def. 1]{machado2019goodmodels} -- for that approximate subgroup. By \cite[Lem. 3.12]{machado2019goodmodels}, the graph 
    $$\Delta:=\{(\gamma, \phi(\gamma)) \in A \times B : \gamma \in \langle \Lambda' \rangle\}$$
    is a discrete subgroup (actually, subring since $\phi$ is a ring homomorphism) in $A \times B$ and, by Remark \ref{remark: preimage of a compact set}, $\Lambda'$ is commensurable with $\pi_A\left[ \Delta \cap \left(A \times W_0\right)\right]$ for any choice of symmetric relatively compact neighbourhood $W_0$ of $0$ in $B$. This completes the proof.

    For the sake of completeness, let us recall the short proof of discreteness of $\Delta$. Since $\Delta$ is a group, it is enough to show that $0$ is isolated in $\Delta$. Recall that the neighbourhood $W$ of $0$ in $B$ has been chosen so that $\phi^{-1}\left[W\right] \subseteq \Lambda'$. Now, let $V \subseteq A$ be a neighbourhood of $0$ such that $\Lambda' \cap V =\{0\}$. Such $V$ exists because $\Lambda'$ is discrete as it is covered by finitely many translates of the uniformly discrete subset $\Lambda$ of $A$.
    
    We claim that $\Delta \cap \left(V \times W\right) = \{(0,0)\}$. Indeed, if $(\gamma, \phi(\gamma)) \in \Delta \cap \left(V \times W\right)$, then $\phi(\gamma) \in W$ so $\gamma \in \Lambda'$. Thus, $\gamma \in \Lambda' \cap V =\{0\}$. So $\gamma = 0 $ and $(\gamma, \phi(\gamma))=(0,0)$ which concludes. 
\end{proof}

We will see that syndetic approximate subrings enjoy particularly striking structural results. Recall that a subset $X$ of a locally compact abelian group $A$ is \emph{syndetic} if there is a compact subset $K \subseteq A$ such that $X + K = A$. 

We have: 

\begin{lemma}\label{lemma: M is an approximate subring}
    With notation as above, $M$ is a uniformly discrete approximate subring. It is syndetic if and only if $\Delta$ is co-compact. 
\end{lemma}

\begin{proof}
    According to \cite[Prop. 2.13]{BjHa}, $M$ is a uniformly discrete approximate subgroup and it is syndetic if and only if $\Delta$ is co-compact. Note that in \cite[Prop. 2.13]{BjHa}, it is assumed that the projection of $\Delta$ to the second coordinate is dense and $\Delta$ has finite co-colume. However, the latter assumptions in not used in the proof; the former one is used, but it is easy to deduce that it can be weakened to saying that the projection of $\Delta$ to the second coordinate is syndetic.  
    
    It only remains to show that $M$ is \emph{a fortiori} an approximate subring.  Note first that if $W_1$ is any other relatively compact subset in $B$, $\pi_A\left[\Delta \cap \left(A \times W_1\right)\right]$ is covered by finitely many (additive) translates of $M$, see Remark \ref{remark: preimage of a compact set}. But 
    $$ M\cdot M = \pi_A\left[\Delta \cap \left(A \times W_0\right)\right]\pi_A\left[\Delta \cap \left(A \times W_0\right)\right] \subseteq \pi_A\left[\Delta \cap\left(A \times W_0^2\right)\right] \subseteq F + M$$
    for some finite subset $F \subseteq A$. Here, the last inclusion is a consequence of the first observation of this paragraph. 
\end{proof}


So Proposition \ref{Proposition: Weak models for discrete app rings} simply asserts that all uniformly discrete approximate subrings are commensurable with weak model sets. In other words, Proposition \ref{Proposition: Weak models for discrete app rings} asserts the existence of an \emph{internal space} ($B$ in the above) which governs the structure of the approximate subring under consideration. 

An idea of Schreiber \cite[Prop. 2]{schreiber1973approximations} complements Meyer-type results in the \emph{direct space} ($A$ in the above). Adapted to the context of rings, it helps measure the discrepancy between weak model sets and model sets in terms of two-sided ideals in the direct space.

\begin{proposition}\label{Prop: Schreiber}
    Let $\Lambda$ be an approximate subring in a finite-dimensional real algebra $A$. Suppose that $\Lambda$ spans $A$. Then there are a compact subset $K \subseteq A$ and a unique two-sided ideal $I$ such that 
    $$ \Lambda \subseteq I + K \text{ and } I \subseteq \Lambda + K.$$
\end{proposition}

\begin{proof}
    According to \cite[Prop. II.6.2]{schreiber1973approximations}, there are a vector subspace $V \subseteq A$ and $K \subseteq A$ compact such that 
    $$ \Lambda \subseteq V + K \text{ and } V \subseteq \Lambda + K.$$
    It is enough to show that $V$ is a two-sided ideal. To do so, it suffices to show that for every $\lambda \in \Lambda$, $\lambda V \subseteq V$ and $V \lambda \subseteq V$. We have

    $$ \lambda V \subseteq \lambda( \Lambda + K) \subseteq \Lambda_2 + \lambda K \subseteq \Lambda + K' \subseteq V + K + K'$$
    for some further compact subset $K'$. But both $\lambda V$ and $V$ are vector subspaces of the finite-dimensional space $A$, so $\lambda V \subseteq V$. The inclusion $V\lambda \subseteq V$ is obtained in a similar fashion. Finally, uniqueness is clear.
\end{proof}

Putting together these results, we will prove the main result: 

\begin{theorem}[Ring cut-and-project]\label{Theorem: Ring cap}
    Let $\Lambda$ be a uniformly discrete approximate subring in a finite-dimensional real algebra $A$. There are: 
    \begin{enumerate}
        \item a finite-dimensional real algebra $A'$;
        \item a uniformly discrete co-compact subring $\Delta \subseteq A'$;
        \item two ideals $I,J \subseteq A'$;
    \end{enumerate}
    such that:
    \begin{enumerate}
        \item $I \cap J =\{0\}$;
        \item there is a homomorphism $\pi:A'  \to A$ with kernel $J$;
        \item $\Lambda$ is commensurable with $\pi\left[\Delta
    \cap \left(I + W_0\right)\right]$ for any relatively compact symmetric neighbourhood $W_0$ of $0$ in $A'$, and $J \cap \Delta =\{0\}$. 
    \end{enumerate}
\end{theorem}

\begin{proof}
Let $(A,B, \Delta)$ denote the cut-and-project scheme provided by Proposition \ref{Proposition: Weak models for discrete app rings} applied to $\Lambda$. Let $M=\pi_A[\Delta\cap (A \times W)]$ be a weak model set commensurable with $\Lambda$ obtained in Proposition \ref{Proposition: Weak models for discrete app rings}. Choose $W$ so large that the span of $\Delta$ coincides with the span of $\pi_A^{-1}[M] \cap \Delta$, and define $A'$ to be this common span. Note that since $\Delta$ is a subgroup of $A'$ that spans $A'$, $\Delta$ must be co-compact in $A'$. Set now $\pi := (\pi_A)_{\vert A'}$, $J := \ker \pi = A' \cap \left(\{0\} \times B\right)$, and $I$ the two-sided ideal and $K$ the compact subset of $A'$ given by applying Proposition \ref{Prop: Schreiber} to $\pi^{-1}\left[M\right] \cap \Delta$. (Note that $\pi^{-1}\left[M\right] \cap \Delta$ is an approximate subring since $M$ is an approximate subring, $\Delta$ is a subring, and $M \subseteq \pi[\Delta]$.) We claim that such assignments satisfy the conclusions of the theorem.

    Indeed, (2) is obvious by construction. For (3) note first that $J \subseteq \{0\} \times B$, so $J \cap \Delta$ is trivial by injectivity of the projection  $(\pi_A)|_\Delta$. To show that $\Lambda$ is commensurable with $\pi\left[\Delta \cap \left(I + W_0\right)\right]$, first consider the case when $W_0$ contains $K$. Then 
$$\pi\left[\Delta \cap \left(I + W_0\right)\right] \supseteq \pi\left[\Delta \cap \left(I + K\right)\right] \supseteq \pi\left[\Delta \cap \pi^{-1}\left[M\right]\right] =M.$$ On the other hand, 
\begin{align*}
\pi\left[\Delta \cap \left(I + W_0\right)\right] \subseteq \pi\left[\Delta \cap \left(\left(\Delta \cap \pi^{-1}\left[M\right]\right) + K + W_0\right)\right] = \\ \pi\left[\Delta \cap \pi^{-1}\left[M\right]\right] + \pi\left[\Delta \cap \left(K + W_0\right)\right] = M + \pi\left[\Delta \cap \left(K + W_0\right)\right]
\end{align*} 
Moreover, since  $K + W_0$ is compact, by Remark \ref{remark: preimage of a compact set}, $\pi\left[\Delta \cap \left(K + W_0\right)\right]$ is covered by finitely many translates of $M$. Thus, using  Lemma \ref{lemma: M is an approximate subring}, we conclude that $\pi\left[\Delta \cap \left(I + W_0\right)\right]$ is commensurable with $M$ which is commensurable with $\Lambda$, as required. Now, consider an arbitrary relatively compact symmetric neighborhood $W_1$ of $0$ in $A'$, and fix $W_0$ as above.
%
%
We can find a third symmetric relatively compact neighborhood $W_2$ of $0$ in $A'$ such that $2W_2 \subseteq W_1$. 
Now, the subsets $I+W_0$, $I + W_1$ and $I+W_2$ are commensurable, and so $\Delta \cap (I + 2W_0)$, $\Delta \cap (I + 2W_1)$ and $\Delta \cap (I + 2W_2)$ are commensurable by \cite[Lem. 2.3(2)]{machado2019goodmodels}. Hence, $\pi\left[\Delta \cap (I + 2W_1)\right]$ and $\pi\left[\Delta \cap (I + 2W_2)\right]$ are commensurable with $\pi\left[\Delta \cap (I + 2W_0)\right]$  which, in turn, is commensurable with $\Lambda$ as shown above. But 
    $$ \pi\left[\Delta \cap (I + 2W_2)\right] \subseteq \pi\left[\Delta \cap (I + W_1)\right] \subseteq \pi\left[\Delta \cap (I + 2W_1)\right], $$
    so $\pi\left[\Delta \cap (I + W_1)\right]$ is also commensurable with $\Lambda$. This concludes the proof of (3).

     It remains to prove (1). Suppose otherwise, then $I \cap J$ is not trivial and, in particular, unbounded. 
We know that $I \subseteq \left(\pi^{-1}\left[M\right] \cap \Delta\right) + K$. 
Since $I \cap J$ is unbounded and $K$ is bounded, we can find a sequence $(x_n)_{n \geq 0}$ of pairwise distinct elements of $\pi^{-1}\left[M\right] \cap \Delta$ such that $(x_n + K) \cap I \cap J \neq \emptyset$. Thus, the sequence $(\pi(x_n))_{n \geq 0}$ consists of pairwise distinct elements of $M$ contained in the compact subset $-\pi[K]$. This contradicts  uniform discreteness of $M$.
\end{proof}


As already mentioned, the structure of syndetic approximate subrings is even more clear. In fact, we stumble upon Meyer's cut-and-project construction once more in this case: 

\begin{corollary}
    If $\Lambda$ is moreover syndetic, then $I \cong A$ and $A' \cong I \times J$.
\end{corollary}

\begin{proof}
    Let $A',\Delta, I,J, \pi$ be given by Theorem \ref{Theorem: Ring cap}. We have that $\Lambda$, hence $\pi\left[ \Delta
    \cap \left(I + W_0\right)\right]$, is syndetic for some $W_0$ compact. Hence, $\pi[I]$ is co-compact in $A$. Since it is a vector subspace, $\pi[I] = A$. But $J \cap I =\{0\}$ and $J$ is the kernel of $\pi$. So $I$ is isomorphic to $A$ via $\pi$ and $A' \cong I \times J$. 
\end{proof}

\subsection{Notions of arithmeticity}

The notion of arithmeticity we will be interested in is given by the sets of $S$-integers defined at the beginning of this section. Discrete subrings of algebras often have a number-theoretic origin, and so will  uniformly discrete approximate subrings. We make a repeated use of the elementary theory of associative algebras throughout this part, background can be found in \cite{PierceAssociativeAlgebras}. 
As explained in Corollary \ref{Cor: Unitality}, all semi-simple finite-dimensional algebras are unital.

To be able to state the results of this section precisely, we first have to introduce a notion of \emph{irreducibility}.

\begin{definition}
    Let $A$ be a finite-dimensional real algebra and $
\Lambda \subseteq A$ be a uniformly discrete approximate subring. We say that $\Lambda$ is \emph{reducible} if there is a decomposition of $A$ into a non-trivial direct product $B \times C$ such that $\Lambda$ is commensurable with $\left(B \cap 2\Lambda\right) \times \left(C \cap 2\Lambda \right)$. 

Otherwise, we say that $\Lambda$ is {\em irreducible}. 
\end{definition}

While the definition makes sense for general rings, finite dimension implies the existence of a decomposition into irreducibles: 

\begin{lemma}
    Let $A$ be a finite-dimensional real algebra and $
\Lambda \subseteq A$ be a uniformly discrete approximate subring. There is a non-trivial decomposition $A = A_1 \times \cdots \times A_r$ and irreducible approximate subrings $\Lambda_i \subseteq A_i$ for $i=1, \ldots, r$ such that $\Lambda$ is commensurable with $\Lambda_1 \times \cdots \times \Lambda_r$. 
\end{lemma}

\begin{proof}
    This is a straightforward consequence of the fact that any descending sequence of subalgebras must stabilise and that for any subalgebra of $B$ of $A$, the subset $2\Lambda \cap B$ is also an approximate subring, see \cite[Lem. 2.9]{Dries}.
\end{proof}

We can now explore the number-theoretic aspects of irreducible co-compact discrete subrings.

\begin{lemma}\label{Lemma: Arithmeticity}
    Let $\Delta$ be an irreducible discrete co-compact subring in a finite-dimensional semi-simple real algebra $A$. There are a number field $K$, a torsion-free finite rank $\mathcal{O}_K$-algebra $A_{\mathcal{O}_K}$ and  $R$ the set of Archimedean places of $K$ such that 
    $$ A \cong \prod_{v \in R} K_v \otimes_{\mathcal{O}_K} A_{\mathcal{O}_K}$$
    and under this identification 
    $ \Delta$ is commensurable with  $A_{\mathcal{O}_K}$ embedded diagonally. Moreover, for all $v \in R$, $K_v \otimes_{\mathcal{O}_K} A_{\mathcal{O}_K}$ is simple.
\end{lemma}

This result and its proof are inspired by celebrated `arithmeticity' results due to Margulis for lattices (i.e. discrete and finite co-volume subgroups) in simple Lie groups, \cite{Margulis}.  Recall first that for any number field $K \subseteq \mathbb{C}$, the set of Archimedean places admits a concrete description: each infinite place $v$ corresponds to an absolute value $\lambda \in K \mapsto |\sigma(\lambda)|$ where $\vert \cdot \vert$ denotes the complex modulus and $\sigma\colon K \rightarrow \mathbb{C}$ denotes a field embedding; the completion $K_v$ is then simply built as the closure of $\overline{\sigma(K)} \subseteq \mathbb{C}$ in the usual topology and is equal to either $\mathbb{R}$ or $\mathbb{C}$, see \cite[Prop. II.9.1]{NeukirchAlgebraicNumberTheory} for this and more. In the case when $K_v=\mathbb{R}$, the above embedding $\sigma$ yielding $v$ is unique, whereas for $K_v=\mathbb{C}$ there are exactly two embeddings of $K$ into $\mathbb{C}$ yielding $v$, namely $\sigma$ and $\overline{\sigma}$ (i.e. the composition of $\sigma$ and the complex conjugate). In Lemma \ref{Lemma: Arithmeticity} and other statements below (and also in Theorem \ref{theorem: classical Meyer} and in the discussion preceding it), the rings $\mathcal{O}_{K}$ and $\mathcal{O}_{K,v}$ are identified with $\sigma_v[\mathcal{O}_{K}]$ and $\sigma_v[\mathcal{O}_{K,v}]$, respectively, for any choice of an embedding $\sigma_v$ of $K$ into $\mathbb{C}$ yielding the place $v$ in question, e.g. $K_v \otimes_{\mathcal{O}_K} A_{\mathcal{O}_K}$ really means $K_v \otimes_{\sigma_v[\mathcal{O}_K]} \sigma_v[A_{\mathcal{O}_K}]$ where $\sigma_v[A_{\mathcal{O}_K}]$ is the $\sigma_v[\mathcal{O}_K]$-algebra naturally induced by $\sigma_v$.

\begin{proof}
    Let $A_\mathbb{Q}$ denote the $\mathbb{Q}$-algebra generated by $\Delta$. We first notice that $A \cong \mathbb{R} \otimes_\mathbb{Q} A_\mathbb{Q}$. Indeed, the inclusion map $A_\mathbb{Q} \rightarrow A$ gives rise to a natural homomorphism $\pi\colon \mathbb{R} \otimes_\mathbb{Q} A_\mathbb{Q} \rightarrow A$. Since $A_\mathbb{Q}$ contains $\Delta$, its  $\mathbb{R}$-span must be $A$. Thus, $\pi$ is surjective. But the real dimension of $\mathbb{R} \otimes_\mathbb{Q} A_\mathbb{Q}$ is equal to the dimension of $A_\mathbb{Q}$ over $\mathbb{Q}$, itself equal to the rank as an additive subgroup of $\Delta$. But $\Delta$ has rank $\dim_\mathbb{R} A$ because it is a discrete and co-compact subgroup of $A_\mathbb{R}$. Which concludes. As an immediate consequence, $A_\mathbb{Q}$ is semi-simple \cite[\S 4.1 Cor. a and \S 4.4 Prop.]{PierceAssociativeAlgebras}.

    By irreducibility, we have furthermore that $A_\mathbb{Q}$ is simple as a $\mathbb{Q}$-algebra. Indeed, otherwise by semi-simplicity there are two proper $\mathbb{Q}$-subalgebras $B,C$ such that $A_\mathbb{Q} = B \times C$. But $\Delta \cap B$ and $\Delta \cap C$ must have $\mathbb{Q}$-span equal to $B$ and $C$ respectively according to the following claim.     
    \begin{clm}
        Let $V$ be a finite-dimensional $\mathbb{Q}$-vector space and $\Gamma \subseteq V$ be a subgroup whose $\mathbb{Q}$-span is $V$. Then for every subspace $W \subseteq V$, the subgroup $\Gamma \cap W$ spans $W$. 
    \end{clm}

    \begin{clmproof}
          Since $\Gamma$ is a subgroup spanning $V$, for any $a \in V$ there is $m \geq 0$ such that $m\cdot a \in \Gamma$. So for any basis $(e_1, \ldots, e_r)$ of $W$ there is $m > 0$ such that $(m\cdot e_1, \ldots, m\cdot e_r)$ is  contained in $\Gamma$. But it is still a basis of $W$. This concludes.
    \end{clmproof}
   \noindent Hence, $\left(\Delta \cap B \right)\times \left(\Delta \cap C\right)$ is a subring of $\left(\mathbb{R} \otimes_\mathbb{Q} B\right) \times\left(\mathbb{R} \otimes_\mathbb{Q} C \right)= A$ whose $\mathbb{Q}$-span is $A_\mathbb{Q}$ and is contained in $\Delta$. Since they are finitely generated groups, they must have the same rank, equal to $\dim A_\mathbb{Q}$. So $\left(\Delta \cap B\right) \times \left(\Delta \cap C\right)$ and $\Delta$ are commensurable, contradicting irreducibility of $\Delta$. 

    From now on, we identify $A$ and $\mathbb{R} \otimes_\mathbb{Q} A_\mathbb{Q}$ through the map $\pi$ above. For any field $K$ containing $\mathbb{Q}$, write $A_K : = K \otimes_\mathbb{Q} A_\mathbb{Q}$. Note that because $A_{\mathbb{Q}}$ is semi-simple, so is $A_K$ for any extension $K$ of $\mathbb{Q}$. Since $A_\mathbb{C}$ is semi-simple, we have $A_\mathbb{C} = \prod_{i=1}^n A_i$ where the $A_i$'s are simple $\mathbb{C}$-subalgebras and this decomposition is unique. By uniqueness, the obvious Galois action of $Gal(\mathbb{C}/\mathbb{Q})$ on $A_\mathbb{C}$ permutes the simple subalgebras of $A_i$. 
    
    Because $\Delta$ is irreducible, the action on $\{A_1, \ldots, A_n\}$ is moreover transitive. Indeed, suppose otherwise. Up to reindexing, let $A_1, \ldots, A_r$ be the Galois orbit of $A_1$ where $r < n$. Then both $B_1 := A_1 \times \ldots \times A_r$ and $B_2 := A_{r+1} \times \ldots \times A_n$ are proper subalgebras of $A_\mathbb{C}$ such that $A_\mathbb{C} = B_1 \times B_2$ and both are invariant under the Galois action. Let $\delta \in \Delta$, then $\delta = (\delta_1, \delta_2)$ with $\delta_i \in B_i$ for $i=1,2$. If $\sigma \in Gal(\mathbb{C}/\mathbb{Q})$, then 
    $$(\delta_1, \delta_2) = \delta = \sigma(\delta) = (\sigma(\delta_1), \sigma(\delta_2)).$$
    Since $\sigma(\delta_i) \in B_i$ for $i=1,2$, we find that $\delta_i =\sigma(\delta_i)$. But this holds for all $\sigma \in Gal(\mathbb{C} / \mathbb{Q})$. So $\delta_i \in A_\mathbb{Q}$ for $i=1,2$. Thus, $A_\mathbb{Q} \cap B_i$ contains the projection of $A_\mathbb{Q}$ to $B_i$. In particular, $A_\mathbb{Q} \cap B_i$ has $\mathbb{C}$-span equal to $B_i$. We conclude that $\Delta \cap B_i$ also has $\mathbb{C}$-span equal to $B_i$. Finally, $\Delta':=\left(\Delta \cap B_1\right) \times \left(\Delta \cap B_2\right) $ is a subring contained in $\Delta$ and has $\mathbb{C}$-span equal to $A$. So they must have equal rank. Hence, $\Delta'$ and $\Delta$ are commensurable, which contradicts irreducibility of $\Delta$.

    Let now $H \subseteq Gal(\mathbb{C}/\mathbb{Q})$ denote the stabiliser of $A_1$. 
It is a closed (in the pointwise convergence topology) subgorup of index $n$ in $Gal(\mathbb{C}/\mathbb{Q})$.
Let $K$ be the subfield of $\mathbb{C}$ fixed by $H$. By Galois theory, we know that $Gal(\mathbb{C}/K)=H$ and so $[K:\mathbb{Q}]=n$\footnote{Since the extension $\mathbb{Q} \subseteq \mathbb{C}$ is non-algebraic, it order to use the classical Galois correspondence, one should reduce the situation to the Galois extension $\mathbb{Q} \subseteq \mathbb{Q}^{alg}$ by replacing $H$ by the image of $H$ under the restriction to $Gal(\mathbb{Q}^{alg}/\mathbb{Q})$. We leave the details as an exercise.}.
Choose now any $\delta \in A_\mathbb{Q}$ and $\sigma \in H$. Write $\delta =(\delta_1, \ldots, \delta_n)$ where $\delta_i \in A_i$ for $i = 1, \ldots, n$. As in the fourth paragraph, notice that since $\sigma$ stabilises $A_1$, there are  $\delta_i'$ for $i=2, \ldots, n$ such that
    $$(\delta_1, \ldots, \delta_n) = \delta = \sigma(\delta) = (\sigma(\delta_1), \delta'_2, \ldots, \delta'_n)$$
    which implies $\delta_1 = \sigma(\delta_1)$. Since $\sigma \in H$ is arbitrary and $K$ is the fixed field of $H$, we have $\delta_1 \in A_K$. So the projection $\pi_1[A_\mathbb{Q}]$ of $A_\mathbb{Q}$ to $A_1$ is contained in $A_K \cap A_1$. Moreover, by simplicity of $A_\mathbb{Q}$, it is either trivial or isomorphic to $A_\mathbb{Q}$, and since its $\mathbb{C}$-span is equal to $A_1$, it is isomorphic to  $A_\mathbb{Q}$. In particular, 
$\pi_1[A_\mathbb{Q}]$ has $\mathbb{Q}$-dimension $\dim_\mathbb{Q} A_\mathbb{Q} = \dim_\mathbb{R} A$. 
Since the Galois action is transitive on the $A_i$'s, we have $\dim_\mathbb{R} A = \dim_\mathbb{C} A_\mathbb{C} =  n \dim_\mathbb{C} A_1$. Now, the inclusion $A_K \cap A_1 \subseteq A_1$ induces an injective homomorphism of $\mathbb{C}$-algebras $\mathbb{C} \otimes_K (A_K \cap A_1) \hookrightarrow A_1$. This yields that $A_K \cap A_1$ has $K$-dimension at most $\dim_\mathbb{C} A_1 = \frac{\dim_\mathbb{R} A}{n}$. 
So for $B:=A_K \cap A_1$ we have 
\begin{align*}
\dim_{\mathbb{R}}(A) &= \dim_{\mathbb{Q}} (\pi_1[A_\mathbb{Q}]) \leq \dim_{\mathbb{Q}}(\Span_K(\pi_1[A_\mathbb{Q}]))= \dim_{\mathbb{K}}(\Span_K(\pi_1[A_\mathbb{Q}])) \cdot [K : \mathbb{Q}]\leq \\ & \leq  \dim_K(B) \cdot n \leq \dim_{\mathbb{R}}(A),
\end{align*}
which implies equality everywhere, and hence $\pi_1[A_\mathbb{Q}] =B$. Thus, the projection $p\colon A_\mathbb{Q} \rightarrow B$ is bijective. In particular, $B$ has $\mathbb{C}$-span equal to $A_1$ and $A_1 \cong \mathbb{C} \otimes_K B$ via the natural map.

    We now identify the $\mathcal{O}_K$-algebra in $B$. Let $e_1, \ldots, e_d$ be a $K$-basis of $B$. Upon considering the basis $m\cdot e_1, \ldots, m \cdot e_d$ for some $m$ sufficiently large, we may assume that for all $1 \leq i,j \leq d$, $e_i \cdot e_j \in \sum_{i} \mathcal{O}_K \cdot e_i$. Write $A_{\mathcal{O}_K} := \sum_{i} \mathcal{O}_K \cdot e_i$ which has an obvious $\mathcal{O}_K$-algebra structure. Since $\Delta$ is a finitely generated group, there is also an integer $m > 0$ such that $m \cdot p[\Delta] :=\{mx  : x \in p[\Delta] \} \subseteq A_{\mathcal{O}_K}$. Conversely, we have seen that the $\mathbb{Q}$-span of $p[\Delta]$ is equal to $B$. Since $A_{\mathcal{O}_K}$ is finitely generated as an $\mathcal{O}_K$-algebra and $\mathcal{O}_K$ is finitely generated as an abelian group, $A_{\mathcal{O}_K}$ is finitely generated as an abelian group. So there is  an integer $m' > 0$ such that $m' \cdot A_{\mathcal{O}_K} \subseteq p[\Delta]$. In particular, the rank of $A_{\mathcal{O}_K}$ coincides with the rank of $p[\Delta]$ which (by injectivity of $p$) is the same as the rank of $\Delta$. So the rank of $A_{\mathcal{O}_K}$ equals $\dim_{\mathbb{Q}}(A_{\mathbb{Q}})$.
    
    Finally, since $Gal(\mathbb{C} /\mathbb{Q})$ acts transitively on $\{A_1, \ldots, A_n\}$ and $K$ is the fixed filed of the stabiliser of $A_1$, we can construct an isomorphism $$A_\mathbb{C} \cong \prod_{\sigma :K \rightarrow \mathbb{C}} \mathbb{C} \otimes_{\sigma\left[\mathcal{O}_K\right]} \sigma\left[A_{\mathcal{O}_K}\right].$$  Indeed, let $\{\sigma_1:=\id, \sigma_2, \ldots, \sigma_n \}= Hom(K, \mathbb{C})$. For $i=2, \ldots, n$ choose $\tau_i \in Gal(\mathbb{C}/\mathbb{Q})$ extending $\sigma_i$. For every $i =2, \ldots, n$ there is $j(i) \in \{1, \ldots, n\}$ such that $\tau_i[A_1] = A_{j(i)}$. 
If $i_1, i_2$ satisfy $j(i_1)=j(i_2)$, we have $\tau_{i_1}^{-1}\tau_{i_2}[A_1] = A_1$ which implies $\sigma_{i_1} = \sigma_{i_2}$, and so $i_1=i_2$. Upon reindexing, we may therefore assume that $\tau_i[A_1]=A_i$. 
Since $A_1$ was shown to be isomorphic with $\mathbb{C} \otimes_K B$ which is further isomorphic with $\mathbb{C} \otimes_{\mathcal{O}_K} A_{\mathcal{O}_K}$, we get that $A_i \cong \mathbb{C} \otimes_{\sigma_i\left[\mathcal{O}_K\right]} \sigma_i \left[A_{\mathcal{O}_K}\right]$. This gives rise to the isomorphism that we are looking for. We will now identify the image of $\Delta$ under this isomorphism up to commensurability. Define the map 
    \begin{align*}
        \phi \colon A_{\mathcal{O}_K} &\longrightarrow A_\mathbb{C} = A_1 \times \ldots \times A_n \\
        b &\longmapsto b+\tau_2(b)+ \cdots + \tau_n(b) =(b,\tau_2(b),\dots,\tau_n(b)).
    \end{align*} 
    Any other choice of $\tau_i$'s extending the $\sigma_i$'s provides the same map, 
so $\sigma \circ \phi = \phi$ for all $\sigma \in Gal(\mathbb{C}/\mathbb{Q})$. Hence, $\phi\left[A_{\mathcal{O}_K}\right]\subseteq A_\mathbb{Q}$. But $\phi\left[A_{\mathcal{O}_K}\right]$ has the same rank as $A_{\mathcal{O}_K}$ which was shown to be equal to $\dim_{\mathbb{Q}}(A_{\mathbb{Q}})$, and so its $\mathbb{Q}$-span is $A_\mathbb{Q}$. Since both groups $\Delta$ and  $\phi\left[A_{\mathcal{O}_K}\right]$ are finitely generated and their $\mathbb{Q}$-spans coincide, we conclude that there are positive integers $m_1,m_2$ such that $m_1\cdot \Delta \subseteq \phi\left[A_{\mathcal{O}_K}\right]$ and $m_2\cdot \phi\left[A_{\mathcal{O}_K}\right]\subseteq \Delta$. Hence, $\phi\left[A_{\mathcal{O}_K}\right]$ is commensurable with $\Delta$.  

    Finally,  $A_\mathbb{R} = A$ is the set of points in $A_\mathbb{C}$ that are fixed under complex conjugation. In the identification $A_\mathbb{C} \cong \prod_{\sigma :K \rightarrow \mathbb{C}} \mathbb{C} \otimes_{\sigma\left[\mathcal{O}_K\right]} \sigma\left[A_{\mathcal{O}_K}\right]$, complex conjugation corresponds to the map $\prod_{\sigma :K \rightarrow \mathbb{C}} \mathbb{C} \otimes_{\sigma\left[\mathcal{O}_K\right]} \sigma\left[A_{\mathcal{O}_K}\right] \rightarrow \prod_{\sigma :K \rightarrow \mathbb{C}} \mathbb{C} \otimes_{\sigma\left[\mathcal{O}_K\right]} \sigma\left[A_{\mathcal{O}_K}\right]$ that permutes the factors, sending  $\mathbb{C} \otimes_{\sigma\left[\mathcal{O}_K\right]} \sigma\left[A_{\mathcal{O}_K}\right]$ to $\mathbb{C} \otimes_{\overline{\sigma}[\mathcal{O}_K]} \overline{\sigma}[A_{\mathcal{O}_K}]$ by mapping $\lambda \otimes \sigma(a)$ to $\overline{\lambda} \otimes \overline{\sigma}(a)$ where $\overline{\sigma}$ denotes the composition of $\sigma$ and complex conjugation. The fixed points are therefore spanned over $\mathbb{R}$ by elements of the form:
    \begin{itemize}
        \item $1 \otimes \sigma(a)$ for $a \in A_{\mathcal{O}_K}$ in the coordinate $\sigma$ and $0$ elsewhere if $\overline{\sigma}=\sigma$;
        \item $1 \otimes \sigma(a)$ in the coordinate $\sigma$, $1 \otimes \overline{\sigma}(a)$ in the coordinate $\overline{\sigma}$  for $a \in A_{\mathcal{O}_K}$ and $0$ elsewhere if $\overline{\sigma}\neq\sigma$; and 
        \item $i \otimes \sigma(a)$ in the coordinate $\sigma$, $-i \otimes \overline{\sigma}(a)$ in the coordinate $\overline{\sigma}$ and $0$ elsewhere for $a \in A_{\mathcal{O}_K}$ if $\overline{\sigma}\neq\sigma$.
    \end{itemize}
    Choose now $R':=\{\sigma_1, \ldots, \sigma_r, \ldots, \sigma_{r+s}\}$ a set containing precisely one element of each pair $\{\sigma, \overline{\sigma}\}$ and such that for $i=1, \ldots, r$, we have $\sigma_i =\overline{\sigma_i}$ and for all $j = 1, \ldots, s$, $\sigma_{r+j} \neq \overline{\sigma_{r+j}}$. Then the map 
    $$\prod_{i=1}^r \mathbb{R} \otimes_{\sigma_i[\mathcal{O}_K]} \sigma_i\left[A_{\mathcal{O}_K}\right] \times \prod_{j=1}^s\mathbb{C} \otimes_{\sigma_{r+j}[\mathcal{O}_K]}\sigma_{r+j}\left[A_{\mathcal{O}_K}\right] \rightarrow \prod_{\sigma :K \rightarrow \mathbb{C}} \mathbb{C} \otimes_{\sigma\left[\mathcal{O}_K\right]} \sigma\left[A_{\mathcal{O}_K}\right] $$
    defined on each factor by 
    \begin{itemize}
        \item if $i=1, \ldots, r$, $x \otimes \sigma_i(a)$ is sent to $x \otimes \sigma_i(a)$ in coordinate $\sigma_i$ and $0$ elsewhere;
        \item if $j =1, \ldots, s$, $z \otimes \sigma_{r+j}(a)$ is sent to $z \otimes\sigma_{r+j}(a) $ in coordinate $\sigma_{r+j}$, $\overline{z} \otimes\overline{\sigma_{r+j}}(a) $ in coordinate $\overline{\sigma_{r+j}}$ and $0$ elsewhere;
    \end{itemize}
    is an isomorphism onto $A_\mathbb{R}$. According to the description of Archimedean places of $K$ in terms of embeddings $K \rightarrow \mathbb{C}$ we therefore have
    $$ A \cong \prod_{v \in R} K_v \otimes_{\mathcal{O}_K} A_{\mathcal{O}_K}$$
    and the lemma is proven.
\end{proof}

We can combine Lemma \ref{Lemma: Arithmeticity} and Theorem \ref{Theorem: Ring cap} to conclude the generalisation of Meyer's sum-product theorem: 

\begin{theorem}\label{Theorem: Gen. Pisot Meyer}
Let $\Lambda$ be an irreducible uniformly discrete infinite approximate subring in a finite-dimensional semi-simple $\mathbb{R}$-algebra $A$. Suppose that the $\mathbb{R}$-span of $\Lambda$ is $A$. Then the data $A', \Delta, I,J , \pi$ from Theorem \ref{Theorem: Ring cap} can be constructed as follows.

There is a number field $K \subseteq \mathbb{R}$ with the set of Archimedean places $R$, a finite rank $\mathcal{O}_K$-algebra $A_{\mathcal{O}_K}$ and $R_1 \subseteq R_2 \subseteq R$ such that $A' \cong \prod_{v \in R} K_{v} \otimes_{\mathcal{O}_K} A_{\mathcal{O}_K}$ and through this identification  
$$A \cong \prod_{v \in R_2} K_{v} \otimes_{\mathcal{O}_K} A_{\mathcal{O}_K},$$ $$I \cong \prod_{v \in R_1} K_v \otimes_{\mathcal{O}_K} A_{\mathcal{O}_K} \text{ and } J \cong \prod_{v \in R \setminus R_2 } K_v \otimes_{\mathcal{O}_K} A_{\mathcal{O}_K}$$ and $\Delta$ is commensurable with the diagonal embedding to $A_{\mathcal{O}_K}$. 

\end{theorem}

\begin{proof}[Proof of Theorem \ref{Theorem: Gen. Pisot Meyer}.]
    
    Let $A', \Delta ,I,J, \pi: A' \rightarrow A$ be provided by Theorem \ref{Theorem: Ring cap} and its proof applied to $\Lambda \subseteq A$. By construction, there is a compact subset $K \subseteq A$ such that 
    $$\pi[I] \subseteq \Lambda + K \text{ and } \Lambda \subseteq \pi[I] + K.$$
    By Proposition \ref{Prop: Schreiber} and because $\Lambda$ spans $A$, there is a two-sided ideal $I' \subseteq A$ and a compact subset $K' \subseteq A$ such that 
    $$I' \subseteq \Lambda + K' \text{ and } \Lambda \subseteq I' + K'.$$
    So 
        $$I' \subseteq \pi[I] + K + K' \text{ and } \pi[I] \subseteq I' + K + K'.$$
    Hence, $\pi[I]=I'$. In particular, $\pi[I]$ is a two-sided ideal of $A$.  Since $A$ is semi-simple, we have that $\pi[I]$ is a factor of $A$ i.e. there is a two-sided ideal $I''$ of $A$ such that $A = \pi[I] \times I''$. 

    In addition, we claim that $\pi[A']=A$ as a result of the irreducibility of $\Lambda$. Let $V \subseteq \pi[A']$ denote a vector subspace of lowest dimension such that $\Lambda$ is covered by finitely many translates of $V$. Then for every 
$\lambda \in \Lambda$ that is not a zero divisor, we have that $\lambda \cdot \Lambda$ is commensurable with $\Lambda$. Indeed, since $\lambda$ is not a zero divisor, $\lambda(\pi[I]) \subseteq \pi[I]$ implies that $\lambda\cdot\pi[I] = \pi[I]$ because $\pi[I]$ has finite dimension. Thus, $\pi[I] \subseteq \lambda \cdot (\Lambda + K) = \lambda \cdot \Lambda + \lambda \cdot K$. In particular, 
    $\Lambda \subseteq \lambda \cdot \Lambda + \lambda \cdot K + K$.
    Hence, $\Lambda \subseteq \lambda \cdot \Lambda + \left(\Lambda_{1} \cap L\right)$ where $L$ denotes the compact subset  $\lambda \cdot K + K$. 
Since $\Lambda_1$ is covered by finitely many cosets of $\Lambda$, uniform discreteness of $\Lambda$ implies that $F:=\Lambda_{1} \cap L$ is finite, and so we find that $\Lambda$ and $\lambda \cdot \Lambda$ are indeed commensurable.
But $\Lambda$ is covered by finitely many translates of $V$, so $\lambda \cdot \Lambda$ -- and hence $\Lambda$ -- is covered by finitely many translates of $\lambda \cdot V$. By for instance \cite[Lem. 2.2]{machado2019goodmodels}, $\Lambda$ is covered by finitely many translates of $\lambda \cdot V \cap V$ and, by minimality, $\lambda \cdot V \cap V = V$ i.e. $\lambda \cdot V = V$. Now, for every $\lambda \in \Lambda$, there is $n \geq 0$ such that $n \cdot 1 - \lambda$ is not a zero divisor, so $(n \cdot 1 - \lambda) \cdot V =V$. This implies, $\lambda \cdot V \subseteq V$. But $\Lambda$ spans $A$, so $V$ is a left-ideal. Similarly, one proves that it is a right-ideal. Recall that $V \subseteq \pi[A']$, and we claim that $V=A$. Otherwise, since $A$ is semi-simple, there is a non-trivial two-sided ideal $V'$ such that $V \times V' = A$. Moreover, by \cite[Lem. 2.9]{Dries} and \cite[Lem. 2.2]{machado2019goodmodels}, $2\Lambda \cap V$ is an approximate subring commensurable with $\Lambda$. Hence, $\Lambda$ is commensurable with $(2\Lambda \cap V) \times \{0\} \subseteq V \times V'$ which contradicts irreducibility of $\Lambda$.
    

    Our next step is to show that $A'$ is semi-simple. Suppose otherwise. Write $B$ for the $\mathbb{Q}$-span of $\Delta$ and let $J(B)$ be the Jacobson radical of $B$. Then $B$ is a nilpotent (two-sided) ideal of $B$ by Corollary \ref{corollary: semisimplicity in unitization}. Since the restriction of $\pi$ to $\Delta$ is injective and the $\mathbb{Q}$-span of $\Delta$ is $B$, the restriction of $\pi$ to $B$ is also injective. Moreover, $\pi[J(B)]$ is a nilpotent ideal of $\pi[B]$. The $\mathbb{R}$-span of $\pi[J(B)]$ is a nilpotent ideal  of the $\mathbb{R}$-span of $\pi[B]$. But the $\mathbb{R}$-span of $B$ is $A'$ since $\Delta$ is co-compact in the finite-dimensional space $A'$, so the $\mathbb{R}$-span of $\pi[B]$ is $A$ by the previous paragraph. So the $\mathbb{R}$-span of $\pi[J(B)]$ is a nilpotent ideal of $A$. Since $A$ is semi-simple, it is the trivial ideal by \cite[\S 4.4 Cor.]{PierceAssociativeAlgebras}. Hence, $\pi[J(B)]$ is trivial, and so is $J(B)$ by injectivity of $\pi$ restricted to $B$. So $B$ has trivial Jacobson radical, i.e. $B$ is semi-simple. In particular, $B$ is unital (Corollary \ref{Cor: Unitality}) and so is $A'$. Now, there is a surjective map $\mathbb{R} \otimes_\mathbb{Q} B \rightarrow A'$ induced by the inclusion $B \subseteq A'$. So $A'$ is semi-simple as well.  

    We now wish to apply Lemma \ref{Lemma: Arithmeticity}. To do so, it remains to prove that $\Delta$ is irreducible. Assume, for the sake of contradiction, that $A'$ is the direct product of two non-trivial ideals $I_1$ and $I_2$ such that $I_1 \cap \Delta$ and $I_2 \cap \Delta$ span $I_1$ and $I_2$, respectively. If $I_1 \subseteq J$ or $I_2 \subseteq J$, then the restriction of $\pi$ to $\Delta$ is not injective, a contradiction. Thus, $\pi[I_1]$ and $\pi[I_2]$ are non-trivial ideals of $\pi[A']=A$, and $A$ is the direct product of $\pi[I_1]$ and $\pi[I_2]$ (relying here on the semi-simplicity of $A$). For each $i=1,2$, the subring $\Delta_i:=\Delta \cap I_i$ is co-compact in $I_i$. Let $I_i':= I_i \cap I$ and $I_i''$ be the ideal of $I_i$ such that $I_i = I_i' \times I_i''$. Then the restriction of $\pi$ to $I_i'$ is an isomorphism onto $\pi[I_i']$. Moreover, for any symmetric compact neighbourhood $W_i$ of $0$ in $I_i''$, $\pi\left[\Delta_i \cap \left(I_i' \times W_i \right)\right]$ is a uniformly discrete approximate subring of $A$. 
We conclude that 
\[ \Lambda':=\pi\left[\left(\Delta_1 \times \Delta_2\right) \cap (I_1' \times W_1 \times I_2' \times W_2)\right] \]
is reducible. Indeed,  
\begin{align*}
\Lambda' & =\pi\left[\left(\Delta_1 \cap (I_1' \times W_1)\right) \times \left(\Delta_2 \cap (I_2' \times W_2)\right)\right] =
\pi\left[\Delta_1 \cap (I_1' \times W_1)\right] \times \pi\left[\Delta_2 \cap (I_2' \times W_2)\right] \subseteq \\ &
\subseteq \pi[I_1] \times \pi[I_2] =A,
\end{align*}
so $2 \Lambda' \cap \pi[I_i] = 2 \pi\left[ \Lambda_i \cap (I_i' \times W_i)\right]$ is commensurable with $\pi\left[\Delta_i \cap \left(I_i' \times W_i \right)\right]$, which implies that $(2 \Lambda' \cap \pi[I_i]) \times  (2 \Lambda' \cap \pi[I_2])$ is commensurable with $\Lambda'$, hence $\Lambda'$ is reducible.
However, $\Delta_1 \times \Delta_2$ is contained in and commensurable with $\Delta$; thus, according to \cite[Lem. 2.3]{machado2019goodmodels}, $\Lambda'$ is commensurable with 
\[\pi\left[\Delta \cap (I_1' \times 2W_1 \times I_2' \times 2W_2)\right]. \]
Since by semisimplity of $A'$ we have  $I_1' \times I_2' = I$, the latter is commensurable with $\Lambda$ by Theorem \ref{Theorem: Ring cap}(3) (to use Theorem \ref{Theorem: Ring cap}(3), note that the set $I_1' \times 2W_1 \times I_2' \times 2W_2$ can be written as $I + (U \times (2W_1 \times 2W_2))$, where $U$ is any symmetric compact neighborhood of $0$ in $I$ and so $U \times (2W_1 \times 2W_2)$ is a compact symmetric neighborhood of $0$ in $A'$). So $\Lambda$ is reducible, a contradiction.

    Now, by Lemma \ref{Lemma: Arithmeticity} applied to $A'$ and $\Delta$ together with the above observation that $\pi[A']=A$, there is a number field $K$ with the set of Archimedean places $R$, a finite-dimensional  $\mathcal{O}_K$-algebra such that $K_v \otimes_{\mathcal{O}_K}A_{\mathcal{O}_K}$ is a simple real algebra for all $v \in R$ and $R_1 \subseteq R_2 \subseteq R$ such that $A' \cong \prod_{v \in R} K_{v} \otimes_{\mathcal{O}_K} A_{\mathcal{O}_K}$ and through this identification  $A \cong \prod_{v \in R_2} K_{v} \otimes_{\mathcal{O}_K} A_{\mathcal{O}_K}$, $I \cong \prod_{v \in R_1} K_v \otimes_{\mathcal{O}_K} A_{\mathcal{O}_K}$ and $J \cong \prod_{v \in R \setminus R_2 } K_v \otimes_{\mathcal{O}_K} A_{\mathcal{O}_K}$ and $\Delta$ is identified with the diagonal embedding to $A_{\mathcal{O}_K}$. This concludes the proof.
\end{proof}

\begin{remark}
    Using the uniqueness of decomposition into indecomposable - valid for larger classes of algebras (e.g. perfect algebras by Krull--Schmidt-type theorems) - one can in fact adapt the proof of Lemma \ref{Lemma: Arithmeticity} for these more general classes. We refrain from doing so for the sake of brevity.
\end{remark}

We conclude with a corollary of Theorem \ref{Theorem: Gen. Pisot Meyer} involving Pisot--Salem numbers:

 \begin{corollary}\label{corollary: gen. of Meyer to simple real algebras}
Let $X$ be an infinite uniformly discrete approximate subring in a finite-dimensional simple real algebra $A$. Assume that $X$ spans $A$. There are a number field $K$ and an Archimedean place $v$ such that $A$ admits a $K_v$-algebra structure extending the $\mathbb{R}$-algebra structure and a $K_v$-basis $(e_1, \ldots, e_d)$ of $A$ such that $X$ is commensurable with $\sum_{i=1}^d \mathcal{O}_{K,v} \cdot e_i.$
 \end{corollary}

\begin{remark}
    With $A, X, K, v$ as above. In the specific case where $A = M_n(\mathbb{R})$ for some $n \geq 1$, we have $K_v \cong \mathbb{R}$. Indeed, $K_v \cdot 1$ is then contained the centre of $A$ which is isomorphic to $\mathbb{R}$. Hence, in that case $\mathcal{O}_{K,v}\setminus \{-1,1\}$ is the set of primitive Pisot--Salem numbers of $K$ as defined in the introduction (see also the first page of Section \ref{Section: Sum-product for real algebras}, and be aware that this is under the identification of $K$ with $\sigma_v[K]$ and $\mathcal{O}_{K,v}$ with $\sigma_v[\mathcal{O}_{K,v}]$ as explained the paragraph after Lemma \ref{Lemma: Arithmeticity}).
\end{remark}

\begin{proof}[Proof of Corollary \ref{corollary: gen. of Meyer to simple real algebras}]
    Let us apply Theorem \ref{Theorem: Gen. Pisot Meyer} to $X$ and let $K$, $A_{\mathcal{O}_K}$, $R_1 \subset R_2 \subset R$ be as in the conclusions. Since $A$ is simple, $R_2$ is made of a single element $v \in R$. Since $X$ is infinite, $R_1$ contains at least one element. So $R_1 = R_2 = \{v\}$. Therefore, by Theorem \ref{Theorem: Gen. Pisot Meyer}, there exists a real algebra $A'$, a co-compact discrete subring $\Delta \subseteq A'$, and two ideals $I, J \subseteq A'$ such that $A'$ is the direct product $I \times J$ , $\pi\colon A' \to A$ is a projection with kernel $J$, and $X$ is commensurable with 
    \[ M := \pi\left[ \Delta \cap (I \times W) \right] \]
    for any compact neigbourhood $W \subseteq J$ of $0$. Let $W_0$ be such a neighbourhood to be chosen later. The algebra $A'$ furthermore decomposes as:
    \[ A' = \prod_{w \in R} K_w \otimes_{\mathcal{O}_K} A_{\mathcal{O}_K} \]
   and $\Delta$ is commensurable with the diagonal embedding of $A_{\mathcal{O}_K}$. 
    
    Moreover, $I = K_v \otimes_{\mathcal{O}_K} A_{\mathcal{O}_K}$ and $J$ corresponds to the product over the remaining places $R \setminus \{v\}$. Let $\{p_w\colon A_{\mathcal{O}_K} \rightarrow K_w \otimes_{\mathcal{O}_K} A_{\mathcal{O}_K}\}_{w \in R}$ denote the natural projections.  Under these identifications:
    \begin{enumerate}
        \item $\Delta$ is commensurable with $A_{\mathcal{O}_K}$;
        \item The projection of $\delta \in A_{\mathcal{O}_K}$ to $I$ corresponds to the embedding $p_v(\delta)$;
        \item The projection of $\delta\in A_{\mathcal{O}_K}$ to $J$ corresponds to the tuple of embeddings $(p_w(\delta))_{w \in R \setminus \{v\}}$.
    \end{enumerate}
    The condition $\delta \in I \times W_0$ is therefore equivalent to requiring that $(p_w(\delta))_{w \in R \setminus \{v\}} \in W_0$. Since $\Delta$ and $A_{\mathcal{O}_K}$ are commensurable we know by \cite[Lem. 2.3]{machado2019goodmodels}, that $M$, and hence $X$, is commensurable with:
    \[ \Sigma_{W_0} := \{ \delta \in A_{\mathcal{O}_K} : (p_w(\delta))_{w \in R \setminus \{v\}} \in W_0 \}. \]
    To complete the proof, we compare $\Sigma_{W_0}$ to the ``module generated'' by $\mathcal{O}_{K,v}$. Recall that $\mathcal{O}_{K,v} = \{ \alpha \in K : \forall w \neq v, |\alpha|_w \le 1 \}$. For any $K$-basis $(e_i)_{i=1}^d$ of $A_K:=K \otimes_{\mathcal{O}_K} A_{\mathcal{O}_K}$ contained in $A_{\mathcal{O}_K}$, we have the inclusion
    \[ \Lambda_0 := \sum_{i=1}^d \mathcal{O}_{K,v} \cdot e_i \subseteq \{ \delta \in A_{\mathcal{O}_K} : \forall w \in R \setminus \{v\}, \|p_w(\delta)\|_w \le C \} \]
    where $\|\cdot\|_w$ denotes any choice of $K_w$-vector space norm on $K_w \otimes_{\mathcal{O}_K}A_{\mathcal{O}_K}$ and $C$ is sufficiently large depending on the choice of norms and the basis. The latter subset in fact corresponds to $\Sigma_{W_0}$ for the choice of a window $W_0$ defined by a product of balls. In other words, we may choose $W_0$ such that $\Lambda_0 \subset \Sigma_{W_0}$. 

    But $\Lambda_0$ is a syndetic approximate subgroup of $A$, since $\mathcal{O}_{K,v}$ is a syndetic approximate subgroup in $K_v$. Moreover, $\Sigma_{W_0}$ is a uniformly discrete approximate subgroup. So $\Lambda_0$ and $\Sigma_{W_0}$ are commensurable according to \cite[Lem. A.4]{hru2}. Thus, $X$ is commensurable with $\sum_{i=1}^d \mathcal{O}_{K,v} \cdot e_i$.
\end{proof}

\begin{remark}\label{corollary: gen. of Meyer to semi-simple real algebras}
    In the context of semi-simple algebras, we can also obtain a result of a similar nature. Namely:

    Let $X$ be an infinite uniformly discrete approximate subring in a finite-dimensional semi-simple real algebra $A$. Assume that $X$ spans $A$ and is irreducible. There are a number field $K$ and two subsets of inequivalent archimedean places $R_1 \subseteq R_2$ such that $A$ admits a $\prod_{v \in R_2}K_v$-algebra structure extending the $\mathbb{R}$-algebra structure and a $\prod_{v \in R_2}K_v$-basis $(e_1, \ldots, e_d)$ of $A$ such that $X$ is commensurable with $\left(\sum_{i=1}^d \mathcal{O}_{K,R_1} \right)\cdot e_i.$

    We leave the verification of this result to the reader. 
    
\end{remark}

\section{Other structural results}\label{section: additional structural results}

In this section, we will prove results \ref{theorem: thickness and sum-product} - \ref{corollary: classification of finite approximate subfields} from the introduction. They are essentially also consequences of the existence of definable locally compact models of definable approximate subrings, but it is very convenient to use model-theoretic connected components which in fact stand behind definable locally compact models which is briefly recalled below.

Let $M$ be an infinite structure in a language $\mathcal{L}$, $X$ an approximate subring definable in $M$, $\C \succ M$ a sufficiently saturated elementary extension (so-called {\em monster model}), $R:=\langle X \rangle$, $\bar X:=X(\C)$, $\bar R :=\langle \bar X \rangle$. More generally, for a set $D$ which is definable in $M$, by $\bar D$ we will denote its interpretation $D(\C)$ in $\C$. In the definition of a monster model one usually assumes $\kappa$-saturation and strong $\kappa$-homogeneity (each elementary map between subsets of cardinality less than $\kappa$ extends to an automorphism of $\C$) for a strong limit cardinal greater than $|M| + |\mathcal{L}|$. It is folklore that such a model always exists. (In fact, in the applications below, it is enough to assume that $\C$ is $(2^{|M|+ |\mathcal{L}|})^+$-saturated.). By a bounded cardinal we mean any cardinal less than $\kappa$. 

The model-theoretic component ${\bar R}^{00}_M$ is defined as the smallest $M$-type-definable, bounded index two-sided ideal of $\bar R$. As stated in \cite[Proposition 3.1 ]{Kru} (which is based on \cite{GJK}), in this definition, instead of ``two-sided ideal'' we can equivalently write ``left ideal'' or ``right ideal'' or ``subring''. However, the existence of  ${\bar R}^{00}_M$ is a non-trivial issue. It is contained in  the main result of \cite{Kru}, but let us first recall more basic observations.

The equivalence between (1) and (3) in the next fact is \cite[Corollary 3.4]{Kru}. The equivalence between (2) and (3) is easy.

\begin{fact}
The following conditions are equivalent.
\begin{enumerate}
\item A definable locally compact model of $X$ exists.
\item There exists an $M$-type-definable two-sided  ideal of $\bar R:=\langle \bar X \rangle$ of bounded index. 
\item $\bar R^{00}_M$ exists.
\end{enumerate}
\end{fact} 

The next fact is \cite[Proposition 3.3]{Kru}. Here, the quotient ring  $\bar R/\bar R^{00}_M$ is equipped with the {\em logic topology}: a subset $F$ of $\bar R/\bar R^{00}_M$ is closed if $\pi^{-1}[F] \cap \bar X_n$ is type-definable (equivalently, $M$-type-definable) for every $n \in \mathbb{N}$, where $\pi \colon \bar R \to \bar R/\bar R^{00}_M$ is the quotient map.

\begin{fact}
Assume that ${\bar R}^{00}_M$ exists. Then the quotient map $R \to \bar R/\bar R^{00}_M$ is the universal definable locally compact model of $X$.
\end{fact} 

Finally, we have \cite[Theorem 4.1]{Kru}. In this statement, $(\RR,+)^{00}_M$ denotes the smallest $M$-type-definable, bounded index subgroup of $(\bar R,+)$, which always exists as $(R,+)$ is abelian (see \cite[Fact 4.2]{Kru}).

\begin{fact}\label{fact: 00 exists}
 $\RR^{00}_M$ exists and equals $(\RR,+)^{00}_M +\bar R (\RR,+)^{00}_M = (\RR,+)^{00}_M +\bar X (\RR,+)^{00}_M \subseteq 4\bar X + \bar X \cdot 4\bar X$.
\end{fact}

In a forthcoming paper of the first author with Mateusz Rzepecki, more formulas for  $\RR^{00}_M$ are given from which it follows that $4\bar X + \bar X \cdot 4\bar X$ can be decreased to $4\bar X + \bar X \cdot 2\bar X$, and there is an example showing that it cannot be further decreased to $4\bar X + \bar X \cdot \bar X$. Having this mind, in the statements below the set $4X + X \cdot 4X$ can be decreased to $4X + X \cdot 2X$. 

Recall Definition 4.1 from \cite{HKP} of a thick subset of $\bar R$. 

\begin{definition}\label{definition: thick}
A definable, additively symmetric subset $D$ of $\bar R$ is {\em $\bigvee$-thick} if for every sequence $(r_i)_{i<\lambda}$ of unbounded length which consists of elements of $\bar R$ there are $i<j<\lambda$ with $r_j -r_i \in D$.
\end{definition}

By compactness, we get the following \cite[Remark 4.4]{Kru}.

\begin{remark}\label{remark: thick} A definable, additively symmetric subset $D$ of $\RR$ is $\bigvee$-thick if and only if  for every $m \in \omega$ there exists a positive integer $N$ such that for every $r_0,\dots,r_{N-1} \in \bar X_m$ there are $i<j<N$ with $r_j -r_i\in D$. For any $N$ with this property, 
we will say that $D$ is {\em $N$-thick in $\bar X_m$}.
\end{remark}

More generally,

\begin{definition}\label{definition: thickness abstractly}
Let $Y$ be a symmetric subset of a group $G$ and $N$ be a positive integer. We say that a symmetric subset $D$ of $G$ is {\em $N$-thick in $Y$} if for every $r_0,\dots,r_{N-1} \in Y$ there are $i<j<N$ with $r_j^{-1}r_i\in D$. If $D \subseteq Y$, then instead of saying that $D$ is $N$-thick in $Y$  we will be also saying that $D$ is an {\em $N$-thick subset of $Y$}. By ``thick'' we mean ``$N$-thick'' for some $N>0$.
\end{definition}

\begin{remark}
If $D \subseteq G$ is $N$-thick in $Y \subseteq G$ (where $G$ is a group) and both $D$ and $Y$ are finite, then $|D| \geq \frac{|Y|}{N-1}$.
\end{remark}

\begin{proof}
From the definition of $N$-thickness it follows that $N>1$ and $Y$ is covered by $N-1$ left translates of $D$.
\end{proof}

The next remark is \cite[Remark 4.5]{Kru}.

\begin{remark}\label{remark: intersection of thick sets}
Every definable, additively symmetric subset of $\bar R$ which contains $(\RR,+)^{00}_M$ is thick. Thus, $(\RR,+)^{00}_M$ is the intersection of a downward directed family of $M$-definable $\bigvee$-thick subsets of $\RR$.
\end{remark}

We will prove now  Theorem \ref{theorem: thickness and sum-product}. The proof is an elaboration of the proof of the weaker \cite[Theorem 5.4]{Kru}, where it was assumed that there are no zero divisors. It is a rather standard ultraproduct argument but playing around with thick sets. Thickness in Theorem  \ref{theorem: thickness and sum-product} is of course with respect to the additive group $(R,+)$.

\begin{proof}[Proof of Theorem \ref{theorem: thickness and sum-product}]
Suppose for a contradiction that for every $n \in \omega$ there is a finite $K$-approximate subring $X_n$ for which there is no $n$-thick subset of $Y_n:=4X_n + X_n \cdot 4X_n$ consisting of zero divisors and $Y_n$ is not a subring. (Note that $|Y_n| \geq n$, so $|X_n|$ tends to $\infty$; note also that the sets $X_n$ and $Y_n$ have nothing to do with the sets $X_n$ recursively defined in the introduction.) Let $R_n:= \langle X_n \rangle$, all considered in the language of rings expanded by an additional relation symbol $P$ interpreted in $R_n$ as $P(R_n):=X_n$. Take a non-principal ultrafilter $\mathcal{U}$ on $\omega$, and let $M := \prod R_n/\mathcal{U}$ (the ultraproduct with respect to $\mathcal{U}$) and $X:= P(M)=\prod X_n/\mathcal{U}$. Then $M$ is a ring and $X$ is an infinite, definable $K$-approximate subring.  Let $R :=\langle X \rangle$. Pass to a monster model $\C \succ M$. 

By Fact \ref{fact: 00 exists}, the ideal $\bar R^{00}_M$ of $\bar R$ exists and is contained in $\bar Y:=4\bar X + \bar X \cdot 4\bar X$. Hence, by Remarks \ref{remark: intersection of thick sets} and \ref{remark: thick}, $\bar R^{00}_M =\bigcap \bar D_i$, where $\{D_i\}_{i \in I}$ is a downward directed family of $M$-definable thick subsets of $Y$. As $\bar R^{00}_M$ is a left ideal, by compactness (or rather $|M|^+$-saturation of $\C$), we conclude that there is $i \in I$ such that $\bar{D}_i (\bar Y \cdot \bar Y + \bar Y + \bar Y) \subseteq \bar Y$. Hence, $D_i (Y \cdot Y + Y + Y) \subseteq Y$ (where $Y:=4X + X \cdot 4X$). Denote $D_i$ by $D$. It is an $n$-thick subset of $Y$ for some $n$. We can write $D=D(M,(a_{1m})_{m<\omega}/\mathcal{U},\dots, (a_{km})_{m<\omega}/\mathcal{U})$, where $D(x,y_1,\dots,y_k)$ is a formula without parameters. Then there is $U \in \mathcal{U}$ such that for every $m \in U$, $D(R_m,a_{1m},\dots,a_{km})$ is an $n$-thick subset of $Y_m$ and $D(R_m,a_{1m},\dots,a_{km})(Y_m \cdot Y_m + Y_m +Y_m) \subseteq Y_m$. Since $Y_m$ is additively symmetric and not a subring, we see that $Y_m \subsetneq Y_m \cdot Y_m + Y_m +Y_m$, and so, by finiteness of $Y_m$, we conclude that $D(R_m,a_{1m},\dots,a_{km})$ consists of zero divisors. Taking $m \geq n$ in $U$ (it exists as $\mathcal{U}$ is non-principal), we get a contradiction with the choice of the $X_m$'s.

The fact that $K^{11}$ additive translates of $X$ cover $4X + X \cdot 4X$ is an easy computation. 
\end{proof}

Regarding Theorem \ref{theorem: NSOP}, first we recall the definition of NSOP (the non strict order property).

\begin{definition}\label{definition: NSOP}
Let $T$ be a consistent theory. We say that a formula $\varphi(x,y)$ has {\em SOP} if for every model $M$ of $T$ there is no sequence $(a_i)_{i<\omega}$ in $M$ with $\varphi(M,a_0) \subsetneq \varphi(M,a_1) \subsetneq \dots$. 
We say that $T$ has {\em SOP} if some formula has SOP. A structure $M$ has {\em SOP} if $\Th(M)$ has SOP. Finally, {\em NSOP} is defined as the negation of SOP. 
\end{definition}



\begin{proof}[Proof of Theorem \ref{theorem: NSOP}]
By Fact \ref{fact: 00 exists}, the ideal $\bar R^{00}_M$ of $\bar R$ exists and is contained in $\bar Y:=4\bar X + \bar X \cdot 4\bar X$. Hence, since by Remarks \ref{remark: intersection of thick sets} and \ref{remark: thick} the ideal ${\bar R}^{00}_M$ is the intersection of a downward directed family of $M$-definable thick subsets of $\bar Y$, there is an $M$-definable thick $D \subseteq Y:=4X + X\cdot 4X$ such that $\bar D (\bar Y \bar Y + \bar Y + \bar Y) \subseteq \bar Y$ and so $D (Y\cdot Y + Y + Y) \subseteq  Y$.

Suppose $Y$ is not a subring. Since $Y$ is additively symmetric, this implies that $Y \subsetneq Y \cdot Y + Y +Y$. We claim that $D$ consists of zero divisors. Otherwise,  there is $d \in D$ such that $x \mapsto dx$ is an injection from $R$ to $R$. Since $d(Y \cdot Y + Y +Y) \subseteq Y \subsetneq Y \cdot Y + Y +Y$, this leads to the following strictly decreasing chain of uniformly definable sets, contradicting NSOP:
$$d(YY +Y+Y) \supsetneq d^2(YY+Y+Y) \supsetneq d^3(YY+Y+Y) \supsetneq \dots.$$
In order to see that this sequence is indeed uniformly definable (i.e. of the form $(\varphi(M,a_n))_{n<\omega}$ for some formula $\varphi(x,y)$ and some parameters $a_n$), it is enough to show that there is $m\in \mathbb{N}$ such that $d^n \in X_m$ for all $n >0$ (because $\cdot$ is definable on any given $X_m$). Since $d \in Y$, we clearly have $d \in Y \subseteq X_3$, and for $n>1$ we have $d^n \in d^{n-1}(YY +Y+Y) \subseteq d(YY +Y+Y) \subseteq X_5$. 
\end{proof}

In order to prove Proposition \ref{proposition: infinite approximate subfields}, one has to slightly extend the context of Fact \ref{fact: 00 exists}. Namely, instead of assuming that $X$ is a definable approximate subring, it is enough to assume that $X$ is a definable in $M$ subset of a ring such that there is an increasing sequence $(X_n)_{n<\omega}$ of definable symmetric subsets of that ring such that $R:=\bigcup X_n$ is a subring, $X$ is contained in some $X_k$, for every $n \in \mathbb{N}$ there exists $m \in \mathbb{N}$ such that $X_n \cdot X_n \cup (X_n+X_n)\subseteq X_m$, finitely many additive translates of $X$ by elements of $R$ cover $X_n$, and $+$ and $\cdot$ restricted to $X_n$ are definable. (In other words, $R$ is a $\bigvee$-definable ring and $X$ is a definable generic subset of $R$.)  The proof of Fact \ref{fact: 00 exists} (i.e. Theorem 4.1 from \cite{Kru}) goes through word for word in this more general context. The only minor thing to be observed is that $(\bar R,+)^{00}_M$ exists, where $\bar R= \bigcup \bar X_n$. This follows from the fact that $\bar X_k$ is a definable approximate subring of $\bar R$ and so $(\langle \bar X_k \rangle,+)^{00}_M$ exists, as then clearly $(\bar R,+)^{00}_M=(\langle \bar X_k \rangle,+)^{00}_M$ exists.

\begin{proof}[Proof of Proposition  \ref{proposition: infinite approximate subfields}]
Take $R$ to be the subfield generated by $X$, and let $M$ be $R$ equipped with the full structure (i.e. all subsets of all finite Cartesian powers of $M$ are predicates). Presenting $R$ as $\bigcup_{n \in \mathbb{N}} \textrm{Alg}_{n+1}(X)$, we are in the context from the above paragraph.  So, using the aforementioned extension of Fact \ref{fact: 00 exists}, we have that the ideal $\bar R^{00}_M$ exists and equals $(\RR,+)^{00}_M +\bar R (\RR,+)^{00}_M = (\RR,+)^{00}_M +\bar X (\RR,+)^{00}_M \subseteq 4\bar X + \bar X \cdot 4\bar X$ (where by $\bar R$ we mean $\bigcup_{n >0} \textrm{Alg}_n(\bar X)$). On the other hand, $\bar R$ is a field, so does not have proper ideals, and hence $\bar R^{00}_M = \bar R$. Thus, $\bar R = 4\bar X + \bar X \cdot 4\bar X$, so $R = 4X + X \cdot 4X$.
\end{proof}

\begin{proof}[Proof of Corollary \ref{corollary: generics in fields}]
By Proposition \ref{proposition: infinite approximate subfields}, $R:=4X + X \cdot 4X$ is a subfield of $F$. If it was a proper subfield of $F$, then $(R,+)$ would be an infinite index subgroup of  $(F,+)$ which would contradict the assumption that $X$ is additively generic.
\end{proof}

\begin{proof}[Proof of Corollary \ref{corollary: classification of finite approximate subfields}]
If not, then a non-principal ultraproduct of a sequence forming a counter-example would contradict Proposition \ref{proposition: infinite approximate subfields}.
\end{proof}

Let us finish with  a question (that we did not really think about) on a potential extension of \cite[Lemma 5.2]{Bre} to infinite approximate subfields defined in the same way as the finite ones.

\begin{question}
Let $X$ be an approximate subfield. 
Is it true that $\textrm{Alg}_n(X)$ is covered by finitely many additive translates of $X$?
\end{question}

\appendix

\section{Pseudofinite approximate subrings}\label{appendix: pseudofiniteness}

There are several ways to formalize the idea of pseudofinitness. Below we describe a more or less self-contained way to formalize the notion of a pseudofinite approximate subring. We work with quotients (so imaginary sorts), as they are needed in Section \ref{section: structure of finite approximate rings}.


Consider the class $\mathcal{K}$ of all 2-sorted structures consisting of a countable ring $(G,+_G,\cdot_G)$ and the ordered field of reals $(\mathbb{R},+,\cdot,0,1, \leq, \mathbb{Z})$ (where the set of integers $\mathbb{Z}$ is considered as a predicate) together with a predicate $R \subseteq \mathbb{N} \times G \times \mathbb{N}$ defined as follows. Enumerate all functions $f \colon [n] \to G$ (where $n$ ranges over $\mathbb{N}$, and $[n]:=\{0,\dots,n-1\}$) as $f_0,f_1,\dots$, and declare that $R(a,b,c)$ holds in our structure if $c \in \mathbb{N}$ and $f_c(a)=b$. In particular, our language  is finite, and we will denote it by $\mathcal{L}$.

Let $T_0$ be the $\mathcal{L}$-theory which consists of all sentences which are true in all structures from the class $\mathcal{K}$. It is clearly consistent.

By a {\em pseudofinite approximate subring} we mean an approximate subring of the form $\varphi(G,a)/\psi(G,a)$ of the ring $H/\psi(G,a)$, where $M=(G,\mathbb{R}^*)$ is a model of $T_0$, $\varphi(x,a), \psi(x,a)$ are formulas with some parameters $a$ from $M$ such that $\psi(G,a)$ is a two-sided ideal of a definable subring $H$ of $G$, and there is $n \in \mathbb{N}^*$ (where $\mathbb{N}^*$ is the interpretation of $\mathbb{N}$ in $\mathbb{R}^*$) and definable bijection $f \colon [n] \to \varphi(G,a)/\psi(G,a)$ (where $[n]:=\{ i \in \mathbb{N}^*: i <n\}$).

The following remark trivially holds in every structure $M=(G,\mathbb{R})$ from the class $\mathcal{K}$.

\begin{remark}
For all formulas $\varphi(x,a),\psi(x,a)$, where $a$ is a tuple of parameters from $M$, such that $\psi(G,a)$ is an additive subgroup of $G$ and for all $n \in \mathbb{N}$, every function $f \colon [n] \to \varphi(G,a)/\psi(G,a)$ is induced by $R(x,y,c)$ for some $c \in \mathbb{N}$ in the sense that $\textrm{dom}(R(\cdot,\cdot,c))=[n]$ and $f(i) = y/\psi(G,a)$ for a unique $y \in G$ with $R(i,y,c)$.
\end{remark}

Therefore, the following holds in every model $M=(G,\mathbb{R}^*)$ of $T_0$.

\begin{remark}\label{remark: uniform definability}
For all formulas $\varphi(x,a),\psi(x,a)$, where $a$ is a tuple of parameters from $M$, such that $\psi(G,a)$ is an additive subgroup of $G$ and for all $n \in \mathbb{N}^*$, every definable function $f \colon [n] \to \varphi(G,a)/\psi(G,a)$ is induced by $R(x,y,c)$ for some $c \in \mathbb{N}^*$ in the sense that  $\textrm{dom}(R(\cdot,\cdot,c))=[n]$ and $f(i) = y/\psi(G,a)$ for a unique $y \in G$ with $R(i,y,c)$.
\end{remark}

The next remark also trivially holds in every structure $M=(G,\mathbb{R})$ from the class $\mathcal{K}$.

\begin{remark}
Let $\star \in \{+_G,\cdot_G\}$. For all formulas $\varphi(x,a),\psi(x,a)$, where $a$ is a tuple of parameters from $M$, such that $\psi(G,a)$ is a two-sided ideal of a definable subring of $G$, for every positive $n \in \mathbb{N}$ and every function $f \colon [n] \to \varphi(G,a)/\psi(G,a)$, there is $d \in \mathbb{N}$ with $\textrm{dom}(R(\cdot,\cdot,d)) =[n]$ such that $R(0,y,d) \rightarrow y/\psi(G,a)=f(0)$ and for every $\alpha+1 <n$ we have 
$$R(\alpha+1,y,d) \rightarrow (\exists y')(R(\alpha,y',d) \wedge y/\psi(G,a) = y'/\psi(G,a) \star f(\alpha)).$$
\end{remark}

Therefore, the following holds in every model $M=(G,\mathbb{R}^*)$ of $T_0$.

\begin{remark}\label{remark: initial products}
Let $\star \in \{+_G,\cdot_G\}$. For all formulas $\varphi(x,a),\psi(x,a)$, where $a$ is a tuple of parameters from $M$, such that $\psi(G,a)$ is a two-sided ideal of a definable subring $H$ of $G$, for every positive $n \in \mathbb{N}^*$ and every definable function $f \colon [n] \to \varphi(G,a)/\psi(G,a)$, there is $d \in \mathbb{N}$ with $\textrm{dom}(R(\cdot,\cdot,d)) =[n]$ such that $R(0,y,d) \rightarrow y/\psi(G,a)=f(0)$ and for every $\alpha+1 <n$ we have 
$$R(\alpha+1,y,d) \rightarrow (\exists y')(R(\alpha,y',d) \wedge y/\psi(G,a) = y'/\psi(G,a) \star f(\alpha)).$$
\end{remark}

The last remark tells us that whenever we have a definable pseudofinite sequence (i.e., a definable function $f \colon [n] \to \varphi(G,a)/\psi(G,a)$ for some $n \in \mathbb{N}^*$) of elements of  $\varphi(G,a)/\psi(G,a)$, then we have a well-defined definable pseudofinite sequence of all initial pseudofinite  sums [resp. products]. In particular, for any element $h \in H/\psi(G,a)$ and $n \in \mathbb{N}^*$ we have a well-defined element $n h$ (the $n$-fold sum of $h$). By a {\em definable pseudofinite sum [product]} we will mean the pseudofinite sum [product] of a definable pseudofinite sequence.

By Remarks \ref{remark: uniform definability} and \ref{remark: initial products}, we get 

\begin{corollary}\label{corollary: <X>*}
If $X$ is a pseudofinite approximate subring (of some $H/\psi(G,a)$), then there exists a definable set $\langle X \rangle^*$ consisting of all definable pseudofinite sums of all definable pseudofinite products of elements of $X$. Moreover, this is the smallest definable subring of $H/\psi(G,a)$ containing $X$.
\end{corollary}

\begin{proof}
By Remarks \ref{remark: uniform definability} and \ref{remark: initial products}, we easily get that the set of all definable pseudofinite products of elements of $X$ is definable. Applying again Remarks \ref{remark: uniform definability} and \ref{remark: initial products} to the obtained definable set, the conclusion follows. For the moreover part
note that working with explicit formulas, we obtain a formula $\theta(x,y)$ depending only on $\varphi(x,y)$ and $\psi(x,y)$ (and not on $M$, $H$, $a$) such that $\theta(x,a)$ defines $\langle X \rangle^*$ in $M$ in the sense that $\theta(G,a)/\psi(G,a) = \langle X \rangle^*$. Now, the fact that it is the smallest definable subring containing $X$ follows from the observation that this is true in every structure from $\mathcal{K}$.
\end{proof}

These kinds of properties are used throughout the paper. We will not justify them, as they follow easily from the above remarks. 

Note also that whenever we have a pseudofinite approximate subring $X$ of $H/\psi(G,a)$, $R:=\langle X \rangle^*$, and $I \lhd R$ is definable, then $X/I$ is naturally a pseudofinite approximate subring of the quotient $R/I$. This is because $R/I$ can be naturally identified with $H'/J$ for some definable ideal $J \supseteq \psi(G,a)$ of a definable subring $H'$ of $H$.


\section{Unitality and finite dimensional algebras}\label{section: unitality}
Standard references on associative algebras usually assume unitality of the algebras, whereas in this paper we do not have this assumption. Herstein's book \cite{Herstein} presents the theory of non-commutative rings which are not necessarily unital, but it does not fully develop the parallel theory for non-unital algebras. The goal of the short discussion below is to show that semi-simple, finite-dimensional algebras are always unital, and so the whole standard knowledge is at our disposal.

Let $A$ be a (not necessarily unital) $K$-algebra, where $K$ is a field. By ideals we mean ideals in the sense of algebras (not just rings). By the {\em Jacobson radical}, denoted by $J(A)$, we mean the subset of $A$ defined as the intersection of all regular maximal left ideals of $A$ if at least one regular left ideal exists and $J(A):=A$ otherwise. ({\em Regularity} of a left ideal $I$ means that there is $a \in A$ such that for every $x \in A$ we have $x-xa \in I$; if $A$ is unital, then every left ideal is regular, so this property may be removed from the definition.) We say that $A$ is {\em semi-simple} if $J(A) =\{0\}$. It trivially follows from the definitions that $A/J(A)$ is a semi-simple algebra provided that $J(A)$ is a two-sided ideal (which for instance is the case for unital algebras by Proposition 4.1 in \cite{PierceAssociativeAlgebras}).

The next result is contained in Theorem 3.5 of \cite{PierceAssociativeAlgebras}.

\begin{theorem}[Wedderburn's Theorem]\label{theorem: Wedderburn} A semi-simple, unital, finite-dimensional (or Artinian) $K$-algebra $A$ is isomorphic with a product of finitely many algebras of matrices with coefficients from some division algebras over $K$. Conversely, every product of finitely many such matrix algebras is a finite-dimensional, unital, semi-simple algebra.
\end{theorem}

For the next result see Prop. 4.4 \cite{PierceAssociativeAlgebras}

\begin{fact}\label{fact: nilpotent Jacobson for algebras}
The Jacobson radical of an Artinian unital algebra is nilpotent. In particular, any finite-dimensional, unital $K$-algebra has nilpotent Jacobson radical.
\end{fact}

\begin{lemma}\label{lemma: J(A)=J(A')}
Let $A$ be a finite-dimensional (not necessarily unital) $K$-algebra and let $A':=K \oplus A$ be the unitization of $A$.  Then $J(A) = J(A')$.
\end{lemma}

\begin{proof}
Note that $A$ is a maximal left ideal of $A'$, so $J(A') \subseteq A$.


First, we show that $J(A') \subseteq J(A)$. The case when $J(A)=A$ is clear by the above observation. In the remaining case, let $I$ be a regular maximal left ideal of $A$, in particular there is $a \in A$ such that $x-xa \in I$ for all $x \in A$. Note that $I$ is a left ideal of $A'$. Consider any $r \in A \setminus I$. Let $J:=(1-a) + I$, where $(1-a)$ is the left ideal of $A'$ generated by $1-a$. We claim that $r \notin J$. Otherwise, $r=(k+s)(1-a) +i= k -ka +s -sa +i$ for some $k \in K$, $s \in A$, $i \in I$. Since $r,a,s,i \in A$, we get that $k=0$, but then the fact that $s-sa, i \in I$ yields $r \in I$, a contradiction. Thus, we can extend $J$ to a maximal left ideal $J'$ of $A'$. Since $1-a \in J$, we get that $a \notin J'$. So $J' \cap A$ is a proper left ideal of $A$ extending $I$, and hence $J' \cap A=I$ by maximality of $I$. So $J(A') \subseteq I$. This shows that  $J(A') \subseteq J(A)$.

On the other hand, by Theorem \ref{theorem: Wedderburn}, $A'/J(A')$ is a product of finitely many matrix algebras over some division rings, so the two-sided ideal $A/J(A')$ of $A'/J(A')$ is a subproduct of some of these matrix algebras (because all these matrix algebras are simple), and hence it is a semi-simple algebra. Thus, $J(A) \subseteq J(A')$.
\end{proof}

\begin{corollary}\label{corollary: semisimplicity in unitization}
Let $A$ be a finite-dimensional $K$-algebra.
\begin{enumerate}
\item $J(A)$ is a nilpotent two-sided ideal.
\item $A$ is semi-simple iff its unitization $A'$ is semi-simple.
\end{enumerate}
\end{corollary}

\begin{corollary} [Wedderburn's theorem and unitality]\label{Cor: Unitality}
A semi-simple, finite-dimensional $K$-algebra $A$ is isomorphic with a product of finitely many algebras of matrices with coefficients from some $K$-division algebras. In consequence, a semi-simple, finite-dimensional $K$-algebra $A$ is always unital.
\end{corollary}

\begin{proof}
Let $A'$ be the unitization of $A$. By Corollary \ref{corollary: semisimplicity in unitization}(2), $A'$ is also semi-simple. Hence, by Theorem \ref{theorem: Wedderburn}, it is isomorphic to the product of matrix algebras over some division rings. Hence, as in the proof of Lemma \ref{lemma: J(A)=J(A')}, the two-sided ideal $A$ of $A'$ is a subproduct of some of these matrix algebras. In particular, $A$ is of the desired form and so unital.
\end{proof}

\nocite{*}
\printbibliography
\end{document}